\documentclass{amsart}
\usepackage{hyperref}
\usepackage{mathabx}
\usepackage{amsmath, amsthm, amssymb}
\usepackage{mdwlist} 
\usepackage{graphicx}
\usepackage{ytableau}
\usepackage[utf8]{inputenc}
\usepackage[linesnumbered,boxruled,algo2e]{algorithm2e}

\usepackage[author-year]{amsrefs}
\usepackage{enumerate}
\usepackage[nomain,symbols]{glossaries}

\usepackage[mathscr]{euscript}
\usepackage{mathrsfs}

\usepackage{pictexwd,dcpic}

\newtheorem*{main*}{Main Theorem}
\newtheorem{theorem}{Theorem}[section]

\newtheorem{corollary}[theorem]{Corollary}
\newtheorem{proposition}[theorem]{Proposition}
\newtheorem{lemma}[theorem]{Lemma}

\theoremstyle{definition}
\newtheorem{definition}[theorem]{Definition}

\newtheorem{hypothesis}{Hypothesis}

\theoremstyle{remark}
\newtheorem{remark}[theorem]{Remark}
\newtheorem{genremark}[theorem]{General remark}
\newtheorem{example}[theorem]{Example}

\definecolor{absolutezero}{rgb}{0.0,0.28,0.73}

\newcommand{\changed}[1]{#1}

\newcommand{\verybigstrut}{\rule{0pt}{4ex}}

\renewcommand{\Re}{\mathrm{Re}}
\renewcommand{\Im}{\mathrm{Im}}

\newcommand{\toricI}[1]{I-#1}
\newcommand{\toricII}[1]{II-#1}

\newcommand{\renlength}[1]{\mathcal L_{#1}}
\newcommand{\condlength}{\mathscr L}

\DeclareMathOperator*{\MP}{\pmb{\mathbb P}}

\newcommand{\finsler}[2]{\left\vvvert #1 \right\vvvert_{#2}}

\newcommand{\conv}[1]{\mathrm{Conv}(#1)}

\newcommand{\Log}[1]{\mathrm{Log}(#1)}

\newcommand{\vol}{\mathrm{Vol}}
\newcommand{\rank}{\mathrm{Rank}}
\newcommand{\dd}{\,\mathrm{d}}

\newcommand{\binomial}[2]{\ensuremath{\left( \begin{matrix}#1 \\ #2 \end{matrix} \right)}}

\newcommand{\defun}[5]{\ensuremath{\begin{array}{lrcl}
#1:&#2 & \longrightarrow & #3\\&#4 & \longmapsto & #5\end{array}}}

\newcommand{\defeq}{\stackrel{\mathrm{def}}{=}}
\newcommand{\diag}[1]{\mathrm{diag}\left(#1\right)}

\newcommand{\bigmatrix}[2]{\ensuremath{\begin{pmatrix}
\begin{matrix}
-1 & 0 & \dots & 0\\
\vdots & \vdots && \vdots \\
-1 & 0 & \dots & 0\\
\end{matrix}       & \hspace{3em} & \Huge{#1_1}& \hspace{3em}\\
\begin{matrix}
0 & -1 & \dots & 0\\
\vdots & \vdots && \vdots \\
0 & -1 & \dots & 0\\
\end{matrix}       & \hspace{3em}& \Huge{#1_2}& \hspace{3em}\\
\ddots                              & &\vdots& \\
\begin{matrix}
0 & 0 & \dots & -1\\
\vdots & \vdots && \vdots \\
0 & 0 & \dots & -1\\
\end{matrix}       & \hspace{3em}& \Huge{#1_{#2}}& \hspace{3em}\\
\end{pmatrix}
}}

\author{Gregorio Malajovich}
\title
{Complexity of sparse polynomial solving 3: Infinity}
\address{Departamento de Matemática Aplicada, Instituto de Matemática, Universidade Federal do
Rio de Janeiro. Caixa Postal 68530, Rio de Janeiro, RJ 21941-909, Brasil.}
\email{gregorio@im.ufrj.br}
\date{\today}
\thanks{This research was partially funded by the {\em Coordenação de Aperfeiçoamento de Pessoal de Nível Superior} (CAPES) of Brazil, PROEX grant}
\subjclass[2010]{
Primary 65H10. 
Secondary 65H20,
14M25,
14Q20.
}
\keywords{Sparse polynomials, \changed{mixed volume, mixed area}, Newton iteration, 
homotopy algorithms, 
toric varieties, toric infinity, 
renormalization, condition length}

\begin{document}
\begin{abstract}
A theory of numerical path-following in toric varieties was suggested in two previous papers. The motivation
is solving systems of polynomials with real or complex coefficients. When those polynomials are not assumed `dense',
solving them over projective space or complex space may introduce spurious, degenerate roots or components. 
Spurious roots may be avoided by solving over toric varieties.

In this paper, a homotopy algorithm is locally defined on charts of the toric variety. 
Its complexity is bounded linearly
by the condition length, that is the integral along the lifted path (coefficients and solution) of the
toric condition number. Those charts allow for stable computations near "toric infinity",
which was not possible within the technology of the previous papers.
\end{abstract}
\maketitle

\section{Introduction}
This is the third of a series of papers on systems of polynomial equations. The objective is designing efficient algorithms for solving polynomial systems, and proving their efficiency by rigorous complexity bounds. For simplicity, systems are 
assumed with as many equations as variables:
\begin{equation}\label{system}
	\sum_{i \mathbf a \in A_i} f_{i,\mathbf a} Z_1^{a_1} Z_2^{a_2} \dots Z_n^{a_n} = 0,
\hspace{2em} i=1, \dots, n.
\end{equation}
The {\em supports} $A_1, \dots A_n$ must be finite subsets of $\mathbb N_0^m$.  More generally, {\em Laurent polynomials} 
admit finite subsets of $\mathbb Z^m$ as supports. The coefficients $f_{i\mathbf a}$ are assumed to be complex numbers. 
Solutions are expected to belong to $\mathbb C_{\times}^n$ or a compactification thereof. Most popular choices of
the compactification are $\mathbb C^n$ or projective space $\mathbb P^n$. 
Since exact solutions cannot possibly be computed exactly in general, they will be represented by a numerically 
certified approximation. 

The algorithm of choice for complex polynomial systems is {\em path-following} or {\em homotopy}. 
A major theoretical achievement in this subject was the positive solution of Smale's $17^{\mathrm{th}}$ problem. 
This problem was restricted to random systems of {\em dense} polynomials of fixed degree $d_1, \dots d_n$. 

A randomized polynomial time algorithm was found by Beltrán and Pardo \ycites{Beltran-Pardo-2009,Beltran-Pardo-2011} 
and then a `deterministic' algorithm was providided by \ocite{Lairez}. Both of those algorithms and 
most of the recent literature on the subject assume the output space is $\mathbb P^n$, and that input space
is endowed by the induced unitary invariant probability distribution. 

\medskip
\par

A general polynomial system of the form \eqref{system} can be homogenized to a dense polynomial system by 
increasing the supports. The homogenized polynomial does not qualify as
`random' or generic in the space of homogeneous polynomial systems, as it is stuffed with zero
coefficients. In that sense, results on Smale's 
$17^{\mathrm{th}}$ problem do not apply to random or non-random non-dense polynomial systems.

Suppose for instance that one obtains a root for a random dense system and then deform the pair 
(system, solution) to solve a non-dense target system, random or given. According to  
\ocite{Bernstein}*{Theorem A}, the number of isolated roots of the sparse system
is bounded in terms of the mixed volume of the tuple of supports. This bound is possibly exponentially smaller than the
Bézout bound. In this case, the solution path converges
with overwhelming probability to a point  at projective infinity.
\ocite{Bernstein}*{Theorem B} implies that the limit point is a degenerate solution of the homogenized system, possibly a positive dimensional solution.

A system is said to be {\em sparse} if the mixed volume bound is strictly smaller than the Bézout bound.
From the discussion above, the homogenization of a sparse system will be always degenerate and the
toric compactification $\mathscr V_{\mathbf A}$ is to be preferred.

\begin{example} The eigenvalue problem $M\mathbf u=\lambda \mathbf u$ with $u_1=1$
	has generically $n$ different eigenpairs, while its homogenization
	admits an $n-2$-dimensional 
	component $[t=0, \lambda = 0$, $M_{12}u_2+\dots+M_{1n}u_n = 0]$ at projective infinity,
	where $t$ is the homogenizing variable.
\end{example}
\subsection{What this paper is about}
\cite{toric1} proposed
a theory of homotopy continuation over toric varieties.
The cost
of homotopy continuation was bounded in terms of an integral or {\em condition length} of the
(system, solution) path. A certain unbounded term $\nu(z)$ in the  
integral prevented the construction of a global algorithm with expected finite complexity.
\par
In the second volume, \ocite{toric2} exploited the toric (multiplicative) action on the solution variety
to reduce most of path-following to the fiber above $\mathbf Z=\mathbf 1=(1, \dots, 1)^T$.
The cost of path-following was bounded in terms of the {\em renormalized condition length},
where the term $\nu(Z)$ was replaced by a constant. 
For each choice of $(A_1, \dots, A_n)$, the expected cost of renormalized homotopy is polynomial
on the renormalized condition of the limiting systems. 
Unfortunately,
the renormalized condition number can be unboundedly worse than the natural condition number.
(See Theorem \ref{cost-of-renorm-legacy} below).
The main result in the present paper is:
\begin{main*}
Fix finite supports $A_1, \dots, A_n \subset \mathbb N_0$ so that
	the corresponding toric variety $\mathscr V_{\mathbf A}$
	is $n$-dimensional and non-singular. Then there is a finite open cover
	$\mathscr U$ of $\mathscr V_{\mathbf A}$, and for each $U \in \mathscr U$ there is 
	an algorithm with the following properties. 
	\begin{enumerate}[(a)]
\item
			The {\bf input} $(\mathbf g_t, \mathbf v_0)$
	is a smooth path $(\mathbf g_t)_{t\in[0,T]}$, $T \in \mathbb R \cup \{\infty\}$ in the space 
	$\mathbb P(\mathbb C^{A_1}) \times \dots \times \mathbb P(\mathbb C^{A_n})$ 
	of sparse polynomial 
	systems, and a point $\mathbf v_0 \in \mathscr V_{\mathbf A}$.
\item
	The {\bf output} is a point
	$\mathbf v_{T} \in \mathscr V_{\mathbf A}$. 
\item	{\bf Correctness condition:} 
	Assume the input $\mathbf v_0$ is an approximate root associated to an exact root $\tilde {\mathbf v}_0$ of $\mathbf g_0$,
	and denote its continuation by 
			$\mathbf v_t$, so that  $\mathbf g_t(\tilde{\mathbf v}_t) \equiv 0$. 
			If the algorithm terminates, $\mathbf v_{T}$ is an approximate root associated to the exact root $\tilde {\mathbf v}_{T}]$.
\item	{\bf Cost of the algorithm:} 
	If $\tilde {\mathbf v}_t \subset U$ for $0 \le t \le T$, then the algorithm terminates after at most
	$C_U(1+\condlength)$ homotopy steps,
	where 
			\[
\condlength=	\condlength(\mathbf g_t, \mathbf v_t; 0,T)=
	\int_{0}^{T}
\left(\left\| \frac{\partial}{\partial t} \mathbf g_t \right\|_{\mathbf g_t}
	+\left\| \frac{\partial}{\partial t} \tilde{\mathbf v}_t
\right\|_{\mathbf v_t} \right) \mu(\mathbf g_t, \mathbf v_t)
\dd t 
	\]
	is the true condition length, $\mu$ is the toric condition number and $C_U$ is a constant.
	\end{enumerate}
\end{main*}

\begin{remark}
The condition on $\mathscr V_{\mathbf A}$ corresponds to certain condition on the supports,
	see Hypothesis~\ref{NDH} below. If $\mathscr V_{\mathbf A}$ is singular, the 
	conclusion of the Theorem still holds as long as the
path $v_t$ stays clear from a singularity, see Hypothesis~\ref{NDIH} instead.
\end{remark}

\begin{remark}
	The term {\em approximate root} means an approximate root in the sense Smale, see
	Definition 4.1.1 in the second paper of this series or
	Corollary~\ref{cor-quadratic} below. Technicality: this Corollary applies to the local charts 
	in Theorem~\ref{th-coords} after reduction to normal form (Definition~\ref{normal-form}).
\end{remark}

\begin{remark}\label{rem-global}
We may also produce a global algorithm exploiting the compacity of $\mathscr V_{\mathbf A}$.
The algorithm will swap charts from an open set $U \in \mathscr U$ to the next $U' \in \mathscr U$
by first improving the precision of $v_T$ at the
cost of a few Newton iterations, then solving one or two linear programming problems, and finally
adjusting an extra parameter. The overall cost will depend on the number of transitions in
$\mathscr U$ and this may depend on the regularity of the path $\mathbf g_t$. 
\end{remark}

\begin{remark}
One of the sets $U \in \mathscr U$ is the set of points of $\mathscr V_{\mathbf A}$
at a certain distance from `toric infinity'.
The associated algorithm is the renormalized homotopy from the previous paper.
\end{remark}

\begin{remark}
A typical application would be to follow a path of the form $\mathbf g + t \mathbf f$ for
$t \in [0, \infty]$ starting from a random $\mathbf g$ and terminating at a given
non-random $\mathbf f$. Since generically $(g_t)_{0 \le t < \infty}$ does not
cross the discriminant variety, we expect renormalized homotopy to succeed with high
probability except possibly while approaching a root at infinity. In that case one or 
more cover transitions may be necessary.
\end{remark}

\begin{remark}
This work was inspired by the ideas of \ocites{DTWY}.
	My second paper \cite{toric2} lacked a proper way of representing roots close to
`toric infinity'. 
\ocites{DTWY} suggested using Cox coordinates for homotopy on the toric variety. As there
is one new variable for each ray in the fan associated to the tuple $(A_1, \dots, A_n)$, 
this procedure can brutally increase the dimension. The approach here is more focused.
In order to produce a viable algorithm, we will associate a system of $n$ coordinates to
each cone in the fan of $(A_1, \dots, A_n)$. I call this system Carathéodory coordinates because
variables are selected by means of Carathéodory's Theorem (Th.\ref{caratheodory} below). The open covering
$U \in \mathscr U$ contains domains for all the Carathéodory coordinate systems.
\end{remark}
\begin{genremark}
	To avoid a profusion of self-citations, references to the two
	previous papers in this series will be indicated by roman numerals I or II prepended to the 
	quoted section, theorem or equation number. 
\end{genremark}

\subsection{General notations}

We denote by $\mathbb C_{\times}$ the multiplicative ring $(\mathbb C \setminus \{0\}, \times)$ of $\mathbb C$. Vectors in $\mathbb C^n$ or $\mathbb C_{\times}^n$ will be denoted by boldface $\mathbf Z$, as well as tuples $\mathbf a = [a_1, \dots, a_n] \in \mathbb Z^n$. Monomials are written  $\mathbf Z^{\mathbf a} = Z_1^{a_1} Z_2^{a_2} \dots Z_n^{a_n}$. Given a finite subset $A \in \mathbb Z^n$, the space of Laurent polynomials with support contained in $A$ is
\[
	\mathscr P_A = \left \{ \sum_{\mathbf a \in A} f_{\mathbf a} \mathbf Z^{\mathbf a},\ f_{\mathbf a} \in \mathbb C \right\}.
\]
The space $\mathscr P_A \simeq \mathbb C^A$ is endowed with the canonical Hermitian metric of $\mathbb C^A$. 
The {\em evaluation map} associated to $A$ is
\[
	\defun{V_{A}}{\mathbb C_\times^n}{\mathbb C^{A}}{\mathbf Z}{
		\begin{pmatrix}\vdots \\ \mathbf Z^{\mathbf a} \\ \vdots \end{pmatrix}_{\mathbf a \in A}}
\]
and satisfies
\[
	f(\mathbf Z) = \begin{pmatrix} \dots & f_{\mathbf a} & \dots \end{pmatrix}_{\mathbf a \in A}
	V_A(\mathbf Z).
\]
As a general convention, the vector of coefficients of $f$ belongs to the dual vector space to $\mathbb C^{A}$ and shall be written as a row vector.

The canonical projection onto projective space is denoted by $[\cdot]: \mathscr P_A \setminus \{0\} \rightarrow \mathbb P(\mathscr P_A)$. This projection is well-defined for all
$V_A(\mathbf Z)$, $\mathbf Z \in \mathbb C_{\times}^n$. The composition
\[
	\defun{[V_{A}]}{\mathbb C_\times^n}{\mathbb P(\mathbb C^{A})}{\mathbf Z}{[V_A(\mathbf Z)]}
\]
is known as the sparse {\em Veronese} embedding, or sometimes the {\em Kodaira} embedding.

\medskip 
\par

Through this paper, $\mathbf A = (A_1, \dots, A_n)$ are subsets of $\mathbb Z^n$ 
satisfying:
\begin{hypothesis}[non-coplanarity]\label{NDH}
	The mixed volume of $(\conv{A_1}, \dots, \conv{A_n})$ is not zero.
\end{hypothesis}
Equivalently, for positive indeterminates $t_1, t_2, \dots, t_n$, 
the polynomial
\[
	\vol(t_1 \conv{A_1} + \dots + t_n \conv{A_n})
\] has a non-zero coefficient
in $t_1t_2 \dots t_n$. For instance, the hypothesis above fails if the $A_i$
are contained in translates of the same hyperplane. It can also fail if $A_1$ and
$A_2$ are contained in translates of the same line. But the hypothesis is trivially
true if all of the $\conv{A_i}$ have non-zero $n$-dimensional volume.

The meaning of the Hypothesis can be grasped through \ocite{Bernstein}*{Theorem A}. It 
states that the number of isolated roots in $\mathbb C^n_{\times}$ of a generic system 
of Laurent polynomials with support $\mathbf A$ is precisely 
$n!\, V(\conv{A_1}, \dots, \conv{A_n})$. A vanishing mixed volume means that generic 
systems of Laurent polynomials with support $\mathbf A$ have no isolated roots.
\medskip
\par

The mixed {\em Veronese} or {\em Kodaira} embedding of the tuple $\mathbf A$ is 
\[
	\defun{\mathbf V_{\mathbf A}}{\mathbb C_{\times}^n}
{\mathbb P(\mathbb C^{A_1})
\times \dots \times \mathbb P(\mathbb C^{A_n})}
{\mathbf Z}
{\mathbf V_{\mathbf A}(\mathbf Z) \defeq ([V_{A_1}(\mathbf Z)], \dots, [V_{A_n}(\mathbf Z)]).}
\]

The {\em toric variety} $\mathscr V_{\mathbf A}$ associated to the
tuple $\mathbf A=(A_1, \dots, A_n)$ is the Zariski closure of 
the image of $V_{\mathbf A}$ in $\mathbb P(\mathbb C^{A_1})
\times \dots \times \mathbb P(\mathbb C^{A_n})$. Each projective space
$\mathbb P(\mathbb C^{A_i})$ is endowed with the Fubini-Study metric.
The toric variety $\mathscr V_{\mathbf A}$ is therefore naturally
endowed with the restriction
of the product Fubini-Study metric.

Let $\mathscr P_{\mathbf A} = \mathscr P_{A_1} \times \dots \times \mathscr P_{A_n}$ be the space of tuples $\mathbf f = (f_1, \dots, f_n)$.
Also, let $\MP(\mathscr P_{\mathbf A}) = \mathbb P(\mathscr P_{A_1}) \times \dots \times \mathbb P(\mathscr P_{A_n})$ where the product of Fubini-Study metrics is assumed.
A {\em zero} of $\mathbf f \in \mathscr P_{\mathbf A}$ in the toric variety $\mathscr V_{\mathbf A}$ is some $\mathbf v = ([\mathbf v_1], \dots, [\mathbf v_n]) \in \mathscr V_{\mathbf A}$ with $f_1 v_1= \dots  = f_n v_n = 0$. If $\mathbf v = V_{\mathbf A}(\mathbf Z)$ for some $\mathbf Z \in \mathbb C_{\times}^n$,
the zero $\mathbf v$ is said to be 
`finite'. Zeros in the toric variety that are not finite are said to be at toric infinity. 

\medskip
\par
The solution variety $\mathscr S_{\mathbf A} \subset \MP(\mathscr P_{\mathbf A}) \times \mathscr V_{\mathbf A}$ is the set of pairs $([\mathbf f], \mathbf v)$ so that $\mathbf v$ is a zero of $\mathbf f$. Let $\pi_1: \mathscr S_{\mathbf A}
\rightarrow \MP(\mathscr P_{\mathbf A})$ and $\pi_2: \mathscr S_{\mathbf A} \rightarrow
\mathscr V_A$ be the canonical projections. Because of Bernstein's theorem, $\pi_1$ is surjective, and the number of preimages for a generic $\mathbf f \in \mathscr P_A$ is $n!\, V(\conv{A_1}, \dots, \conv{A_n})$. The condition number 
$\mu(\mathbf f, \mathbf v)$ is defined on regular points of $\pi_1$ 
as the norm of the derivative of the implicit function $\mathbf v([\mathbf f])$, 
as a map $T_{[\mathbf f]}\MP(\mathscr P_{\mathbf A})
\rightarrow T_{\mathbf v} \mathscr V_{\mathbf A}$. 
By convention, the condition number is infinite at singular points of $\pi_1$. An explicit expression of the condition number at a pair $(\mathbf f, \mathbf v) \in \mathscr S_{\mathbf A}$ with  
$\mathbf v=\mathbf V_{\mathbf A}(\mathbf Z)$, $\mathbf Z \in \mathbb C_{\times}^n$ finite
is given by
\begin{equation}\label{defmu}
\mu(\mathbf f, \mathbf v)
=
\left\| 
\left(
\diag{\frac{1}{\| f_i\| \|V_{A_i}(\mathbf Z)\|}}
\begin{pmatrix}
\vdots \\
	f_i \, \diag{V_{A_i}(Z)} A_i\\
\vdots
\end{pmatrix}
\diag{\mathbf Z}^{-1}
\right)^{-1}
	\right\|_{\mathbf v}
\end{equation}
	where $\|\cdot\|_{\mathbf v}$ stands for 
	the operator norm $\mathbb C^n \rightarrow T_{\bf v}\mathscr V_{\mathbf A}$.
If $\mathbf v$
is a point at toric infinity, 
\[
	\mu(\mathbf f, \mathbf v) = \lim_{\mathbf w \rightarrow \mathbf v}
\mu(\mathbf f, \mathbf w)
.
\]
Toric varieties are not smooth in general. 
If $\mathbf v$ is a singular point of the toric
variety $\mathscr V_{\mathbf A}$, then 
$\mu(\mathbf f, \mathbf v) = \infty$ regardless of $\mathbf f$.
Otherwise, $\mu(\mathbf f, \mathbf v)$ is finite on a Zariski open subset
of the systems $\mathbf f \in \mathscr P_{\mathbf A}$ vanishing at $\mathbf v$.

\subsection{Outline of the paper}
The open cover $\mathscr U$ is constructed in Section \ref{sec:caratheodory}, leading 
in Section \ref{sec:normal} to a {\em normal form} for systems in $\mathscr P_{\mathbf A}$ 
as sums of monomials and exponentials. This normal form has many desirable properties,
and most invariants such a s condition are preserved by reduction to normal form.
Section \ref{sec:renormalization} introduces a partial renormalization that generalizes the
renormalization technique from the previous paper. The condition of renormalized and partially
renormalized systems is bounded in term of their natural condition on the toric variety
$\mathscr V_{\mathbf A}$. 

Section~\ref{sec:higher} contains a higher derivative estimate. This allows for the application
of Smale's alpha theory near toric infinity, through partial renormalized coordinates. The cost of
partially renormalized homotopy is estimated in Section~\ref{sec:homotopy}, and the overall cost
of the algorithm is wrapped up in Section~\ref{sec:global}.

\section{Carathéodory coordinates}
\label{sec:caratheodory}

In this section we construct the open cover $\mathscr U$ of the toric variety $\mathscr V_{\mathbf A}$. One of
the sets $U$ is contained in the domain of the chart logarithmic coordinates $\mathbf v(\mathbf z)$ from
the previous papers.  Because the exponential has an essential singularity at infinity,
this chart is not suitable for points close to toric infinity.

Points at toric infinity can be classified through the {\em fan} of $\mathscr V_{\mathbf A}$, see below. For each
element of the fan, we will define a certain number of charts so that locally, $\mathscr V_{\mathbf A}$
is parameterized by a mix of usual and logarithmic coordinates.
There will be enough usual coordinates to avoid introducing artificial singularities 
at toric infinity. 
Nevertheless, toric varieties are not necessarily smooth. Near a singularity,
those systems of coordinates may be singular or ramified.

\subsection{Supports, the lattice and the fan}\label{sub:supports}
The tuple of supports $(A_1, \dots, A_n)$ defines a finite collection of cones
known as the {\em fan} of the toric variety $\mathscr V_{\mathbf A}$, see
Definition~\toricII{4.3.3.} Recall the basic construction and notations:

Given $\boldsymbol \xi \in \mathbb R^n$, $A_i^{\boldsymbol \xi}$ is the
subset of all $\mathbf a \in A_i$ with $\mathbf a \boldsymbol \xi$ maximal.
Equivalently, $\conv{A_i^{\boldsymbol \xi}}$ is the closed facet of
$\conv{A_i}$ furthest in the $\boldsymbol \xi$ direction.

Given non-empty subsets $B_1 \subset A_1$, \dots, $B_n \subset A_n$, define the outer cone
$C(B_1, \dots, B_n)$ as the set of all $\boldsymbol \xi \in \mathbb R^n$
so that $A_i^{\boldsymbol \xi} = B_i$. This is an open cone in $\mathbb R^n$.
Let $\bar C(B_1, \dots, B_n)$ be its topological closure.

The (outer) {\em fan} associated to $\mathbf A$ is the set of all closed cones of the form 
$\bar C(B_1, \dots, B_n)$. A {\em ray} is a one-dimensional cone 
in the outer fan.

Let  $\Lambda$ be the $\mathbb Z$-module generated by all the $\mathbf a - \mathbf a' \in A_i - A_i$, $1 \le i \le n$. Because of Hypothesis~\ref{NDH}, this module is $n$-dimensional and hence a lattice.
Its dual $\Lambda^*$ is the set of $\mathbf x \in \mathbb R^n$ so that
for all $\boldsymbol \lambda \in \Lambda$, $\boldsymbol \lambda \mathbf x =
\sum \lambda_j x_j \in \mathbb Z$. It is also a lattice.

Cones in the fan of $\mathbf A$ 
are always polyhedral cones, that is positive linear combinations
of rays (one-dimensional cones). Moreover, those rays are spanned by 
some minimal $\boldsymbol \xi \in \Lambda^*$. Now we can state:

\begin{theorem}\label{th-coords}
	The toric variety $\mathscr V_{\mathbf A}$ admits a finite open
	cover
	by possibly singular,
	possibly many-to-one polynomial-exponential
	parameterizations of the form
\[
	\defun{\Omega}{\mathcal D_{\Omega} \subset \mathbb C^l \times \mathbb C^{n-l}}{\mathcal V_{\mathbf A}}
	{(\mathbf X, \mathbf y)}{ \left( \dots , [\Omega_{A_i}(\mathbf X, \mathbf y)], 
	\dots \right)}
\] with $0 \le l \le n$ and with coordinates
\[
	\Omega_{i \mathbf a}(\mathbf X, \mathbf y) =
	\left(\prod_{j=1}^l  
			X_j^{(\mathbf a - \boldsymbol \theta_i) (-\boldsymbol \xi_j)}
	\right)
	e^{(\mathbf a - \boldsymbol \theta_i) \sum_{j=l+1}^n (-\boldsymbol \xi_j) y_j}
.
\]
	Each of $\boldsymbol \xi_j$ is minimal in $\Lambda^*$
	spanning a ray in the outer fan of $\mathbf A = (A_1, \dots,
	A_n)$. There is an $n$-cone in the outer fan containing the
	$\boldsymbol \xi_j$ and possibly more rays.
	The shift $\boldsymbol \theta_i \in \mathbb Z^n$ can be chosen so that
	each $\Omega_{A_i}$ is well-defined at $\mathbf X=0$ and is not a multiple
	of $X_j$.
The domains are of the form
\[
		\begin{split}
	\mathcal D_{\Omega} = \left\{ \verybigstrut (\mathbf X, \mathbf y)
			\subset \mathbb C^l \times \mathbb C^{n-l} : \right. \ &
			|X_1|, \dots, |X_l| < e^{-\Phi \|\Re (\mathbf y)\|_{\infty} - \Psi},
			\\ & \left.
			-\frac{\Phi^{n-l}-1}{\Phi-1}\Psi - \epsilon <
	\Re(y_{l+1}), \dots,  \Re(y_{n}) < \epsilon 
	\verybigstrut \right\}
\end{split}
\]
	for arbitrary constants $\Phi>1$ and $\Psi>0$ to be determined later.
	The constant $\epsilon > 0$ can be taken as small as desired. 
\end{theorem}

\begin{example}[Main chart] When $l=0$ and $(-\boldsymbol \xi_1, \dots, -\boldsymbol \xi_n)$ is the
	canonical basis, $\Omega_{\mathrm D}=V_{\mathbf A} \circ \exp$ is the chart used in
	the two previous papers.
\end{example}

\begin{example}\label{example1} Here is an example of a singular and ramified 
chart for a toric variety with $\mathbf A=(A, \dots, A)$
	with $n=l=1$. Let $A=\{ 0 ; 2\}$, $\xi=-1$ and
\[
\Omega(X) = \left[ \begin{pmatrix} 1 \\ X^2 \end{pmatrix}
	\right]
	.
\]
	At the origin, $D\Omega(0) = 0$ so the parameterization is singular.
	Every other point $[1: X^2]$ 
	in $\mathcal V_{\mathbf A}$ 
	has exactly two
	preimages $\pm X$ by the parameterization, so the chart $\Omega^{-1}$
	is ramified at toric infinity.
\end{example}

\begin{example}\label{example2}  We can modify this example to obtain a singular unramified
	chart:
	Let $A=\{ 0 ; 2; 3;\}$, just set
\[
\Omega(X) = \left[ \begin{pmatrix} 1 \\ X^2 \\ X^3 \end{pmatrix}
	\right]
	.
\]
	At the origin, $D\Omega(0) = 0$ so the chart is still singular.
	But the paremeterization is locally injective.
\end{example}

\begin{remark}
	If $l<n$, we always have $\Omega(\mathbf X, \mathbf y + 2 \pi \sqrt{-1} \mathbf r) = \Omega(\mathbf X, \mathbf y)$ for all $\mathbf r \in \mathbb Z^n$.
Properly speaking, the atlas consists of local inverses of
the parameterizations $\Omega$. As the examples above show, those charts may 
be ramified or singular at toric infinity.
\end{remark}

\subsection{Classification of points at infinity}

At this point, it is convenient to switch to logarithmic coordinates.
Let $\mathbf Z = \exp(\mathbf z)$ coordinatewise, and define
$v_{A_i}(\mathbf z) = V_{A_i}(\mathbf Z)$. Also, $v_{\mathbf A}(\mathbf z) =
V_{A_i}(\mathbf Z)$.

Recall that
$\mathbf v=([\mathbf v_1],\dots,[\mathbf v_n]) \in \mathcal V_{\mathbf  A}$ 
is finite if 
$\mathbf v=([V_{A_1}(\mathbf Z)],$ $\dots,$ $[V_{A_n}(\mathbf Z)])$ for $\mathbf Z \in \mathbb C_{\times}^n$.
This is the same as 
$\mathbf v = ([v_{A_1}(\mathbf z)], \dots, [v_{A_n}(\mathbf z)])$ for $\mathbf z \in \mathbb C^n$.
A point is at toric infinity if it is not finite. Points at toric infinity can be classified through
the following Lemma:

\begin{lemma}\label{classification-infinity}
	Let $\mathbf v \in \mathcal V_{\mathbf A}$ be at toric infinity. Then there
	is a minimal non-trivial 
	cone $\sigma_{\infty}=\sigma_{\infty}(\mathbf v)$ in the outer 
	fan of $\mathbf A=(A_1, \dots, A_n)$
	with the following property: there are $\boldsymbol \chi$ 
	in the relative
	interior of $\sigma_{\infty}$ and $\mathbf z \in \mathbb C^n$,
	$\mathbf z \perp \boldsymbol \chi$ so that
\[
	\mathbf v = \lim_{ \tau \rightarrow \infty}
	\mathbf v_{\mathbf A}(\mathbf z + \tau \boldsymbol \chi)
.
\]
\end{lemma}

The proof of this Lemma is postponed to Section~\ref{proof-classification-infinity}.
We extend this association to finite points in the toric variety through the
following convention.
If $\mathbf v$ is finite, then $\sigma_{\infty}(\mathbf v)=\{\mathbf 0 \}$ is the $0$-dimensional cone and $\boldsymbol \chi = \mathbf 0$. 
The generators of the cone $\sigma_{\infty}(\mathbf v)$ will tell us which of the $X$ coordinates of $\mathbf v$ must vanish. To
find the proper system of coordinates around $\mathbf v$, we want also to know which of the $X$ coordinates are small
in absolute value. So we introduce a new family of cones.

Suppose first that $\mathbf v = \mathbf v(\mathbf z)$ is `finite'.
Then we define $\sigma(\mathbf v)$ as the smaller cone in the outer fan containing $\Re(\mathbf z)$.
Typically but not always, $\sigma(\mathbf v)$ is $n$-dimensional. However, there are situations where $\sigma(\mathbf v)$ is lower dimensional. For instance, if $\Re(\mathbf z)$ generates a ray. The construction above can be extended to all of $\mathscr V_{\mathbf A}$
through the following definition.

\begin{definition}
	For every $\mathbf v \in \mathcal V_{\mathbf A}$, we define
	$\sigma(\mathbf v)$ as the minimal cone in the outer fan of
	$\mathbf A=(A_1, \dots, A_n)$, containing
	$\Re(\mathbf z) + \tau \boldsymbol \chi$, for $\tau$ large enough and some $\chi \in \sigma_{\infty}(\mathbf v)$.
\end{definition}

At this point, we associated to each $\mathbf v \in \mathcal V_{\mathbf A}$
a couple of cones $\sigma_{\infty}(\mathbf v) \subset \sigma(\mathbf v)$ in the outer
fan of $\mathbf A$. Some of the generators of the cone $\sigma$ correspond to the small $X$ coordinates. But there can be
more generators than coordinates, so we need an extra result before constructing the atlas.

\subsection{Carathéodory's theorem}
Suppose $\mathbf v \in \mathcal V_{\mathbf A}$ is given. Towards the proof of Theorem~\ref{th-coords}, 
we need to produce a suitable set of generators 
$\boldsymbol \xi_1, \dots, \boldsymbol \xi_n$ 
of the cones 
$\sigma_{\infty}(\mathbf v) \subset \sigma(\mathbf v)$.
A well-known result allows
to reduce the number of non-zero coordinates in a positive linear combination, so not
all the generators for $\sigma_{\infty}(\mathbf v) \subset \sigma(\mathbf v)$ will be
necessary.

\begin{theorem}[Carathéodory]\label{caratheodory} Let $\mathbf x$ be a non-negative linear combination 
of vectors $\boldsymbol \xi_1, \dots, \boldsymbol \xi_m \in \mathbb R^n$.
	Let $d = \dim (\mathrm{Span}(\boldsymbol \xi_1, \dots, \boldsymbol \xi_m))$.
	Then, $\mathbf x$ is a positive linear combination of a subset of at most
	$d$ of the $\boldsymbol \xi_i$.
\end{theorem}

The choice of the $d$ vectors is not unique in general. 
It is convenient to restate Carathéodory's Theorem in an effective form,
so that the choice of the coordinate charts can be done efficiently by algorithm. 

Let $\Xi$ be the $n \times m$ matrix with columns $\boldsymbol \xi_1, \dots, \boldsymbol \xi_m$, its rank is $d$. Let $\mathbf b \in \mathbb R^n$ be strictly positive.
Consider now the Linear Programming Problem:
	\begin{equation}\label{min1}
	\begin{array}{lrcl}
		\text{Minimize}& \varphi(\mathbf y) &\defeq& \mathbf b^T \mathbf y \\
	\text{subject to } 
		&\mathbf x &= &\Xi \mathbf y\\
		&\mathbf y &\ge& 0.
	\end{array}
\end{equation}
\begin{theorem}[Carathéodory as a LPP]\label{caratheodory2}
	Let $\mathbf x = \Xi \mathbf y$ for $\mathbf y \ge 0$. Let
	$d=\rank(\Xi)$. The Linear Programming Problem \eqref{min1} admits
	a minimum
	$\mathbf y^*$ with at most $d$ non-zero coordinates.
\end{theorem}

A direct proof of Theorem~\ref{caratheodory2} is postponed to section~\ref{proof-caratheodory2}, as well as a trivial algorithmic solution.  The following Lemma is easy. It will be needed to find the optimal splitting $\mathbb R^l \times \mathbb R^{n-l}$,
once the basis for the system of coordinates is selected.
\begin{lemma}\label{phipsi}
	Let $\Phi > 1,\Psi > 0$ be given. If $\infty=h_0 \ge h_1 \dots \ge h_n \ge h_{n+1}=0$,
	then there is $l$ maximal such that $h_l > \Phi h_{l+1} + \Psi$ and in that case,
	$h_{l+1}\le \frac{\Phi^{n-l}-1}{\Phi-1}\Psi$.
\end{lemma}

\begin{remark}

\end{remark}

\begin{proof}[Proof of Theorem~\ref{th-coords}] Let $\mathbf v \in \mathscr V_{\mathbf A}$
	From Lemma~\ref{classification-infinity}, we can write
	\[
	\mathbf v = \lim_{ \tau \rightarrow \infty}
	\mathbf v_{\mathbf A}(\mathbf z + \tau \boldsymbol \chi),
	\]
	for $\boldsymbol \chi$ in the relative interior of $\sigma_{\infty}(\mathbf v)$,
	and $\mathbf z$ with real part in $\sigma(\mathbf v)$, $\mathbf z \perp 
	\boldsymbol \chi$.
	For short, write $\sigma_{\infty}=\sigma_{\infty}(\mathbf v)
	\subset \sigma=\sigma(\mathbf v)$.
	Also, let $k=\dim(\sigma_{\infty}) \le d=\dim(\sigma)$. We still
	must find $l$ with $0 \le k \le l \le d \le n$.
\medskip
\noindent \\
{\bf Step 1:}
	There are finitely many generators of $\sigma$.
	Pick a large value of $\tau$.
	From Theorem~\ref{caratheodory},
	there are rays 
	$\boldsymbol \xi_1, \dots \boldsymbol \xi_d$
	linearly independent, so that 
	\begin{equation}\label{C1}
	\mathbf z + \tau \boldsymbol \chi \in \mathrm{Cone}\left( \boldsymbol \xi_1,
	\dots, \boldsymbol \xi_d\right). 
	\end{equation}
	Without loss of generality, we take the $\boldsymbol \xi_j \in \Lambda^*$.
	In case $\mathbf z + \tau' \boldsymbol \chi$ leaves the cone
	above for some $\tau' > \tau$, the choice of the $\boldsymbol \xi_j$
	is invalidated. Replace $\tau$ by $\tau'$ and start again. Since there
	are only finitely many possible choices of the $\boldsymbol \xi_j$,
	this choice will eventually stabilize.
\medskip
\noindent \\
{\bf Step 2:}
	From Theorem~\ref{caratheodory} again, there is a subset
	(say) $\{\boldsymbol \xi_1, \dots,\boldsymbol \xi_k\} \subseteq \{\boldsymbol \xi_1, \dots, \boldsymbol \xi_d\} $ so that $\boldsymbol \chi \in \mathrm{Cone}\left( \boldsymbol \xi_1,
	\dots, \boldsymbol \xi_k\right)$. 	
	Pick $\boldsymbol \xi_{d+1}, \dots, \boldsymbol \xi_n$ inductively
	so
	that $\mathrm{Cone}\left( \boldsymbol \xi_1,
	\dots, \boldsymbol \xi_j\right)$ is a $j$-dimensional
	cone and a subset of a cone in the fan, $j=d+1$ to $n$.
	So we obtained $n$ rays contained in $\sigma$. The dimension
	of $\sigma$ may be $<n$ but in this case, Hypothesis~\ref{NDH} implies
	that there is an $n$-dimensional
	cone $\sigma_n \supset \sigma$ in the fan. It can be obtained by an
	infinitesimal generic perturbation of $\mathbf z + \tau \boldsymbol \chi$.
\medskip
\noindent \\
{\bf Step 3:} 
	Write $\mathbf z = -\sum_{j=k+1}^n  y_j \boldsymbol \xi_j$, and
	notice that the minus sign guarantees $\Re(y_j) \le 0$ for all $j$.
	Let $h_j=-\Re(y_j)$, the $h_j$ are non-negative. Reorder the $\boldsymbol \xi_j$, $h_j$
	and $y_j$ so that the $h_0=\infty = \dots = h_k > h_{k+1} \ge \dots \ge h_n \ge 0$.
	Pick $l$ as in Lemma~\ref{phipsi}, then $k \le l$ and moreover,
\[
	h_l > \Phi h_{l+1} + \Psi \hspace{1em}\text{and}\hspace{1em} 
	h_{l+1}\le \frac{\Phi^{n-l}-1}{\Phi-1}\Psi.
\]
	Finally, set
	$X_1=\dots=X_k=0$, and $X_j = e^{-y_j}$ for $k < j \le l$. Then
	$|X_1| \le \dots \le |X_l| \le 
	e^{-\Phi |\Re (y_{l+1})| - \Psi}$. 
	We retain coordinates $X_1, \dots, X_l$, $y_{l+1}, \dots, y_n$.
	They satisfy $|\Re y_l| \le \frac{\Phi^{n-l}-1}{\Phi-1}\Psi$. In particular, this point
	$(\mathbf X, \mathbf y)$
	belongs to one of the open domains $\mathcal D_{\Omega}$ in the statement of the Theorem.
\medskip
\noindent \\
	{\bf Step 4:} Hypothesis \ref{NDH} implies that the cones in the fan of $\mathbf A$
	are {\em pointed}, that is they cannot contain
	a straight line through the origin. In particular,
	there is $\boldsymbol \theta_i
	\in A_i$ with $(\boldsymbol \theta_i - \mathbf a) \boldsymbol \xi_j
	\ge 0 \ \forall \mathbf a \in A_i, 1 \le j \le n$ and with at least one
	equality for each pair $(i,j)$.
	Since the cone generated by $\boldsymbol \xi_1, \dots,
	\boldsymbol \xi_n$ is $n$-dimensional, the shift $\boldsymbol \theta_i$
	is unique.  
\end{proof}

\begin{remark}\label{effective} An effective proof of the
Theorem goes as follows. It is assumed that $\mathbf v=\mathbf v_{\mathbf A}(\mathbf z)$
is finite or that $\mathbf v$ is infinite, but is approximated well enough by
$\mathbf v=\mathbf v_{\mathbf A}(\mathbf z+\tau \chi)$ for a large $\tau$ given.
In both cases, use the algorithm for Theorem~\ref{caratheodory2} to select
the generators $\boldsymbol \xi_i$ as in Step 1.
\end{remark}

\subsection{Proof of the Lemma~\ref{classification-infinity}}
\label{proof-classification-infinity}

\begin{proof}[Proof of Lemma \ref{classification-infinity}]
Suppose that $\mathbf v=([\mathbf v_{1}], \dots, [\mathbf v_{n}])$ and scale the
representatives $\mathbf v_{i}$ so that $\|\mathbf v_{i}\|_{\infty} = 1$.
	There is a sequence of $(v_{\mathbf A}(\mathbf x_t))_{t \in \mathbb N} \in \mathcal V_A$, $x_t \in \mathbb C^n$,
	converging to $\mathbf v$.
	The block
	vector
\[
\mathbf v_t = \begin{pmatrix}
	\frac{1}{\|\mathbf v_{A_1}(\mathbf x_t)\|_{\infty}} \mathbf v_{A_1}(\mathbf x_t)
	\\
	\vdots
	\\
	\frac{1}{\|\mathbf v_{A_n}(\mathbf x_t)\|_{\infty}} \mathbf v_{A_n}(\mathbf x_t)
\end{pmatrix}
\]
can be expressed as $\mathbf v_t = \exp \mathbf w_t$ with 
\[
\mathbf w_t = \bigmatrix{A}{n}
\begin{pmatrix} \boldsymbol \lambda_t \\ \mathbf x_t \end{pmatrix}
\]
Above, $A_i$ is the matrix with rows $\mathbf a \in A_i$, and the vector $\boldsymbol \lambda_t$
	is picked so that $\| \mathbf v_t \|_{\infty}=1$, that is
	$\max \Re(\mathbf w_{it}) = 0$. 
	The sequence $(\mathbf v_t)$ is trapped into the unit ball for the infinity norm, and since the unit ball is compact the sequence $(\mathbf v_t)$ has at least one
	converging
	subsequence. From now on we replace $(\mathbf v_t)$ by this converging 
	subsequence. The sequence $(\mathbf w_t)$ is not convergent:
Let $C_i = \{ a \in A_i: v_{i\mathbf a} \ne 0\}$. 
	Since $\mathbf v$ is not finite, there is at least one value of $i$
	with $C_i \subsetneq A_i$.
	When $\mathbf a \in C_i$, the real part of the 
	coordinate $(\mathbf w_t)_{i,\mathbf a}$ converges. When $\mathbf a \in A_i \setminus C_i$,
	the real part of $(\mathbf w_t)_{i,\mathbf a}$ 
diverges to $-\infty$.

Let $L_{\mathbb C}$ be the complex linear span of $C_1-C_1 \cup \dots \cup C_n-C_n$. 
Decompose $\mathbf x_t = \mathbf z_t + \mathbf n_t$ where $\mathbf z_t \in L_{\mathbb C}$ and $\mathbf n_t \perp L_{\mathbb C}$.
	The sequence $(\Re (\mathbf z_t))_{t \in \mathbb N}$ is bounded and hence 	$(\mathbf z_t)_{t \in \mathbb N}$ admits an accumulation point $\mathbf z$ modulo $2 \pi \sqrt{-1} \Lambda_{\mathbf A}^*$.
	For all $i$ and $\mathbf a \in C_i$, $\lim_{t \rightarrow \infty} v_{i,\mathbf a}(\mathbf z_t) = v_{i,\mathbf a}$.

Consider now the open cone of all pairs $(\boldsymbol \lambda, \boldsymbol \chi) \in \mathbb R \times \mathbb R^n$ so that
\[
\bigmatrix{A}{n}
\begin{pmatrix} \boldsymbol \lambda \\ \boldsymbol \chi \end{pmatrix} \le 0
\]
with equality precisely at rows $(i,\mathbf a)$ with $\mathbf a \in C_i$. This cone is
clearly non-empty and for any pair $(\boldsymbol \lambda, \boldsymbol \chi)$ in this cone,
\[
	\lim_{t \rightarrow \infty} \mathbf v_{A_i}( \mathbf z + t \boldsymbol \chi) =\mathbf  v_i
\]
whence
\[
	\lim_{t \rightarrow \infty} \mathbf v_{A}( \mathbf z  + t \boldsymbol \chi) = \mathbf v.
\]
The above limit does not change if we replace $\mathbf z$ by its orthogonal projection
	$\mathbf z - \frac {\langle \mathbf z , \chi\rangle}{\| \chi \|^2} \chi$.
\end{proof}
		
\subsection{Carathéodory's Theorem as a LPP}
\label{proof-caratheodory2}

\begin{proof}[Proof of Theorem~\ref{caratheodory2}] 
	A viable point $\mathbf y^0$ for \eqref{min1} 
	is given by the hypothesis of Theorem~\ref{caratheodory2}. The intersection of the viable set with
	the half-space $\mathbf b^T \mathbf y \le \mathbf b^T \mathbf y^0$ is convex and 
	compact, so the minimum of the linear form $\mathbf b^T y$ is attained.
	We rewrite the Linear Programming Problem as a Primal-Dual problem with
	complementary slackness condition, see for instance \ocite{BenTal-Nemirovski}*{Th 1.3.2 and 1.3.3}
	\begin{eqnarray*}\label{min2}
		\mathbf x   &=& \Xi \mathbf y \\
		\mathbf b^T &=& \boldsymbol \lambda^T \Xi + \boldsymbol \mu^T \\
		y_j \mu_j &=& 0 \hspace{5em} \text{for }1 \le j \le m \\
		\mathbf y &\ge& 0 \\
		\boldsymbol \mu &\ge & 0
	\end{eqnarray*}

	In order to prove the Theorem, we replace $\mathbf b$ by a generic
	infinitesimal perturbation $\mathbf b(\epsilon)= \mathbf b + \epsilon \dot{\mathbf b}$. This means that no non-trivial integral polynomial 
	of the $\dot b_j$ can vanish.

	Let $J = \{ j \in [m]: y_j > 0\}$. 
	The trivial case is $J=\emptyset$ 
	with $\mathbf x=0$ the linear combination of zero vectors.
	We assume therefore that $J \ne \emptyset$.
	Let $\mathbf b_J$ and $\Xi_J$ be,
	respectively, the subvector with coordinates in $J$ and the
	submatrix with columns in $J$. 
	Then $\mathbf b_J^T(\epsilon) = \boldsymbol \lambda_J^T \Xi_J$. Thus, 
	$\mathbf b_J(\epsilon) \in \Im (\Xi_J^T)$ and this is the same as 
	$\mathbf b_J(\epsilon) \perp \ker (\Xi_J)$.

Let $I$ be a maximal subset of linearly independent rows of $\Xi_J$, we
	will denote by $\Xi_{IJ}$ the submatrix of $\Xi$ with 
	rows in $I$ and columns in $J$. Since $\Xi$ has rank $d$, $|I| \le d$.
	Moreover, $\Xi_{IJ}$ is surjective and $\ker (\Xi_{IJ})= \ker (\Xi_J)$.
	Hence,
	$\Xi_{IJ}^{\dagger} \Xi_{IJ} \mathbf b_J(\epsilon) =  \mathbf b_J(\epsilon)$.
	Applying the well-known formula for the pseudo-inverse of 
	surjective matrices,
	we obtained a system of linear equations
\[
	\Xi_{IJ}^T\ 
	\left(\Xi_{IJ}\ 
	\Xi_{IJ}^T\right)^{-1}\
	\Xi_{IJ}\  
	\mathbf b_J(\epsilon) = \mathbf b_J(\epsilon). 
\]
	But since $\mathbf b(\epsilon)
	= \mathbf b + \epsilon \dot{\mathbf b}$, we have that
\[
	\left( \Xi_{IJ}^T\ 
	\left(\Xi_{IJ}\ 
	\Xi_{IJ}^T\right)^{-1}\
	\Xi_{IJ}\ - I\right) 
	\dot {\mathbf b_J} = 0 .
\]
	The coefficients of $\Lambda^*$ are rational. So are those of the matrix $\Xi$ and submatrices. Clearing
	denominators in the expression above,
	we obtain an integer linear equation on $\dot {\mathbf b_J}$.
	Since $\dot {\mathbf b}$ is generic, its vanishing implies
	that the matrix $\Xi_{IJ}^T\ 
	\left(\Xi_{IJ}\ 
	\Xi_{IJ}^T\right)^{-1}\
	\Xi_{IJ}$ is equal to the identity, that is
	$|I|=|J|$. Thus, $|J|\le d$.

	Thus, the minimum $\mathbf y$ has at most $d$ non-zero coordinates.
	By taking the limit with $\epsilon \rightarrow 0$, we obtain a minimum $\mathbf y^*$ for the original problem, with at most $d$ non-zero coordinates. 
\end{proof}

\begin{remark} \label{algorithmic-reduction}
The primal-dual system can be solved 
by steepest descent for the potential $\mathbf y \mapsto \mathbf b^T \mathbf y$
along the kernel of $\mathbf \Xi_J$, and then updating $J$. No claim of optimality in the
algorithm below.

\begin{algorithm2e}[H]
\SetKwInput{KwInput}{Input}                
\SetKwInput{KwOutput}{Output}              
\SetKwRepeat{Repeat}{repeat}{end}
\KwInput {An integral $m \times n$ matrix $\Xi$, and vectors
	$\mathbf b, \mathbf y^{0} \in \mathbb R^m$ with $\mathbf y, \mathbf b > 0$ such that $\mathbf x = \Xi \mathbf y^{0}$. It is assumed that $\mathbf b$ is generic over $\mathbb Q$.}
\KwOutput {A vector $\mathbf y \ge 0$ so that $\mathbf x = \Xi \mathbf y$. If $b$ is generic over 
	$\mathbb Q$, the vector $\mathbf y$ has at most $\mathrm{rank}(\Xi)$ non-vanishing
	coordinates.}
$J \leftarrow [m]$\;
	$\mathbf y \leftarrow \mathbf y^{0}$\;
\Repeat{}{
	$\dot {\mathbf y} \leftarrow -(I - \Xi_J^{\dagger} \Xi_J) \mathbf b$\;
\lForEach{$j \in J$}{$h_j \leftarrow -\dot y_j/y_j$}
\lIf{$\forall j \in J, h_j \le 0$}{\Return {($\mathbf y$)}}
$h \leftarrow \text{smallest positive value of $h_j$, $j \in J$}$\;
	$\mathbf y_J \leftarrow \mathbf y_J + h \dot {\mathbf y}_J$ \;
$J \leftarrow J \setminus \{ j \in J: y_j = 0 \}$\;
}
	\caption{Selection of $d$ vectors\label{alg-selection}}
\end{algorithm2e}

\end{remark}

\section{Normal form at toric infinity}
\label{sec:normal}

The coordinate charts introduced in Theorem~\ref{th-coords} are akin to monomial
transforms or changes of coordinates. In this section, 
the tuple $\mathbf A$ will be put into  
normal form while preserving the geometry of the toric variety $\mathscr V_{\mathbf A}$.
In normal form, 
points near `toric infinity' will have a particularly simple representation,
in the line of Examples~\ref{example1} and \ref{example2} above.

\subsection{Normal form at toric infinity}
\begin{definition}\label{normal-form}
	The tuple $\mathbf A = (A_1, \dots, A_n)$, $A_i \subset \mathbb R^{n}$, $A_i - A_i \subset \mathbb Z^n$, $\# A_i < \infty$, is in {\em normal form} with splitting $\mathbb R^n = \mathbb R^l \times \mathbb R^{n-l}$, $l \ge 1$, if and only if,
\begin{enumerate}[(a)]
	\item For all $i$ and for all $\mathbf a =[\mathbf b,\mathbf c] \in A_i$, $\mathbf b \ge \mathbf 0$. 
	\item For all $i$, there is $[\mathbf b,\mathbf c] \in A_i$ such that $\mathbf b=\mathbf 0$.
	\item For all $i$, $\sum_{[\mathbf 0,\mathbf c] \in A_i} c =\mathbf 0$. 
	\item $\mathrm{Cone}(-\mathrm e_1)$, \dots,
		$\mathrm{Cone}(-\mathrm e_n)$ are rays ($1$-dimensional cones) in the outer fan of $\mathbf A$.
	\item There is an $n$-cone $\sigma'$ in the outer fan of $\mathbf A$ containing $-\mathrm e_1$, \dots, $-\mathrm e_n$.
\end{enumerate}
\end{definition}

Through this and the next subsection, we assume that $\mathbf A$ is in normal form with splitting
$\mathbb R^n = \mathbb R^l \times \mathbb R^{n-l}$. We will explore some
consequences of the normal form. Next, we will show how to reduce an arbitrary
tuple of supports to normal form.

First consequence of normal form, the local parameterization of Theorem~\ref{th-coords} 
simplifies to
\[
	\boldsymbol{\Omega}_{\mathbf A}: \mathbf X,\mathbf y \mapsto \left([\Omega_{A_1}(\mathbf X,\mathbf  y)], \dots, [\Omega_{A_n}(\mathbf X,\mathbf y)]\right)
\]
with coordinates $\Omega_{i \mathbf a}$ as follows. We split each
multi-index $\mathbf a$ in two blocks, $\mathbf a = [\mathbf b \, \mathbf c]$
with $\mathbf b \in \mathbb Z^l$ and $\mathbb c \in \mathbb R^{n-l}$.
Then,
$\Omega_{i \, \mathbf a}(\mathbf X, \mathbf y) =
\mathbf X^{\mathbf b} e^{\mathbf c  \mathbf y}$ and the points at infinity are those with at least one of the
$X_i$ equal to zero. 
\medskip
\par
The condition number can be computed explicitly. Since we would like to
extend definition \eqref{defmu} to points that are not zeros of
a system $\mathbf f$, we need to project the image of
$D[\Omega_{A_i}]$ into tangent space $T_{[\Omega_{A_i}(\mathbf X,\mathbf y)]} \mathbb P(\mathbb C^{A_i})$.
The projection is
\begin{equation}\label{projection}
P_i=	P_{T_{[\Omega_{A_i}(\mathbf X, \mathbf y)]} \mathbb P(\mathbb C^{A_i})} = 
\left(I - \frac{1}{\|\Omega_{A_i}(\mathbf X, \mathbf y)\|^2} \Omega_{A_i}(\mathbf X, \mathbf y)
\Omega_{A_i}(\mathbf X, \mathbf y)^*\right)
\end{equation}
and the condition number is 
\[
\mu( \mathbf f, \Omega_{\mathbf A}(\mathbf X,\mathbf y) ) =
\left\|
\left(
\diag{
	\frac{1}{\|f_i\| \| \Omega_{A_i}(\mathbf X,\mathbf y)\|}
}
\begin{pmatrix}
	\vdots \\
	f_i \, P_i \, D\Omega_{A_i}(\mathbf X,\mathbf y)
	\\
	\vdots
	\\
\end{pmatrix}\right)^{-1}
\right\|_{\boldsymbol{\Omega}_{\mathbf A}(\mathbf X, \mathbf Y)}
\]
with 
\[
D\Omega_{A_i} = \diag{\Omega_{i \mathbf a}} A_i 
\begin{pmatrix}
	\diag{\mathbf X}^{-1} & \\ 
	 & I
\end{pmatrix}
\]
and where $\|\cdot\|_{\boldsymbol{\Omega}_{\mathbf A}(\mathbf X, \mathbf Y)}$ stands for the
operator norm
$\mathbb C^n \rightarrow T_{\boldsymbol{\Omega}_{\mathbf A}(\mathbf X, \mathbf Y)} \mathscr V_{\mathbf A}$.
The expression above remains valid if $X_i=0$, as only rows with $b_i \ge 1$ are
multiplied by $X_i^{-1}$.
When $\mathbf X=e^{\mathbf x}$ the two expressions for the condition coincide:
\[
	\mu(\mathbf f, \Omega_{\mathbf A}(\mathbf X, \mathbf y))
	=
	\mu(\mathbf f, \mathbf v_{\mathbf A} (\mathbf x, \mathbf y))
.
\]
\medskip
\par
It is worth noting that if $\Omega_A(\mathbf X, \mathbf y)$ is a toric zero of $\mathbf f$,
$f_i P_i = f_i$ so we can omit $P_i$ from the expression above.

\subsection{The condition number at toric infinity}
Of particular interest is the condition number at toric infinity.
It can be computed using the parameterizations $\Omega_{\mathbf A}$
when $\mathbf X \rightarrow 0$. Computations are easier with $\mathbf y=0$. 
For short, write $\omega_i = \Omega_{A_i}(0,0)$ and 
$\boldsymbol \omega=([\omega_1],\dots,[\omega_n])$. 
The tangent space $T_{\boldsymbol \omega}\mathscr V_{\mathbf A}$ of the
toric variety $\mathscr V_{\mathbf A}$ can be parameterized by $D[\boldsymbol \Omega_{\mathbf A}](\mathbf 0,\mathbf 0)$,
so we define the norm
\[
	\| \mathbf u \|_{\boldsymbol \omega} = \| D[\boldsymbol \Omega_{\mathbf A}](\mathbf 0,\mathbf 0) \mathbf u\| 
\]
as the pull-back of the Fubini-Study norm restricted to 
$T_{\boldsymbol\omega}\mathscr V_{\mathbf A}$. We define
coordinatewise
\[
\| \mathbf u \|_{\omega_i} = \| D[\Omega_{A_i}](\mathbf 0,\mathbf 0) \mathbf u\|
\]
so that
\[
\|\mathbf u \|_{\boldsymbol \omega} = \sqrt{ \sum_i \| \mathbf u\|_{\omega_i}^2}
.
\]
We can also define a Finsler-style norm
\[
\finsler{\mathbf u}{\boldsymbol \omega} = \max_i \| \mathbf u \|_{\omega_i}
.
\]

The condition number at $\boldsymbol \omega$ simplifies to 
\[
	\mu(\mathbf f,\boldsymbol \omega) = \left\| \left(\diag{\frac{1}{\|f_i\|}}
\begin{pmatrix}
f_1 D[\Omega_{A_1}](\mathbf 0,\mathbf 0) \\
\vdots\\
f_n D[\Omega_{A_n}](\mathbf 0,\mathbf 0) 
\end{pmatrix}
\right)^{-1}
\right\|_{\boldsymbol \omega} .
\]

In order to investigate the condition number and the higher derivatives of $\Omega_{A_i}$ at $(0,0)$, we decompose each matrix $A_i$ 
into blocks:
\begin{equation} \label{A-blocks}
A_i= 
\begin{pmatrix}
	0 & C_i^{(0)} \\
	B_i^{(1)} & C_i^{(1)} \\
	\vdots & \vdots \\
	B_i^{(r)} & C_i^{(r)}
\end{pmatrix}
\end{equation}
where it is assumed that each row $\mathbf b$ of $B_i^{(r)}$ satisfies $|\mathbf b|=r$.
We also write $A_i^{(r)}=(B_i^{(r)}\ C_i^{(r)})$.
Let
\begin{equation}\label{Li}
	L_i = \frac{1}{\|\omega_i\|}
	\begin{pmatrix}0 & 
	C_i^{(0)} \\ B_i^{(1)} & 0 \end{pmatrix}
	\hspace{2em}\text{and} \hspace{2em}
	L = \begin{pmatrix} L_1 \\ \vdots \\ L_n \end{pmatrix} .
\end{equation}

The derivative of $[\Omega_{A_i}]$ at $(\mathbf 0,\mathbf 0)$ can be computed explicitly:

\begin{lemma}\label{lem-Li}
\[
	D[\Omega_{A_i}](\mathbf 0,\mathbf 0) = L_i
\hspace{2em}	\text{and}
\hspace{2em}
	D[\mathbf \Omega_{\mathbf A}](\mathbf 0,\mathbf 0) = L
\]
\end{lemma}

\begin{remark}
	Definition~\ref{normal-form}(b) guarantees that $\#A_i^{(0)} \ge 1$.
There is no guarantee that $L_i$ is invertible.
\end{remark}

\begin{proof}[Proof of Lemma~\ref{lem-Li}]
	The derivative $D\Omega_{A_i}(\mathbf 0,\mathbf 0)$ is precisely
\[
D\Omega_{A_i}(\mathbf 0,\mathbf 0) = 
	\begin{pmatrix} 0 & C_i^{(0)}\\
		B_i^{(1)} & 0 \\
	0 & 0 \\
	\vdots & \vdots\end{pmatrix}
\]
	Because of Definition\ref{normal-form}(c), 
$\omega_i ^* D\Omega_{A_i}(\mathbf 0,\mathbf 0) = 0$ and hence
\begin{eqnarray*}
	D[\Omega_{A_i}(\mathbf 0,\mathbf 0)] &=& 
	\frac{1}{\|\omega_i\|} 
	\left(I-\frac{1}{\|\omega_i\|^2} \omega_i \omega_i^*\right)
	D\Omega_{A_i}(\mathbf 0, \mathbf 0)
\\
&=&
 		\frac{1}{\|\omega_i\|}
	\begin{pmatrix} 0 & C_i^{(0)}\\
		B_i^{(1)} & 0 \\
	0 & 0 \\
	\vdots & \vdots\end{pmatrix}
\\
	&=&
	\begin{pmatrix} L_i \\ 0 \end{pmatrix}.
\end{eqnarray*}
\end{proof}

The following is obvious:
\begin{lemma} If $\mu(\mathbf f, \boldsymbol \omega) < \infty$, then
	$L$ has full rank.
	$L$ has full rank if and only if  
	$\boldsymbol \omega$ is a smooth point of $\mathscr V_{\mathbf A}$.
\end{lemma}

For later reference, the Lemma below is an immediate consequence of Lemma~\ref{lem-Li}:

\begin{lemma}\label{inf-norms}
Let $\mathbf A$ be in normal form, and $1 \le l \le n-1$. If $L$ is non-singular, then
\begin{enumerate}[(a)]
\item	
	For any $\mathbf w_1 \in \mathbb R^{l}$,
		$
		\| \mathbf w_1 \|_{\infty} \le \finsler{(\mathbf w_1,\mathbf 0)}{\omega} 
		\le 
		\max_i (\|\omega_i\|) \ 
		\|\mathbf w_1\|_{\infty}$.
\item   There are constants $0<\underline{\sigma} \le \overline{\sigma}$ such that
	For any $w_2 \in \mathbb R^{n-l}$,
		$\underline{\sigma}\| \mathbf w_2 \|_{\infty} \le \|\mathbf w_2\|_{\omega} \le \overline{\sigma} \|\mathbf w_2\|_{\infty}$.
\end{enumerate}	
\end{lemma}

\subsection{Monomial changes of coordinates}
The reduction of a tuple $\mathbf A$ into normal form
can be achieved by a 
monomial change of coordinates composed with multiple shifts.

\begin{definition}
Let $F=F(\mathbf X)$ be a polynomial in variables $\mathbf X=(X_1, \dots, X_n)$. A
{\em monomial change of coordinates} is a change of coordinates of the form
$X_j = 
Z_1^{\Xi_{1j}}
Z_2^{\Xi_{2j}}
\dots
Z_n^{\Xi_{nj}}$ where $\Xi$ is an $n \times n$ matrix, invertible
over $\mathbb R$. For short, we write $\mathbf X=\mathbf Z^{\Xi}$. 
	This convention has the advantage that $\mathbf X = \exp (\Xi\, \Log{\mathbf Z})$ where
	$\mathbf X$ and $\mathbf Z$ are treated as vectors.

If $\Xi$ is invertible over $\mathbb Z$, we say that
the change of coordinates is a {\em monomial transform}. If the support
of $F=F(\mathbf X)$ is $A \subset \mathbb Z^n$, then the support of $G=G(\mathbf Z)=
F(Z^{\Xi})$ is $A \, \Xi=\{ \mathbf a\, \Xi: \mathbf a \in A\}$ (we are using row notation for the exponents in a polynomial).
\end{definition}

The same definition applies for systems of  exponential sums, and a monomial change of coordinates maps a system with supports $(A_1, \dots, A_n)$ into a system with supports $(A_1\,  \Xi, \dots, A_n \, \Xi)$. The lattice
$\Lambda_A$ spanned by $(A_1 - A_1) \cup \dots \cup (A_n-A_n)$ is mapped into the
lattice $\Lambda_A \Xi$, where the row vector convention for the lattice is also assumed. The dual lattice $\Lambda_A^*$ is mapped to $\Xi^{-1} \Lambda_A^*$, using column vector convention for the dual lattice.

Under a monomial transform, the toric variety associated to systems with support $(A_1 \Xi, \dots, A_n \Xi)$ is exactly the same as the toric variety associated to sysems with support $(A_1, \dots, A_n)$. But the main parameterization becomes
\[
	\defun{\mathbf V_{\mathbf A\, \boldsymbol \Xi}}{\mathbb C_{\times}^n}{\mathbb P(\mathbb C^{A_1 \Xi}) \times \dots \times  \mathbb P(\mathbb C^{A_n \Xi})}{\mathbf Z}{
		([V_{A_1\Xi}(\mathbf Z)],\dots,
		[V_{A_n\Xi}(\mathbf Z)])}
\]
so the points at toric infinity for one of the charts are still at toric infinity for the other chart.

Nothing changes in the parameterization above 
if one replaces each $\mathbf a \in A_i$ by $\mathbf a + \boldsymbol \theta_i$. The new support will be sloppily denoted $A_i + \boldsymbol \theta_i$. This is
an instance of the {\em momentum action} of \toricI{\S3.2}.

Formally, let $(\mathfrak S, \circ)$ be the monoid of monomial changes of coordinates and multiple shifts,  
$\mathfrak S = \{ (\Xi, \boldsymbol \theta_1, \dots, \boldsymbol \theta_n) \in 
\mathbb Z^{n \times n} \cap GL(\mathbb R^{n}) \times (\mathbb R^n)^n\}$
and the composition law is
\[
(\Xi, \boldsymbol \theta_1, \dots, \boldsymbol \theta_n) 
\circ
(\Xi', \boldsymbol \theta_1', \dots, \boldsymbol \theta_n') 
=
(\Xi' \Xi, \theta_1 + \theta_1' \Xi, \dots, \theta_n + \theta_n' \Xi)
.
\]
This monoid acts on the right of the $n$-tuple $(A_1, \dots, A_n)$,
\[
(A_1, \dots, A_n) \mapsto
(A_1, \dots, A_n)^{(\Xi,\boldsymbol \theta_1, \dots, \boldsymbol \theta_n)} \defeq
(A_1 \Xi + \boldsymbol \theta_1, \dots, A_n \Xi + \boldsymbol \theta_n).
\]
The same definition applies to the group 
$(\mathfrak T, \circ)$ 
of monomial transforms and multiple shifts,
$\mathfrak T = \{ (\Xi, \boldsymbol \theta_1, \dots, \boldsymbol \theta_n) \in 
GL(\mathbb Z^{n}) \times (\mathbb R^n)^n\}$.

Denote by $\Lambda_A$ the lattice $\Lambda(A_1-A_1 \cup \dots \cup A_n - A_n)$
and $\Lambda_A^*$ its dual lattice. Then $(\Xi, \theta_1, \dots, \theta_n) \in \mathfrak S$ takes the lattice
$\Lambda_A$ into $\Lambda_{A\, \Xi} = \Lambda_A\, \Xi$ (row notation assumed) and
$\Lambda_A^*$ into $\Lambda_{A\, \Xi}^* = \Xi^{-1}\, \Lambda_A^*$.

If the determinant of the support lattice is accounted for,
the generic root count given by Bernstein's Theorem is invariant by
$\mathfrak S$:

\begin{theorem} 
	Let $S \in \mathfrak S$. Let $B=(B_1, \dots, B_n) = (A_1, \dots, A_n)^S$.
Then,
\[
	n!\frac{ V(\conv{B_1}, \dots, \conv{B_n})}{\det \Lambda_B}
	=n!\frac{V(\conv{A_1}, \dots, \conv{A_n})}{\det \Lambda_A}
\]
\end{theorem}

Let $S =(\Xi, \boldsymbol \theta) \in \mathfrak S$. If $\mathbf f$ is a system of exponential sums with support
$(A_1, \dots, A_n)$, we denote by $\mathbf g = \mathbf f^{S}$ the system
with support $B=(A_1, \dots, A_n)^{S}$ and 
same coefficient vectors $g_{i,\mathbf a\Xi+\theta}= f_{i\mathbf a}$. 
One can find first all the roots of $\mathbf g$, and then recover all the
roots of $\mathbf f$. Moreover:

\begin{proposition}
Let $S \in \mathfrak S$. Assume that $\mathbf f$ is a system of exponential sums with support
	$\mathbf A=(A_1, \dots, A_n)$. Write $\mathbf B=\mathbf A^S$ 
	and $\mathbf g = \mathbf f^{S}$, 
then 
\begin{enumerate}[(a)]
\item $\mu(\mathbf g,v_{\mathbf B}(\mathbf x)) = 
	\mu(\mathbf f,v_{\mathbf A}(\Xi \mathbf x))$.
\item The momentum maps satisfy
$m_{B_i}(\mathbf x) = m_{A_i}(\Xi \mathbf x)\Xi + \theta_i $.
\item For any vector $\mathbf u$, 
$\|\Xi \mathbf u\|_{A_i, \Xi \mathbf x} = \| \mathbf u\|_{B_i, \mathbf x}$
and 
$\|\mathbf u\|_{A, \Xi \mathbf x} = \| \mathbf u\|_{B, \mathbf x}$.
\item For all $i$, the invariant $\nu$ from Definition~\toricI{3.16} 
satisfies $\nu_{B_i}(\mathbf x) = \nu_{A_i}(\Xi \mathbf x)$.
\item The invariant $\kappa$ from (\toricI{4.18}) satisfies: for each $i$, $\kappa_{g_i}=\kappa_{f_i}$
and $\kappa_{\mathbf g}=\kappa_{\mathbf f}$
\end{enumerate}
\end{proposition}
\begin{proof}
The toric varieties $\mathscr V_A$ and $\mathscr V_B$ are exactly the same
	and this implies item (a). 
	For item (b), $v_{A_i}(\Xi \mathbf x) = e^{\theta_i \mathbf x} v_{B_i}(\mathbf x)$. Hence,
\[
	m_{B_i}(\mathbf x) = \sum_{\mathbf a \in A_i} 
	\frac{|v_{i\mathbf a}(\mathbf x)|^2}
	{\| v_i(\mathbf x) \|^2} (\mathbf a \Xi + \theta_i)
	= m_{A_i}(\Xi \mathbf x)\Xi + \theta_i.
\]
	Item (c) is trivial since $x(t) = \mathbf x + t \mathbf u$ and
	$z(t) = \Xi \mathbf x + t \Xi \mathbf u$ are the pull-back of the same
	path in $\mathscr V_B = \mathscr V_A$.
	We now prove item (d):
\begin{eqnarray*}
\nu_{A_i}(\Xi \mathbf x) 
	&=& 
\max_{\|\Xi \mathbf u\|_{A_i, \Xi \mathbf x} \le 1} 
\ \max_{\mathbf a \in A_i}
|(a - m_{A_i}(\Xi x)) \Xi \mathbf u |
\\
&=&
\max_{\|\mathbf u\|_{V_i, \mathbf x} \le 1} 
\ \max_{\mathbf a \in A_i}
	|(a \Xi + \theta_i - m_{B_i}(\mathbf x)) \mathbf u |
\\
	&=&
\nu_{B_i}(\mathbf x) 
\end{eqnarray*}
	Item (e) follows directly from the definition of $\kappa$.
\end{proof}

\subsection{Reduction to normal form}
\begin{lemma}\label{lem-normal-form}
Let $\mathbf A= (A_1, \dots, A_n)$ be a tuple of finite subsets of $\mathbb R^{n}$, with $A_i - A_i \subset \mathbb Z^n$.
Let $\sigma$ be an $l$-dimensional cone in
the fan of $\mathbf A$, with 
	$1 \le l$. Let $\boldsymbol \chi$ be in the interior of $\sigma$.
Then there is   
$S=(\Xi, \boldsymbol \theta_1, \dots, \boldsymbol \theta_n) \in \mathfrak S$ such that $\mathbf A^{(S)}$
	is in normal form with respect to a splitting $\mathbb R^n = \mathbb R^l \times \mathbb R^{n-l}$ and 
	$
	S^{-1} \boldsymbol \chi \in
	\mathrm{Cone}(- \mathrm e_1, \dots, -\mathrm e_l)
	\subseteq 
	S \sigma 
	$.
\end{lemma}

\begin{example} Consider the case where $A_1=A_2=A_3$ are the
	rows of
\[
A = 
\begin{pmatrix}
	 1 & 1 &-1 \\
	 1 &-1 &-1 \\
	-1 & 1 &-1 \\
	-1 &-1 &-1 \\
	 0 & 0 & 1
\end{pmatrix}.
\]
The rays of the fan of $\mathbf A$ are spanned by the columns of 
\[
\begin{pmatrix}
	2 & 0 &-2 & 0 & 0 \\
	0 & 2 & 0 &-2 & 0 \\
	1 & 1 & 1 & 1 & -1 
\end{pmatrix}
.
\] 
Let $l=1$ and $\boldsymbol \chi = \begin{pmatrix} 2 \\ 0 \\ 1 \end{pmatrix}$. The first column of $S$ will be $-\boldsymbol \chi$,
and then we complete the matrix using the adjacent rays, that is minus columns
2,4 or 5. 
Let's take
\[
S = 
\begin{pmatrix}
	-2 &  0 & 0 \\
	 0 & -2 &  2 \\
	-1 & -1 & -1  
\end{pmatrix}
\]
We obtain
\[
AS = 
\begin{pmatrix}
	-1 &-1 & 3\\
	-1 & 3 & -1\\
	 3 &-1 & 3\\
	 3 & 3 &-1\\
	-1 &-1 &-1\\
\end{pmatrix}
	\hspace{1em}
	\text{and}
	\hspace{1em}
	AS+ (1, -1/3, -1/3) = 
\begin{pmatrix}
	 0 & -4/3& 8/3\\
	 0 & 8/3& -4/3\\
	 4 & -4/3& 8/3\\
	 4 & 8/3& -4/3\\
	 0 & -4/3& -4/3
\end{pmatrix}.
\]
The vectors $-\mathrm e_1, -\mathrm e_2, -\mathrm e_3$
are outer normals of 2-facets of the convex hull of
$AS+(1,-1/3,-1/3)$. All those facets contain the
point $(0, -4/3, -4/3)$ from the last row. The cone $\sigma'$
is the positive convex hull of the outer normals for all facets
containing that point. There are 4 facets, those normal to
$-\mathrm e_1, -\mathrm e_2, -\mathrm e_3$ and the facet normal
to $\boldsymbol v = \mathrm e_1 - \mathrm e_2 + \mathrm e_3$.
The cone $\sigma'$ is the positive linear hull of
$-\mathrm e_1, -\mathrm e_2, -\mathrm e_3$ and $\boldsymbol v$.
\end{example}

\begin{proof}[Proof of Lemma~\ref{lem-normal-form}]
	Let $\mathbf v = \lim_{t \rightarrow \infty} v(t \boldsymbol \chi)$
	be a point at $\boldsymbol \chi$-infinity. Let $S$ be the matrix
	with columns 
	$-\boldsymbol \xi_1, \dots, -\boldsymbol \xi_n$ where the
	$\boldsymbol \xi_j$ are the outer normals of Theorem~\ref{th-coords}.
	Since $\boldsymbol \chi$ is in the relative interior of $\sigma$,
	there is at least one point $\mathbf a \in A_i$ with $\mathbf a \boldsymbol \chi$ maximal.  
	Pick initially a shift $-\mathbf a S$ for $A_i$ to enforce Definition \ref{normal-form}(b). Then
	pick a shift in the last $n-l$ coordinates to satisfy Definition \ref{normal-form}(c).
\end{proof}

\section{Deterioration of condition, or the cost of renormalization}
\label{sec:renormalization}

The renormalization operator introduced in Section~\toricII{3} acts on the product space
$\mathscr P_{\mathbf A} \times \mathscr V_{\mathbf A}$, and preserves the solution
variety. This action does not fix the condition number, in this section we bound how
badly can condition deteriorate. We also bound the deterioration of condition near toric infinity
for systems in normal form as in Definition~\ref{normal-form}. 
If the toric variety $\mathscr V_{\mathbf A}$ is smooth, this implies on a global bound for the deterioration
of condition under the most favourable chart. 

\subsection{Condition number bounds under renormalization}

Recall that the renormalization operator acts on spaces of polynomials by
\[
\defun{R_i(\mathbf z)}{\mathscr P_{A_i}}{\mathscr P_{A_i}}
{(\cdots, f_{i\mathbf a}, \cdots)}
{f_i R(\mathbf z) = (\cdots, f_{i\mathbf a} e^{\mathbf a \mathbf z}, \cdots)}
\]
and 
\[
	\defun{R(\mathbf z)}{\mathscr P_{A_1} \times \cdots \times \mathscr P_{A_n}}{\mathscr P_{A_1} \times \cdots \times \mathscr P_{A_n}}
{\mathbf f = (f_1, \dots, f_n)}
{\mathbf f R(\mathbf z) = (f_{1} R_1(\mathbf z), \dots, f_{n} R_n(\mathbf z)).}
\]
This is equivalent to the coefficientwise $\mathbb C_{\times}^n$ action 
with $\mathbf Z=e^{\mathbf z}$. It has the property that
\[
	f_i R_i(\mathbf z) v_{A_i}(\mathbf x) 
	=
	f_i v_{A_i}(\mathbf z + \mathbf x) 
\]
or equivalently
\[
	f_i R_i(\mathbf z) V_{A_i}(\mathbf X) 
	=
	f_i V_{A_i}(Z_1 X_1, \dots, Z_n X_n) 
.
\]

In order to bound the renormalized condition number in terms of the classical one, the invariant
below is needed.

\begin{definition}
	The {\em joint thickness} of the tuple $\mathbf A=(A_1, \dots, A_n)$ is
at a point $\mathbf z$ in the domain of the main chart $\mathbf v_{\mathbf A}$
	is
	\begin{equation}\label{joint1}
		\lambda_{\mathbf z} = \inf_{\mathbf 0 \ne \mathbf w \in \mathbb R^n}
		\frac{1}{\finsler{\mathbf w}{\mathbf z}} 
		\max_i 
	\max_{\mathbf a, \mathbf a' \in A_i} 
	(\mathbf a - \mathbf a') \mathbf w.
	\end{equation}
\end{definition}

Hypothesis \ref{NDH} implies that $\lambda_{\mathbf z} \ne 0$.
A particular case is the joint thickness $\lambda_{\mathbf 0}$ 
at the origin $\mathbf z=\mathbf 0$ of the main chart.

\begin{theorem}\label{cost-of-renorm-legacy}
	Assume that the tuple $\mathbf A=(A_1, \dots, A_n)$ satisfies Hypothesis
	\ref{NDH}. If $\mathbf f \cdot \mathbf V_{\mathbf A}(\mathbf z) = \mathbf 0$,
	then
	\[
	\mu(\mathbf f R(\mathbf z), \mathbf 0) \le
		\frac{\sqrt{8 n \max \# A_i}}{\lambda_{\mathbf 0}} 
	e^{2 \max_i \ell_i(z)+\ell_i(-z)} \mu(\mathbf f , \mathbf z) 
\]
with
\[
\ell_i(\mathbf z) \defeq 
	\max_{\mathbf a \in A_i} \Re( \mathbf a \mathbf z).
\]
\end{theorem}

\begin{remark}  
	As in Section \toricI{3.2.4}, we may use the
	{\em momentum action} to reduce supports to the case 
	$m_{A_i}(\mathbf 0)=\mathbf 0$ for all $i$. Under that assumption, 
	\[
		\lambda_{\mathbf 0}\finsler{\mathbf z}{\mathbf 0}
        \le
	\max_i \ell_i(\mathbf z) + \ell_i(-\mathbf z)
	\le 2 \nu(\mathbf 0) \finsler{\mathbf z}{\mathbf 0}.
\]
\end{remark}

\subsection{Partial renormalization}

It is worth noticing that the bound in Theorem~\ref{cost-of-renorm-legacy} is {\bf really} bad if $\|\Re (\mathbf z)\|_{\infty}$ is large, that is if $\mathbf v_{\mathbf A}(\mathbf z)$ is near toric infinity.
This is the typical reason for switching to normal form. Doing so requires
the support tuple $\mathbf A$ to satisfy a condition stronger than Hypothesis~\ref{NDH}. That condition
ensures that normal coordinates are a regular, non-degenerate coordinate system.

\begin{hypothesis}[Local non-degeneracy]\label{NDIH}
	The tuple $\mathbf A$ is in normal form with respect to the splitting
	$\mathbb R^l \times \mathbb R^{n-l}$ (Definition~\ref{normal-form}), and the point $\boldsymbol \omega = \boldsymbol \Omega_{\mathbf A}(\mathbf 0,\mathbf 0)$
	is a smooth point of the toric variety $\mathscr V_{\mathbf A}$.
\end{hypothesis}

\begin{remark}\label{NDIH2} Given $\mathbf A$ in normal form, Hypothesis~\ref{NDIH} above is equivalent to
	requiring that the matrix $L$ from Lemma~\ref{lem-Li} contains a system of $n$ linearly independent rows,
	one from each $L_i$. A consequence is that $\#A_i^{(0)} > 1$ for all $i$.
\end{remark}

The first $l$ coordinates $X_1, \dots, X_l$ will be supposed to be
small in absolute value, say smaller than a constant $h$ to be determined. 
But we still {\em renormalize} the remaining
coordinates. This technique will be referred as 
{\em partial} renormalization, and will send each pair $(\mathbf f, \Omega(\mathbf X, \mathbf y))$ to a pair $(\mathbf q, \Omega(\mathbf X, \mathbf 0))$ near toric infinity $\omega$.

The partial renormalization operator
will be written as $R_i(0, \mathbf y)$,
\[
q_i := f_i R_i(0, \mathbf y) = [ \dots f_{i\mathbf a} e^{\mathbf c \mathbf y} \dots]_{
	\mathbf a = (\mathbf b, \mathbf c) \in A_i}.
\]
Since
\[
	f_i \Omega_{A_i}(\mathbf X, \mathbf y) = f_i R_i(\mathbf 0, \mathbf y)
	\Omega_{A_i}(\mathbf X, \mathbf 0)
\]
we will approximate $\mathbf q$ near $\boldsymbol \omega$ by a local map
\[
	\mathbf Q(\mathbf X',\mathbf y')= \begin{pmatrix}
	\frac{1}{\|\omega_1\|\|q_1\|} q_1 \Omega_{A_1}(\mathbf X',\mathbf y') \\
		\vdots \\
	\frac{1}{\|\omega_n\|\|q_n\|} q_n \Omega_{A_n}(\mathbf X',\mathbf y') \\
	\end{pmatrix}
\]
with $\mathbf X$ coordinatewise small. In lieu of the condition number
		\[
\mu(\mathbf q, \Omega(\mathbf X, \mathbf 0)) = \left\| 
\diag{\frac{1}{\|\Omega_{A_i}(\mathbf X, \mathbf 0)\|\|q_i\|}}	
\mathbf q \cdot	D\Omega(\mathbf X, \mathbf 0)	\right\|_{\Omega(\mathbf X,\mathbf 0)}
		\]
we will bound the condition of the local map at infinity, viz.
		\[
			\| D\mathbf Q(\mathbf X,\mathbf 0)^{-1}\|_{\boldsymbol \omega} = \left\| 
\diag{\frac{1}{\|\omega_i\|\|q_i\|}}	\mathbf q \cdot	D\Omega(\mathbf X, \mathbf 0)	\right\|_{\boldsymbol \omega}
\]
in terms of the true toric condition number $\mu(\mathbf f, \Omega(\mathbf X, \mathbf y))$.

\medskip
\par
Before proceeding we need two invariants at infinity to replace $\nu$ and $\lambda$.
Define
\begin{equation}\label{defnu}
	\nu_{\boldsymbol \omega}= 
	\max_i (\nu_{i,\omega_i}),
	\hspace{2em}
	\nu_{i,\omega_i}=\sup_{\|\mathbf u\|_{i,\omega_i} \le 1} \left(\max_{\mathbf a \in A_i} | \mathbf a \mathbf u |\right)
.
\end{equation}
\begin{lemma} Assume the partition of $A_i$ into blocks as in Equation
	\eqref{A-blocks}.
	If $1<\# A_i^{(0)}$, then
	$1 \le \nu_{i,\omega}$. 
\end{lemma}

\begin{proof} 
	Because $\# A_i^{(0)} = \# C_i^{(0)} > 1$, there is at least one $\mathbf c \ne 0$
	in $C_i^{(0)}$.
	Let $\mathbf u = 
	\begin{pmatrix} 0 \\ \mathbf u'' \end{pmatrix}$
	with $\|\mathbf u \|_{i,\omega} = 1$.
	From Lemma \ref{lem-Li}
\[
	1=\| \mathbf u \|_{i,\omega} = \frac{1}{\sqrt{\# C_i^{(0)}}} 
	\| C_i^{(0)} \mathbf u'' \| 
	\le
	\max_{\mathbf c \in C_i^{(0)}} |\mathbf c \mathbf u''|
	\le
	\max_{\mathbf a \in A_i} |\mathbf a \mathbf u|
.
\]
\end{proof}

\begin{remark} The original definition of $\nu(\mathbf z)$ involved
	the momentum map. Here,
	the momentum 
\[
	m_i(\mathbf \omega) = \sum_{\mathbf a \in A_i} \frac{| \Omega_{i, \mathbf a}(\mathbf 0,\mathbf 0)|^2}{\|\Omega_{A_i}(\mathbf 0,\mathbf 0)\|^2} \mathbf a
\]
	vanishes: if $\mathbf a \in A_i^{(0)}$, then 
	$\frac{| \Omega_{i, \mathbf a}(\mathbf 0,\mathbf 0)|^2}{\|\Omega_{A_i}(\mathbf 0,\mathbf 0)\|^2}=\frac{1}{\# A_i^{(0)}}$
and $\mathbf a$ is of the form  $( \mathbf 0 \, \mathbf c)$. Def.\ref{normal-form}(c)
guarantees that the sum of all $\mathbf c$ with $(\mathbf 0\,\mathbf c)\in A_i^{(0)}$
	vanishes. For all other $\mathbf a \in A_i$,
	the coefficients $\frac{| \Omega_{i, \mathbf a}(\mathbf 0,\mathbf 0)|^2}{\|\Omega_{A_i}(\mathbf 0,\mathbf 0)\|^2}$
are zero. 
\end{remark}
\medskip
\par
Finally, we will need an extention of the idea of joint thickness 
at toric infinity.
\begin{definition}\label{JTI}
	Under Hypothesis~\ref{NDIH}, the joint thickness of the tuple $\mathbf A$ at $\boldsymbol \omega$ is 
\[
	\lambda_{\boldsymbol \omega} = \inf_{0 \ne (\mathbf w_1,\mathbf w_2) \in \mathbb R^l \times \mathbb R^{n-l}}
	\frac{1}{\finsler{(\mathbf w_1,\mathbf w_2)}{\boldsymbol \omega}} \max_i 
	\max \left(
	\max_{\mathbf b \in B_i^{(0)}} |\mathbf b \mathbf w_1|,
	\max_{\mathbf c, \mathbf c' \in C_i^{(0)}} (\mathbf c - \mathbf c') \mathbf w_2,
	\right).
\]
\end{definition}
It is worth noticing that we maximize $\mathbf b \mathbf w_1$ instead of $(\mathbf b - \mathbf b') 
\mathbf w_1$. Adjoining a zero row to $B_i^{(0)}$ and maximizing $\mathbf b - \mathbf b'$ would
possibly increase $\lambda_{\omega}$ and invalidate some of the results below.
\begin{remark}\label{lambda-nu}
	By construction, $\lambda_{\omega} \le 2 \nu_{\omega}$.
\end{remark}

\begin{theorem}\label{cost-of-renorm} Assume that the tuple $\mathbf A$, 
	satisfies Hypothesis \ref{NDIH}.
	If $\mathbf f \cdot \Omega(\mathbf X,\mathbf y)=0$, $|X_j|<h$ with   
	\begin{equation} \label{h-bound}
		h= \frac{\lambda_{\omega}}{
			8 \nu_{\omega}
		\max_i\left( \sqrt{\#A_i} e^{ 2\max( \ell_i(\mathbf y), \ell_i(\mathbf -y) )}\right)}
	,
	\end{equation}
then
	\[
			\| D\mathbf Q(\mathbf X,\mathbf 0)^{-1}\|_{\omega}  
	 \le
	14 
	\sqrt{n}
	\lambda_{\omega}^{-1} \max_i \left(\sqrt{\#A_i}\, e^{2\ell_i(y)+2\ell_i(-y)}\right)
	\, \mu(\mathbf f, \Omega(\mathbf X, \mathbf y)
\]
with
\[
\ell_i(\mathbf y) \defeq 
\ell_i(\mathbf 0, \mathbf y)  = 
	\max_{(\mathbf b \, \mathbf c) \in A_i} \Re(\mathbf c \mathbf y).
\]
\end{theorem}

\subsection{Metric estimates}

The proof of Theorems~\ref{cost-of-renorm-legacy} and \ref{cost-of-renorm} relies on the following estimate, that dramatically improves what would be obtained by
Lemma \toricI{3.4.5}.

\begin{lemma}
If $\mathbf A$ satisfies Hypothesis \ref{NDH}, then
\begin{equation}\label{metric1}
	\frac{\finsler{\mathbf u}{\mathbf 0}}{\finsler{\mathbf u}{\mathbf z}}
		\le
		\frac{\sqrt{8\max \# A_i}}{\lambda_{\mathbf 0}} 
		e^{\max_i \ell_i(\mathbf z) + \ell_i(-\mathbf z)}
\end{equation}
	If furthermore $\mathbf A$ satisfies Hypothesis \ref{NDIH} and $h=\max |X_i|$
	satisfies \eqref{h-bound}, then
\begin{equation}\label{metric2}
	\frac{\finsler{\mathbf u}{\omega}}{\finsler{\mathbf u}{\Omega(\mathbf X,\mathbf y)}}
		\le
	14 \lambda_{\omega}^{-1}
	\max_i \left(\sqrt{\#A_i}\, e^{\ell_i(y)+\ell_i(-y)}\right)
	.
\end{equation}
\end{lemma}

\begin{proof}
	We establish \eqref{metric1} first. For all $i$,
\[
	\finsler{\mathbf u}{\mathbf z} \ge
	\|\mathbf u\|_{i,\mathbf z} = 
	\left\| D[v_{A_i}](\mathbf z) \mathbf u \right \|
\]
	Make $\hat {\mathbf v} = \frac{1}{\|v_{A_i}(\mathbf z)\|} v_{A_i}(\mathbf z)$. There are two reasonable ways to write down the derivative of 
$[v_{A_i}]$, viz.
\[
	D[v_{A_i}](\mathbf z) =
	(I-\hat {\mathbf v}\hat {\mathbf v}^*) \diag{\hat {\mathbf v}} A_i
	=
	\diag{\hat {\mathbf v}} \left( A_i -  \mathbf 1 m_{A_i}(\mathbf z)\right)
\]
	where $m_{A_i}(\mathbf z) =\hat{\mathbf v}^*\diag{\hat{\mathbf v}} A_i$
	is the {\em momentum map}. The second form yields
\begin{eqnarray*}
	\|\mathbf u\|_{i,\mathbf z} &=& 
	\| D[v_{A_i}](\mathbf z) \mathbf u \| 
	\\
	&\ge& 
	\| D[v_{A_i}](\mathbf z) \mathbf u \|_{\infty} 
	\\
	&\ge& 
	\min (|\hat v_{\mathbf a}|) \   
	\left\| \left( A_i -  \mathbf 1 m_{A_i}(\mathbf z)\right)\mathbf u\right\|_{\infty}
	\\
	&\ge&
	\min (|\hat v_{\mathbf a}|) \ 
	\left\| \left( A_i -  \mathbf 1 m_{A_i}(\mathbf z)\right)\Re(\mathbf u)\right\|_{\infty}
\end{eqnarray*}
	because $A_i$ and $m_{A_i}$ are real. 
	Specialize $i$ to the index where the maximum in \eqref{joint1} is attained
	for $\mathbf w = \Re(\mathbf u)$. Now,
\[
	\frac{\|\mathbf u\|_{i,\mathbf z}}{\min |\hat v_{\mathbf a}|}   
		\ge 
\left\| \left( A_i -  \mathbf 1 m_{A_i}(\mathbf z)\right)\mathbf w\right\|_{\infty}
=
	\max_{\mathbf a \in A_i} | (\mathbf a-m_{A_i}(\mathbf z)) \mathbf w|
	=
	| (\mathbf a''-m_{A_i}(\mathbf z)) \mathbf w|
\]
for some $\mathbf a'' \in A_i$.
Pick $\mathbf a$ and $\mathbf a' \in A_i$ so that $\mathbf a \mathbf w$ is maximal
	and $\mathbf a' \mathbf w$ is minimal. Since $m_{A_i}(\mathbf z)$ 
	is a convex linear
	combination of the rows of $A_i$, $m_{A_i}(\mathbf z) \mathbf w$ belongs to the
	segment $[ \mathbf a' \mathbf w, \mathbf a \mathbf w ] \subset \mathbb R$ and
	$\mathbf a''$ is either $\mathbf a$ or $\mathbf a'$. The triangular inequality
	guarantees that $|(\mathbf a''-m_{A_i}(\mathbf z)) \mathbf w| \ge (\mathbf a - \mathbf a')\mathbf w/2 \ge \lambda_{\mathbf 0} \|\mathbf w\|_{\mathbf 0} /2 $. 
	This shows that 
	$\frac{\finsler{\mathbf u}{\mathbf z}}{\min |\hat v_{\mathbf a}|}   
	\ge \frac{\lambda_{\mathbf 0}}{2} \finsler{\Re(\mathbf u)}{0}$. A similar argument shows that 
	$\frac{\finsler{\mathbf u}{\mathbf z}}{\min |\hat v_{\mathbf a}|}   
	\ge \frac{\lambda_{\mathbf 0}}{2} \finsler{\Im(\mathbf u)}{0}$. 
	Squaring and adding both bounds,
	\[
		2 
	\left(\frac{\finsler{\mathbf u}{\mathbf z}}{\min |\hat v_{\mathbf a}|}\right)   
		^2 \ge  
		\frac{\lambda_{\mathbf 0}^2}{4}(\finsler{\Re(\mathbf u)}{\mathbf 0}^2  + \finsler{\Im(\mathbf u)}{\mathbf 0}^2) \ge
		\frac{\lambda_{\mathbf 0}^2}{4} \max_i \|\mathbf u\|_{i,\mathbf 0}^2
		=
		\frac{\lambda_{\mathbf 0}^2}{4}\finsler{ \mathbf u}{\mathbf 0}^2
	\]
	Finally, $\min |\hat v_{\mathbf a}| \ge \frac{\min |v_{\mathbf a}|}{\sqrt{\#A_i}\max |v_{\mathbf a}|} \ge \frac{1}{\sqrt{\#A_i} (e^{\max_i \ell_i(\mathbf z) + \ell_i(\mathbf -z)}
)} $.

\medskip
\par
	In order to prove \eqref{metric2}, 
	we need the notations of \eqref{A-blocks}.
We also write 
$t=t(\mathbf X,\mathbf y)=m_i
\begin{pmatrix}	\diag{\mathbf X}^{-1} & \\ & I \end{pmatrix} \mathbf u
	\in \mathbb C$ where $m_i$ is the momentum at $\Omega_{A_i}(\mathbf X,\mathbf y)$,
	and 
\[
\Delta^{(j)} = \Delta^{(j)}(\mathbf X,\mathbf y) = 
\diag{\frac{\mathbf X^{B_i^{(j)}} \exp(C_i^{(j)}(\mathbf y))}{\| \Omega_{A_i}(\mathbf X,\mathbf y) \|}}. 
\]
For further use, we remark that the coefficients of $\Delta^{(j)}$ are all smaller
than $1$ in absolute value. The coefficients of the diagonal of
$\Delta^{(0)}$ are bounded below by
$\frac{1}{\sqrt{\# A_i}\,e^{\ell_i(\mathbf y)+\ell_i(-\mathbf y)}}$. The same bound
holds also for the diagonal of $\Delta^{(1)} B_i^{(1)}  \diag{\mathbf X}^{-1}$.

	The derivative of $[\Omega_{A_i}](\mathbf X,\mathbf y)$ decomposes as
\[
	D[\Omega_{A_i}](\mathbf X,\mathbf y) \mathbf u =  
{ \scriptsize	\begin{pmatrix}
	\Delta^{(0)} \\
		& \Delta^{(1)} \\
		& & \ddots
\end{pmatrix} }
	\left(
	\begin{pmatrix}
		\mathbf 0 & C_i^{(0)} \\
		B_i^{(1)}& C_i^{(1)} \\
		B_i^{(2)}& C_i^{(2)} \\
		\vdots & \vdots
	\end{pmatrix} 
\begin{pmatrix}	\diag{\mathbf X}^{-1} & \\ & I \end{pmatrix} \mathbf u
        - t   
	\begin{pmatrix}
		\mathbf 1\\
		\vdots \\
		\mathbf 1
	\end{pmatrix}
	\right)
\]
and hence the bound
	\begin{eqnarray*}
		\|	D[\Omega_{A_i}](\mathbf X,\mathbf y) \mathbf u\|_2 &\ge&  
		\left\|
{ \scriptsize	\begin{pmatrix}
	\Delta^{(0)} \\
		& \Delta^{(1)} \\
\end{pmatrix} 
	\left(
	\begin{pmatrix}
		\mathbf 0 & C_i^{(0)} \\
		B_i^{(1)}& C_i^{(1)} \\
	\end{pmatrix} 
\begin{pmatrix}	\diag{\mathbf X}^{-1} & \\ & I \end{pmatrix} \mathbf u
		-t 
	\begin{pmatrix}
		\mathbf 1\\
		\mathbf 1
	\end{pmatrix}
		\right)}
	\right\|_{\infty}
		\\
	&\ge&
		\left\|
{ \scriptsize	\begin{pmatrix}
	\Delta^{(0)} \\
		& \Delta^{(1)} \\
\end{pmatrix} 
	\left(
	\begin{pmatrix}
		\mathbf 0 & C_i^{(0)} \\
		B_i^{(1)}& \mathbf 0 \\
	\end{pmatrix} 
\begin{pmatrix}	\diag{X}^{-1} & \\ & I \end{pmatrix} \mathbf u
		-t 
	\begin{pmatrix}
		\mathbf 1\\
		\mathbf 0
	\end{pmatrix}
		\right)}
	\right\|_{\infty}
\\
		&&-
		\left\|
{ \scriptsize	\begin{pmatrix}
	\Delta^{(0)} \\
		& \Delta^{(1)} \\
\end{pmatrix} 
	\left(
	\begin{pmatrix}
            \mathbf 0  	& \mathbf 0 \\
	\mathbf 0 & C_i^{(1)} \\
	\end{pmatrix} 
\begin{pmatrix}	\diag{X}^{-1} & \\ & I \end{pmatrix} \mathbf u
		-t 
	\begin{pmatrix}
		\mathbf 0\\
		\mathbf 1
	\end{pmatrix}
		\right)}
	\right\|_{\infty} 
		\\
		&\ge&
		\frac{1}{\sqrt{\# A_i}\,e^{\ell_i(\mathbf y)+\ell_i(-\mathbf y)}}
		\left\|
	\begin{pmatrix}
		& C_i^{(0)} \\
		B_i^{(1)}&  \\
	\end{pmatrix} 
\mathbf u
		-t 
	\begin{pmatrix}
		\mathbf 1\\
		\mathbf 0
	\end{pmatrix}
	\right\|_{\infty}
\\
		&&-
		h 
		\left\|
	\begin{pmatrix}
	 0 & C_i^{(1)} \\
	\end{pmatrix} 
\mathbf u
		-t 
		\mathbf 1
	\right\|_{\infty} 
		.
	\end{eqnarray*} 
As in the proof of \eqref{metric1}, we fix $i$ so that 
the maximum in
Definition \eqref{JTI} is attained. Then,
$\finsler{\mathbf u}{\Omega(\mathbf X,\mathbf y)} \ge \| D[\Omega_{A_i}](\mathbf X,\mathbf y)\mathbf u\|$. 
As before,
\begin{equation}\label{lower-partial1}
\finsler{\mathbf u}{\Omega(\mathbf X,\mathbf y)} 
	\ge 
		\frac{\lambda_{\omega}}{\max_i \sqrt{8 \#A_i}\, e^{\ell_i(\mathbf y)+\ell_i(-\mathbf y)}}
	\finsler{\mathbf u}{\omega} - h \nu_{\omega} \finsler{\mathbf u}{\omega} - h |t| .
\end{equation}

Let  $s_i = \sqrt{\frac{\# A_i}{\#A_i^{(0)}}}$. 
We claim that 
\begin{equation}\label{momentumbound}
	|t(\mathbf X,\mathbf y)| \le (1+hs_i e^{2\ell_i(\mathbf y)}) \nu_{\omega} \finsler{\mathbf u}{\omega}. 
\end{equation}
and
\begin{equation}\label{momentumbound2}
	|t(\mathbf X,\mathbf 0)| \le 2hs_i \nu_{\omega} \finsler{\mathbf u}{\omega}. 
\end{equation}

Indeed, write $t=t'+t''$
\[
	t'= \frac{ \sum_{c \in C_i^{(0)}} e^{2 \mathbf c \Re(\mathbf y)} \mathbf c \mathbf u_2}
{\sum_{(\mathbf b,\mathbf c)\in A_i} |\mathbf X|^{2\mathbf b} e^{2\mathbf c\Re(\mathbf y)} }
\]
and
\[
t''=
\frac{ \sum_{(\mathbf b,\mathbf c) \in A_i^{(1)} \cup A_i^{(2)} \cup \dots} |\mathbf X|^{2\mathbf b} e^{2\mathbf c\Re(\mathbf y)} (\mathbf b \mathbf X^{-1} \mathbf u_1 + \mathbf c \mathbf u_2)}
{\sum_{(\mathbf b,\mathbf c)\in A_i} |\mathbf X|^{2\mathbf b} e^{2\mathbf c\Re(\mathbf y)} }
\]
The first term can be bounded by noticing that
\[
|t'| \le 
	\left| \frac{ \sum_{\mathbf c \in C_i^{(0)}} e^{2 \mathbf c \Re(\mathbf y)} \mathbf c \mathbf u_2}
{\sum_{\mathbf c\in C_i^{(0)}} e^{2\mathbf c\Re(\mathbf y)} }
	\right|
\]
and the expression inside the absolute value is a convex linear combination of the $\mathbf c \mathbf u_2$. 
If $\mathbf y = 0$, Definition \ref{normal-form}(c) implies that $t'=0$. Unconditionally,
$|t'| \le \nu_{\omega} \finsler{\mathbf u}{\omega}$. The second term
is more challenging. 

The denominator of $t''$ is not smaller than
\[
	\sum_{\mathbf c\in C_i^{(0)}} e^{2\mathbf c\Re(\mathbf y)} \ge \sqrt{\#A_i^{(0)}} 
	\sum_{\mathbf c\in C_i^{(0)}} \frac{ \exp(2\mathbf c\Re(\mathbf y))} {\sqrt{\#A_i^{(0)}}}
.
\]
Combining Jensen's inequality and Definition~\ref{normal-form}(c),
\[
\sum_{\mathbf c\in C_i^{0}} e^{2\mathbf c\Re(\mathbf y)} \ge 
\sqrt{\#A_i^{(0)}} \exp\left(
\sum_{\mathbf c\in C_i^{0}} 
\frac{ 2\mathbf c\Re(\mathbf y)
	}{\sqrt{\#A_i^{(0)}}}
	\right)
= \sqrt{\#A_i^{(0)}} \ .
\]
All the terms in the numerator of $t''$ are multiple of $|X_j| \le h < 1$, so
\[
	|t''| \le \frac{\sqrt {\#A_i}}{
		\sqrt{\#A_i^{(0)}}} h \max_{\mathbf b,\mathbf c} e^{2\mathbf c\Re(\mathbf y)} \left|\strut\mathbf b\, |\mathbf u_1| + \mathbf c \mathbf u_2\right|
	\le h s_i e^{2 \ell_i(\mathbf y)} \nu_{\omega} 
	\finsler{\strut(|\mathbf u_1|,\mathbf u_2)}{\omega}
\]
where the absolute value of $|\mathbf u_1|$ is meant coordinatewise.
Definition \ref{normal-form}(a) and Lemma \ref{lem-Li} imply that for all $i$,
$ \|(|\mathbf u_1|,\mathbf u_2)\|_{i,\omega} = \| (\mathbf u_1, \mathbf u_2)\|_{i,\omega}$.
The same holds therefore for the Finsler norm.
This establishes bounds \eqref{momentumbound} and \eqref{momentumbound2}.
\medskip
\par
Plugging \eqref{momentumbound} into \eqref{lower-partial1}, we recover
that
\begin{equation}\label{lower-partial2a}
\finsler{
	\mathbf u}{\Omega(\mathbf X,\mathbf y)} \ge 
\left( \frac{\lambda_{\omega}}{\sqrt{8\#A_i}\, e^{\ell_i(\mathbf y)+\ell_i(-\mathbf y)}}
	-h (2+hs_ie^{2\ell_i(\mathbf y)}) \nu_{\omega}\right) \finsler{\mathbf u}{\omega}.
\end{equation}
This expression simplifies after introducing \eqref{h-bound}.
Because of remark~\ref{lambda-nu}, $h s_i e^{\ell_i(\mathbf y)} \le \frac{\lambda_{\omega} s_i}{8 \nu_{\omega} \sqrt{\#A_i}}
\le \frac{1}{4\sqrt{\# A_i}} \le \frac{1}{4}$
\begin{equation}\label{lower-partial2}
\finsler{
	\mathbf u}{\Omega(\mathbf X,\mathbf y)}  \ge 
	\frac{ 
	\lambda_{\omega}}
{14\sqrt{\#A_i}\, e^{\ell_i(y)+\ell_i(-y)}}
\finsler{\mathbf u}{\omega}.
\end{equation}

\end{proof}

\subsection{The condition number bound}\label{proof-condition-bound}

\begin{proof}[Proof of Theorem~\ref{cost-of-renorm-legacy}]
	Since $	 V_{\mathbf A}(\mathbf z)$ is a root of $\mathbf f$, $\mathbf 0$ is a root 
	of $\mathbf f \cdot R(\mathbf 0, \mathbf y)$ and
	\[
		\mu(\mathbf f R(\mathbf z), \mathbf 0)^{-1} =
		\inf_{\mathbf u \ne 0}
		\frac{1}{\|\mathbf u\|_0}
		\left\|
		\diag{\frac{1}{\|f_i R_i(\mathbf z)\| \|V_{A_i}(\mathbf 0)\|}}
	\begin{pmatrix}
		\vdots\\
		f_i R_i(\mathbf z) DV_{A_i}(\mathbf 0) \mathbf u \\
		\vdots
	\end{pmatrix}
		\right\|
	.
\]
	Assume that the infimum is attained at a certain $\mathbf u_*$.
Then,
\begin{eqnarray*}
	\mu(
	\mathbf f , V_{\mathbf A}(\mathbf z))^{-1} 
		&=&
	\inf_{\mathbf w \ne 0}
	\frac{1}{\|\mathbf w\|_{\mathbf z}}
	\left\|
	\diag{\frac{1}{\|f_i \| \|V_{A_i}(\mathbf z)\|}}
	\begin{pmatrix}
		\vdots\\
		f_i DV_{A_i}(\mathbf z) \mathbf w \\
		\vdots
	\end{pmatrix}
		\right\|
		\\
&\le&
	\frac{1}{\|\mathbf u_*\|_{\mathbf z}}
	\left\|
		\diag{\frac{1}{\|f_i \| \|V_{A_i}(\mathbf z)\|}}
	\begin{pmatrix}
		\vdots\\
		f_i DV_{A_i}(\mathbf z) \mathbf u_* \\
		\vdots
	\end{pmatrix}
		\right\|
\end{eqnarray*}
We notice that
	\[
		f_i DV_{A_i}(\mathbf z)= f_i \diag{V_{A_i}(\mathbf z)} A_i = 
	(f_i R_i(\mathbf z) ) \diag{\mathbf 1} A_i =f_i R_i(\mathbf z) DV_{A_i}(\mathbf 0)
.\]
Hence,
\[
	\mu(\mathbf f, V_A(\mathbf z))^{-1} \le 
\frac {\|\mathbf u_*\|_{\mathbf 0}}
	{\|\mathbf u_*\|_{\mathbf z}}
		\max_i \left(\frac
	{\|f_i R_i(\mathbf z)\| \|V_{A_i}(\mathbf 0)\|}
	{\|f_i\| \|V_{A_i}(\mathbf z)\|}
	\right)
	\mu(\mathbf f R(\mathbf z), V_{\mathbf A}(\mathbf 0))^{-1} 
.\]
This is the same as
\[
	\mu(\mathbf f R(\mathbf z), V_{\mathbf A}(\mathbf 0))\le
\frac {\|\mathbf u_*\|_{\mathbf 0}}
	{\|\mathbf u_*\|_{\mathbf z}}
		\max_i \left(\frac
	{\|f_i R_i(\mathbf z)\| \|V_{A_i}(\mathbf 0)\|}
	{\|f_i\| \|V_{A_i}(\mathbf z)\|}
	\right)
	\mu(\mathbf f, V_A(\mathbf z)) .
\]
By the construction of $\ell_i$,
$\| f_i R_i(\mathbf z) \| \le e^{\ell_i(\mathbf z)} \|f_i\|$ and
$\| V_{A_i}(\mathbf 0) \| \le e^{\ell_i(-\mathbf z)} \|V_{A_i}(\mathbf z)\|$.

The ratio
\[
\frac {\|\mathbf u_*\|_{\mathbf 0}}
{\|\mathbf u_*\|_{\mathbf z}}
\le
\sqrt{n} \frac{\sqrt{8 \# A_i}}{\lambda_{\mathbf 0}}
\]
follows directly from \eqref{metric1}. The square root of $n$ goes away if we replace the usual operator
norm by a Finsler operator norm in the definition of the condition $\mu$.
\end{proof}

\begin{proof}[Proof of Theorem~\ref{cost-of-renorm}]
Since $	 \Omega(\mathbf X,\mathbf 0)$ is a root of $\mathbf f R(\mathbf 0, \mathbf y)$,
	\[
	\begin{split}
		\| D\mathbf Q(\mathbf X,\mathbf 0)^{-1}\|_{\omega}^{-1}  
		& =
\\
		\inf_{\mathbf u \ne 0}
		\frac{1}{\|\mathbf u\|_{\omega}}&
	\left\|
		\diag{\frac{1}{\|f_i R_i(\mathbf 0,\mathbf y)\| \|\omega_i\|}}
	\begin{pmatrix}
		\vdots\\
		f_i R_i(\mathbf 0,\mathbf y) D\Omega_{A_i}(\mathbf X,\mathbf 0) \mathbf u \\
		\vdots
	\end{pmatrix}
		\right\|
	\end{split}.
\]
	Assume that the infimum is attained at a certain $\mathbf u_*$.
Then,
\begin{eqnarray*}
	\mu(
	\mathbf f , \Omega(\mathbf X,\mathbf y))^{-1} 
		&=&
	\inf_{\mathbf w \ne 0}
	\frac{1}{\|\mathbf w\|_{\Omega(\mathbf X,\mathbf y)}}
	\left\|
	\diag{\frac{1}{\|f_i \| \|\Omega_{A_i}(\mathbf X,\mathbf y)\|}}
	\begin{pmatrix}
		\vdots\\
		f_i D\Omega_{A_i}(\mathbf X, \mathbf y) \mathbf w \\
		\vdots
	\end{pmatrix}
		\right\|
		\\
&\le&
	\frac{1}{\|\mathbf u_*\|_{\Omega(\mathbf X,\mathbf y)}}
	\left\|
		\diag{\frac{1}{\|f_i \| \|\Omega_{A_i}(\mathbf X,\mathbf y)\|}}
	\begin{pmatrix}
		\vdots\\
		f_i D\Omega_{A_i}(\mathbf X,\mathbf y) \mathbf u_* \\
		\vdots
	\end{pmatrix}
		\right\|
\end{eqnarray*}
As before, 
	\[
f_i D\Omega_{A_i}(\mathbf X,\mathbf y)=
\sum_{\mathbf a = (\mathbf b\, \mathbf c) \in A_i} 
	f_{i \mathbf a} X^{\mathbf b} e^{\mathbf c\mathbf y} (\mathbf b\, \diag{X_i^{-1}}, \mathbf c ) 
=
f_i R_i(\mathbf 0,\mathbf y) D\Omega_{A_i}(\mathbf X,\mathbf 0) 
.\]
Hence,
\[
\mu(\mathbf f, \Omega(\mathbf X,\mathbf y))^{-1} \le 
\frac {\|\mathbf u_*\|_{\omega}}
	{\|\mathbf u_*\|_{\Omega(\mathbf X,\mathbf y)}}
		\max_i \left(\frac
	{\|f_i R_i(\mathbf 0,\mathbf y)\| \|\omega_i\|}
	{\|f_i\| \|\Omega_{A_i}(\mathbf X,\mathbf y)\|}
	\right)
		\| D\mathbf Q(\mathbf X,\mathbf 0)^{-1}\|_{\omega}^{-1}  
.\]
This is the same as
\[
\| D\mathbf Q(\mathbf X,\mathbf 0)^{-1}\|_{\omega}  
\le 
\frac {\|\mathbf u_*\|_{\omega}}
	{\|\mathbf u_*\|_{\Omega(\mathbf X,\mathbf y)}}
\max_i \left(\frac
	{\|f_i R_i(\mathbf y)\| \|\omega_i\|}
	{\|f_i\| \|\Omega_{A_i}(\mathbf X,\mathbf y)\|}
	\right)
\mu(\mathbf f, \Omega(\mathbf X,\mathbf y))
.
\]
By definition, $\| f_i R_i(\mathbf 0,\mathbf y) \| \le e^{\ell_i(\mathbf y)} \|f_i\|$. For all $\mathbf a = (\mathbf b\, \mathbf c) \in A_i$, 
\[
	|\mathbf X^{\mathbf b}| = |\mathbf X^{\mathbf b} e^{\mathbf c\mathbf y} e^{-\mathbf c\mathbf y}| \le |\mathbf X^{\mathbf b} e^{\mathbf c} e^{\ell_i(-\mathbf y)}|
\]
so
\[
\max{\frac
	{\|f_i R_i(\mathbf y)\| \|\Omega_{A_i}(\mathbf X,\mathbf 0)\|}
	{\|f_i\| \|\Omega_{A_i}(\mathbf X,\mathbf y)\|}
	}
\le e^{\max \ell_i(\mathbf y)+\ell_i(-\mathbf y)}.
\]
The ratio
\[
\frac {\|\mathbf u_*\|_{\omega}}
{\|\mathbf u_*\|_{\Omega(\mathbf X,\mathbf y)}}
\le
\sqrt{n} 
\frac {\finsler{\mathbf u_*}{\omega}}
{\finsler{\mathbf u_*}{\Omega(\mathbf X,\mathbf y)}}
\]
is bounded by \eqref{metric2}.
\end{proof}

\section{Higher Derivative Estimate at toric infinity}
\label{sec:higher}

The analysis of Newton iteration based homotopy algorithms requires a
criterion of quadratic convergence for the iteration, such as the
$\beta \gamma \le \alpha_0$ criterion from Smale's alpha-theory. As explained
by \ocite{BCSS}*{Section 14.2}, this criterion can be checked by 
bounding the gamma invariant in terms of the condition number. 
This idea was generalized to the realm of sparse polynomials  
in Theorem~\toricI{3.6.1}. Unfortunately, the bound in that
theorem depends from the $\nu$ invariant. Therefore, it is not sharp
 near toric infinity. This section will provide a criterion for
 quadratic convergence near toric infinity.

Assume that the tuple $\mathbf A=(A_1, \dots, A_n)$ satisfies Hypothesis~\ref{NDIH}, in particular it is in normal form.
In the previous section, we associated to each $\mathbf {\bar y}$
a `local' map at toric infinity, viz.
\begin{equation}\label{ersatz}
	\defun{\mathbf Q}{T_{\boldsymbol \omega}\mathscr V_{\mathbf A}}{\mathbb C^n}
	{(\mathbf X,\mathbf y)}{
\begin{pmatrix}
	\frac{1}{\|\omega_1\|\|f_1 \cdot R_1(\mathbf 0, \mathbf{\bar y})\|} f_1  \cdot R_1(\mathbf 0, \mathbf{\bar y})\cdot \Omega_{A_1}(\mathbf X,\mathbf y)
\\
\vdots \\
	\frac{1}{\|\omega_n\|\|f_n \cdot R_n(\mathbf 0, \mathbf{\bar y})\|} f_n \cdot R_n(\mathbf 0, \mathbf{\bar y}) \cdot \Omega_{A_n}(\mathbf X,\mathbf y)
\end{pmatrix}
	.}
\end{equation}
To simplify notations, we write $\mathbf q=\mathbf f \cdot R(\mathbf 0, \mathbf{\bar y})$
so that
\[
	\mathbf Q(\mathbf X,\mathbf y)= \begin{pmatrix}
	\frac{1}{\|\omega_1\|\|q_1\|} q_1 \Omega_{A_1}(\mathbf X,\mathbf y)
\\
\vdots \\
	\frac{1}{\|\omega_n\|\|q_n\|} q_n \Omega_{A_n}(\mathbf X,\mathbf y)
\end{pmatrix}
\]

This map vanishes at $(\mathbf X, \mathbf y)$ if and only if $\Omega(\mathbf X, \mathbf{\bar y}+\mathbf y)$
is a `toric' zero for the system $\mathbf f$.
Newton iteration of $\mathbf Q$ is defined as usual by
\begin{equation}\label{Newton-ersatz}
	\defun{N(\mathbf Q, \cdot)}
	{T_{\boldsymbol \omega}\mathscr V_{\mathbf A}}
	{T_{\boldsymbol \omega}\mathscr V_{\mathbf A}}
	{(\mathbf X,\mathbf y)}{
		N( \mathbf Q, (\mathbf X, \mathbf y)) 
		= (\mathbf X, \mathbf y) - D\mathbf Q(\mathbf X,\mathbf y)^{-1} \mathbf Q(\mathbf X, \mathbf y).}
\end{equation}

A sufficient condition for the quadratic convergence of the
Newton iterates 
\begin{equation}\label{Newton-infinity}
\left\{
	\begin{array}{lcl}
		(\mathbf X_{t+1}, \mathbf y_{t+1}) &=& N( \mathbf Q, (\mathbf X_t, \mathbf y_t)) 
		,\hspace{2em} t=0, 1, \dots
	\\
		(\mathbf X_0, \mathbf y_0)&=&(\mathbf {\bar X},\mathbf 0), 
	\end{array} \right.
\end{equation}
will be given below. It implies that the sequence converges 
to a zero $(\tilde {\mathbf X},\tilde {\mathbf y}) \in T_{\boldsymbol \omega}\mathscr V_{\mathbf A}$
for $\mathbf Q$, and hence $\Omega_{\mathbf A}(\tilde{\mathbf X},\bar {\mathbf y}+\tilde{\mathbf y})$ is a toric zero
for $\mathbf f$.
The criterion follows from 
Smale's alpha theory applied to the local map $\mathbf Q$. All metric estimates, including the
definition of the $\gamma$ invariant, assume the Hermitian (resp. Finsler) metric of
$T_{\boldsymbol \omega}\mathscr V_{\mathbf A}$. In coordinates,
\[
	\langle \mathbf u, \mathbf w \rangle_{\boldsymbol \omega} = 
	\left(D\boldsymbol \Omega_{\mathbf A}(\mathbf 0,\mathbf 0) 
	\mathbf w\right)^* 
	\diag{\|\omega_i\|^{-2}}
	D\boldsymbol \Omega_{\mathbf A}
	(\mathbf 0,\mathbf 0) \mathbf u
\]
and
\[
	\finsler{\mathbf u}{\omega} = \max \left \| \frac{1}{\|\omega_i\|} D\Omega_{A_i}(0,0) \mathbf u \right \|
\]
\begin{remark}\label{finsler}
All results in this section assume the Hermitian metric of $T_{\boldsymbol \omega}\mathscr V_{\mathbf A}$. Using the
Finsler metric produces sharper results, but I know no efficient way to compute operator norms with respect
to a Finsler structure. 
\end{remark}
Recall the notations $s_i = \sqrt{\#A_i / \#A_i^{(0)}}$.
\begin{theorem}[Higher derivative estimate]\label{th-higher}
	Assume that the tuple $\mathbf A$ satisfies Hypothesis~\ref{NDIH} with splitting $\mathbb R^l \times \mathbb R^{n-l}$.
	Let $\mathbf f$, $\mathbf{\bar y}$ and $\mathbf Q$ be defined
	as above. If $|X_k|<h<1$ for $1\le l \le k$, then
\[
\gamma(\mathbf Q,(\mathbf X,\mathbf 0)) \le
	\| D\mathbf Q(\mathbf X,\mathbf 0)^{-1} \|_{\omega} 
	\frac{\nu_{\mathbf \omega} \sqrt{\sum_i s_i^2}  }{(1-h)^3}
.
\]
\end{theorem}

Smale's alpha-theorem yields:

\begin{corollary}\label{cor-quadratic}
	Let $0<h<1$ and 
	\begin{equation}\label{c-star}
	c_* \ge  
	\frac{\nu_{\mathbf \omega} \sqrt{\sum_i s_i^2}  }{(1-h)^3}.
\end{equation}
	Define $\beta_t = \| (\mathbf X_{t+1}, \mathbf y_{t+1}) - (\mathbf X_t, \mathbf y_t) \|_{\omega}$
	and
	\[
	\bar \mu
	= 
	\|D\mathbf Q(\mathbf X,\mathbf 0)^{-1} \|_{\omega}
	\,
	\]
Let \[
\alpha 
\le
\alpha_0 = \frac{13 - 3 \sqrt{17}}{4},
\]
\[
	r_0 = \changed{r_0(\alpha)=}
\frac{1+\alpha-\sqrt{1-6\alpha+\alpha^2}}{4\alpha} 
\text{ and }
	r_1 =\changed{r_1(\alpha)=}
\frac{1-3\alpha-\sqrt{1-6\alpha+\alpha^2}}{4\alpha} 
.
\]
	If $c_* \beta_0 \bar \mu \le \alpha$, then the sequence $(\mathbf X_y, \mathbf y_t)$ from 
	\eqref{Newton-infinity}
	is well-defined and converges to a zero $(\tilde{\mathbf X}, \tilde{\mathbf y})$ of $\mathbf Q$.
	Moreover,
\begin{enumerate}[(a)]
\item
$
		\| (\mathbf X_t, \mathbf y_t) - (\tilde{\mathbf X}, \mathbf {\tilde y}) \|_{\omega}
	\le 2^{-2^i+1} \beta_0
$
\item 
$
		\| (\tilde{\mathbf X}, \tilde{\mathbf y})-(\mathbf{\bar X}, \mathbf 0) \|_{\omega} \le r_0 \beta_0
$
\item
$
		\| (\tilde{\mathbf X}, \tilde{\mathbf y}) - (\mathbf X_1,\mathbf y_1)\|_{\omega} \le r_1 \beta_0
$
\end{enumerate}
\end{corollary}
It follows from the construction of $\mathbf Q$ that $\boldsymbol \Omega_{\mathbf A}(\tilde{\mathbf X}, \mathbf{\bar y}+\tilde{\mathbf{y}})$
is a toric zero of $\mathbf f$.

\subsection{The $p$-th derivative of $\Omega_{A_i}$}

Towards the proof of Theorem~\ref{th-higher},
we want to estimate 
$\frac{1}{p!}\|D^p\boldsymbol \Omega_{\mathbf A}(\mathbf 0,\mathbf 0)\|$, 
and then $\frac{1}{p!}\|D^p\boldsymbol \Omega_{\mathbf A}(\mathbf X,\mathbf 0)\|$.

\begin{lemma}\label{coord}
	If $\finsler{\mathbf u}{\omega} \le 1$ and $j \le l$, then for all $i$ with $u_j \in B_i^{(1)}$,  
	$|u_j| \le \nu_{i,\omega}$.
	If Hypothesis \ref{NDIH} holds, then
	$|u_j| \le \min \nu_{i,\omega}$ where the minimum is taken over
	all $i$ with $\mathrm e_j \in B_i^{(1)}$ and is always attained, so 
	$|u_j| \le \nu_{\omega}$.
\end{lemma}
\begin{proof}
 	Write $\mathbf u = \begin{pmatrix}\mathbf u'\\ \mathbf u''\end{pmatrix}$ with $\mathbf u' \in \mathbb R^l$. Because of
		\eqref{Li}, 
 	\[
		\finsler{\begin{pmatrix}\mathbf u'\\\mathbf 0\end{pmatrix}}{\omega} \le 1. 
 	\]
	Suppose that 
	$\mathrm e_j \in B_i^{(1)}$. There is $\mathbf c \in C_i^{(1)}$ be such that
	$(\mathrm e_j \ \mathbf c) \in A_i^{(1)}$,
 \[
 	|u_j| = \left|(\mathrm e_j \ \mathbf c)  \begin{pmatrix}\mathbf u'\\0\end{pmatrix} \right| 
 \le
 \nu_{i,\omega}	.
 \]
	Hypothesis \ref{NDIH} implies that the matrix $L$ from \eqref{Li} has full rank. Thus, there is at least one pair $(i,j)$ 
	with $\mathrm e_j \in B_i^{(1)}$.
\end{proof}
\begin{lemma} \label{coord-c}
	Suppose that $(\mathbf b\, \mathbf c) \in A_i$ for some $i$. Then
	\[
		|(\mathbf 0,\mathbf c) \mathbf u| \le \nu_{i,\omega} \|\mathbf u\|_{i,\omega} 
		\le \nu_{\omega} \finsler{\mathbf u}{\omega}
.	\]

\end{lemma}
\begin{proof}
 	As before, write $\mathbf u = \begin{pmatrix}\mathbf u'\\ \mathbf u''\end{pmatrix}$ with $\mathbf u' \in \mathbb R^k$. From the
 block structure of each $L_i$, 
 	\[
		\left \|\begin{pmatrix}\mathbf 0\\\mathbf u''\end{pmatrix}
			\right\|
			_{i,\omega}
			\le \|\mathbf u\|_{i,\omega}. 
 	\]
	Then $|(\mathbf 0,\mathbf c) \mathbf u| = |(\mathbf b,\mathbf c) \mathbf u| \le \nu_{i,\omega} \|\mathbf u\|_{i,\omega}$.
\end{proof}
\begin{lemma}\label{lemma-high-1} Assume Hypothesis \ref{NDIH}.
	Let $p \ge 2$. 
	If $\finsler{\mathbf u_1}{\omega}, \dots, \finsler{\mathbf u_p}{\omega} \le 1$, then
\[
 	\frac{1}{p!} \left\| D^p\Omega_{A_i}(\mathbf 0,\mathbf 0) (\mathbf u_1, \dots, \mathbf u_p) \right\|
 	\le \nu_{\omega}^{p-1} 
	\sqrt{\# A_i}.
\]
If one further assumes that for each $1\le i \le l$, $\mathbf u_i = \begin{pmatrix}
	\mathbf u_i' \\ \mathbf 0
\end{pmatrix}$ with each coordinate of $\mathbf u_i'$ of absolute value less than $h$, 
then 
\[
\frac{1}{p!} \left\| D^p\Omega_{A_i}(\mathbf 0,\mathbf 0) (\mathbf u_1, \dots, \mathbf u_p) \right\|
\le
	h^l \sqrt{\# A_i} \nu_{\omega}^{p-l}
.
\]
\end{lemma}

\begin{proof} We prove the general case first.
First assume that $j_1 \le j_2 \le \dots \le j_n$ is
	given, and let $(\mathbf b, \mathbf c)$ satisfy: 
	\begin{equation}\label{bcj}
	\mathbf X^{\mathbf b} e^{\mathbf c\mathbf y}
=
	X_{j_1} X_{j_2} \dots X_{j_k} 
	e^{y_{j_{k+1}} + \dots + y_{j_n}} 
.
\end{equation}
If $p<|\mathbf b|$, 
\[
	\frac{1}{p!} D^p_{|\mathbf X=\mathbf 0, \mathbf y=\mathbf 0} \mathbf X^{\mathbf b} e^{\mathbf c\mathbf y} (\mathbf u_1, \dots, \mathbf u_p) =
0
\]
If $p\ge|\mathbf b|$, we can write
\[
\begin{split}
	\frac{1}{p!} D^p_{|\mathbf X=\mathbf 0, \mathbf y=\mathbf 0} & \mathbf X^{\mathbf b} e^{\mathbf c\mathbf y} (\mathbf u_1, \dots, \mathbf u_p) =
\\
	=&	\frac{1}{p!}\sum_{\sigma \in S_p} 
	(\mathbf u_{\sigma(1)})_{j_1}
\dots
(\mathbf u_{\sigma(k)})_{j_k}
	(\mathbf 0\  \mathbf c)\mathbf u_{\sigma(k+1)}
\dots
	(\mathbf 0\ \mathbf c)\mathbf u_{\sigma(p)}
.
\end{split}
\]
By triangular inequality
	\begin{equation}\label{Dp1}\begin{split}
	\frac{1}{p!} \left\| D^p\Omega_{A_i}(\mathbf 0,\mathbf 0) (\mathbf u_1, \dots, \mathbf u_p) \right\|
	\le \hspace{-12em} & 
\\
	&\le 
\frac{1}{p!}
\sum_{\sigma \in S_p}
\left\|
\begin{pmatrix}
	\vdots \\
	(\mathbf u_{\sigma(1)})_{j_1}
\dots
(\mathbf u_{\sigma(k)})_{j_k}
	(\mathbf 0\  \mathbf c)\mathbf u_{\sigma(k+1)}
\dots
	(\mathbf 0\ \mathbf c)\mathbf u_{\sigma(p)}
\\
	\vdots
\end{pmatrix}_{(\mathbf b\, \mathbf c) \in A_i}
\right\|
	\end{split}
\end{equation}
        Lemma~\ref{coord} bounds $|u_j|\le \nu_{\omega}$, $j \le l$, while
	Lemma~\ref{coord-c} bounds $|(\mathbf 0,\mathbf c) \mathbf u| \le \nu_{i,\omega}$.
	Hence,

	\begin{eqnarray*}
\frac{1}{p!} \left\| D^p\Omega_{A_i}(\mathbf 0,\mathbf 0) (\mathbf u_1, \dots, \mathbf u_p) \right\|
	\le \hspace{-12em} & &
\\
&\le &
\frac{1}{p!}
\sum_{\sigma \in S_p}
\left\|
\begin{matrix}
\begin{pmatrix}\vdots \\
	\nu_{i,\omega}^{p-1}
	(\mathbf 0\ \mathbf c) \mathbf u_{\sigma(1)}\\
	\vdots
\end{pmatrix}_{(\mathbf 0,\mathbf c) \in A_i^{(0)}}
\\
\begin{pmatrix}\vdots \\
	\nu_{i,\omega}^{p-1}
	(u_{\sigma(1)})_{j_1(\mathbf b)}
	\\ \vdots
\end{pmatrix}_{(\mathbf b\, \mathbf c) \in A_i^{(1)}}
\\
	\begin{pmatrix}\vdots \\
	\nu_{i,\omega}^{p-2}
	\nu_{\omega}
		(u_{\sigma(1)})_{j_1(\mathbf b)}
	\\ \vdots
	\end{pmatrix}_{(\mathbf b,\mathbf c) \in A_i^{(2)}}
\\
\vdots
\end{matrix}
\right\|.
\end{eqnarray*}
All this simplifies to
\begin{eqnarray*}	
\frac{1}{p!} \left\| D^p\Omega_{A_i}(\mathbf 0,\mathbf 0) (\mathbf u_1, \dots, \mathbf u_p) \right\|
	\le \hspace{-12em} & &
\\
	&\le& 
	\frac{\nu_{\omega}^{p-1}}
	{p!}
\sum_{\sigma \in S_p}
\left\|
\begin{matrix}
	A_i^{(0)} \mathbf u_{\sigma(1)} 
\\
	A_i^{(1)} \mathbf u_{\sigma(1)} 
\\
	\nu_{\omega}^{-1} 
	\begin{pmatrix}\vdots \\
		(u_{\sigma(1)})_{j_1(\mathbf b)} \\ \vdots
\end{pmatrix}_{(\mathbf b\,\mathbf c) \in A_i(2)}
\\
\vdots
\end{matrix}
\right\|.
\end{eqnarray*}
According to Lemma~\ref{lem-Li}, the top two blocks are 
\[
	\sqrt{\#A_i} L_i \mathbf u_{\sigma(1)} =
\sqrt{\#A_i} D[\Omega_{A_i}(\mathbf 0,\mathbf 0)] \mathbf u_{\sigma(1)}
\]
with total norm $\le \sqrt{\# A_i^{(0)}}$. 
From Lemma~\ref{coord},
the other coordinates are bounded by $1$,
and there are $\#A_i-\#A_i^{(0)}$ of them. 
Hence,
\[
\frac{1}{p!} \left\| D^p\Omega_{A_i}(\mathbf 0,\mathbf 0) (\mathbf u_1, \dots, \mathbf u_p) \right\|
\le
\nu_{\omega}^{p-1}
\sqrt{\# A_i}
.
\]
\medskip
\par
We prove now the special case $\mathbf u = \begin{pmatrix}
	\mathbf u_i' \\ \mathbf 0
\end{pmatrix}$ with $\| u_i'\|_{\infty} \le h$ for $i \le l$.
Under the notation in \eqref{bcj},
Equation \eqref{Dp1} becomes
\begin{equation}\label{Dp2}\begin{split}
	\frac{1}{p!} \left\| D^p\Omega_{A_i}(\mathbf 0,\mathbf 0) (\mathbf u_1, \dots, \mathbf u_p) \right\|
	\le \hspace{-12em} & 
\\
	&\le 
\frac{1}{p!}
\sum_{\sigma \in S}
\left\|
\begin{pmatrix}
	\vdots \\
	(\mathbf u_{\sigma(1)})_{j_1}
\dots
(\mathbf u_{\sigma(k)})_{j_k}
	(\mathbf 0\  \mathbf c)\mathbf u_{\sigma(k+1)}
\dots
	(\mathbf 0\ \mathbf c)\mathbf u_{\sigma(p)}
\\
	\vdots
\end{pmatrix}_{(\mathbf b\, \mathbf c) \in A_i}
\right\|
	\end{split}
\end{equation}
where the sum ranges over the subset $S \subset S_p$ of permutations
with $\sigma(1), \dots, \sigma(l) \le k$. There are 
$ \le p!$ such permutations.
The rows indexed by $(\mathbf b, \mathbf c)$ in submatrices $A_i^{0}$ to $A_i^{l-1}$ all
vanish. The remaining rows in the matrix can be bounded in absolute value by
\[
	\nu_{\omega}^{k-l}
	\nu_{\omega_i}^{p-k} 
	\le
	h^l 
	\nu_{\omega}^{p-l}
\]
and hence 
\[
\frac{1}{p!} \left\| D^p\Omega_{A_i}(\mathbf 0,\mathbf 0) (\mathbf u_1, \dots, \mathbf u_p) \right\|
\le
h^l \sqrt{\# A_i}\nu_{\omega}^{p-l}
.
\]
\end{proof}

\begin{lemma}\label{lemma-high-2}
	Assume that $|X_1|, \dots, |X_l| \le h$. Then,
	\[
		\frac{1}{p!} \left\| D^p\Omega_{A_i}(\mathbf X, \mathbf 0)\right\|_{\omega}
		\le
		\frac{\sqrt{\#A_i} \nu_{\omega}^p}
		{(1-h)^{p+1}}
	\]
\end{lemma}

\begin{proof}
	Assume that $\|\mathbf u_1\|_{i,\omega}, \dots, \|\mathbf u_p\|_{i,\omega}
	\le 1$.
Expand
\[
	\frac{1}{p!}D^p\Omega_{A_i}(\mathbf X, \mathbf 0)(\mathbf u_1, \dots, \mathbf u_p)=
	\sum_{q \ge 0}
	\frac{1}{p!q!}
	D^{p+q}\Omega_{A_i}(\mathbf 0, \mathbf 0)
	\left(
	\begin{pmatrix} 
	\mathbf X\\\mathbf 0 \end{pmatrix}^q,
	\mathbf u_1, \dots, \mathbf u_p
	\right)  
\]
	The initial term $q=0$ is bounded above by $\sqrt{\#A_i} \nu_{\omega}^p$
	and subsequent terms by 
	$\binomial{p+q}{q} h^q \sqrt{\#A_i} \nu_{\omega}^{p}$.
Thus,
	\begin{eqnarray*}
	\frac{1}{p!}
	\left\| 
D^p\Omega_{A_i}(\mathbf X, \mathbf 0)(\mathbf u_1, \dots, \mathbf u_p)
		\right\| &\le&
		\sqrt{\#A_i} 
		\nu_{\omega}^p
		\left( 1 + \binomial{p+1}{1}h + \binomial{p+2}{2}h^2 + \dots \right)
\\
		&=&
		\frac{\sqrt{\#A_i} \nu_{\omega}^p}
		{(1-h)^{p+1}}
		\end{eqnarray*}

\end{proof}

\subsection{Proof of the Higher Derivative Estimate}

\begin{proof}[Proof of Theorem~\ref{th-higher}]
	We have to prove that
	\begin{equation}\label{gamma-part}
	\gamma (\mathbf Q; \mathbf X, \mathbf 0) 
\le 
	\| D\mathbf Q(\mathbf X,\mathbf 0)^{-1}\|_{\omega}
	\frac{\nu_{\mathbf \omega}^2}{(1-h)^3}
	\sqrt{\sum_i s_i^2}  
	\end{equation}
	The case $
	\| D\mathbf Q(\mathbf X,\mathbf 0)^{-1}\|_{\omega}=
	\infty$ is trivial, so we assume
	$
	\| D\mathbf Q(\mathbf X,\mathbf 0)^{-1}\|_{\omega}
	< \infty$.
	As before, suppose that $\|\mathbf u_1\|_{\omega}, \dots, \|\mathbf u_p\|_{\omega} \le 1$. Lemma~\ref{lemma-high-2} yields:
\[
	\frac{1}{p!} \left\| D^p\Omega_{A_i}(\mathbf X,\mathbf 0) (\mathbf u_1, \dots, \mathbf u_p) \right\|
	\le \frac{\nu_\omega^{p-1} 
	\sqrt{\# A_i}}{(1-h)^{p+1}}
\]
	Let $\Delta = \diag{\|\omega_i\| \|\mathbf q_i\|}$.
	Since $\| \omega_i \|=\sqrt{\#A_i^{(0)}}$, 
\[
	\frac{1}{p!} \left\| \Delta^{-1} \begin{pmatrix}\vdots \\ \mathbf q_i D^p\Omega_{A_i}(\mathbf X,\mathbf 0)\\ \vdots\end{pmatrix} (\mathbf u_1, \dots, \mathbf u_p) \right\|
		\le 
	\frac{\nu_\omega^{p-1} \sqrt{\sum s_i^2}}{(1-h)^{p+1}} 
\]
Therefore, 
	\[
	\frac{1}{p!} \left\|D\mathbf Q(\mathbf X, \mathbf 0)^{-1} 
		\Delta^{-1}
	\begin{pmatrix}\vdots \\ \mathbf q_i D^p\Omega_{A_i}\\ \vdots\end{pmatrix} (\mathbf u_1, \dots, \mathbf u_p) \right\|_{\omega}
		\le  
	\| D\mathbf Q(\mathbf X,\mathbf 0)^{-1}\|_{\omega}
\frac{\nu_{\omega}^{p}
\sqrt{\sum_i s_i^2}  
	}{(1-h)^{p+1}} 
		\]
	and we would like to take the $p-1$-th root of the right-hand-side, which factors as
\[
		\left(\frac{\nu_{\omega}}{1-h} \right)^{p-1}
		\left(\| D\mathbf Q(\mathbf X,\mathbf 0)^{-1}\|_{\omega}
\frac{\sqrt{\sum_i s_i^2}  
		}{(1-h)^{2}}\right) 
\]
	We claim that the expression under parentheses is $\ge 1$. 
By construction, for any $\mathbf u$ with $\|\mathbf u\|_{\omega}=1$ we have
\[
	\|D\mathbf Q(\mathbf X,\mathbf 0)\mathbf u\| \le \|\mathbf u\|_{\omega} \max_i \frac{\|\Omega_{A_i}(X,0)\|}
	{\| \omega_i\|} \le 1+h \max s_i.
\]
	It follows that 
		\[
\| D\mathbf Q(\mathbf X,\mathbf 0)^{-1}\|_{\omega}
(1+h\max s_i)   
\ge 1
\]
and
\[
		\| D\mathbf Q(\mathbf X,\mathbf 0)^{-1}\|_{\omega}
\frac{\sqrt{\sum_i s_i^2}  
		}{(1-h)^{2}} 
\ge
\frac{\sqrt{\sum_i s_i^2}  
}{(1-h)^{2}(1+h \max s_i)} 
\ge
\frac{1}{(1-h)^2} \ge 1
.
\]
We deduce that
\[
	\gamma (\mathbf Q; \mathbf X,\mathbf 0) \le 
        \| D\mathbf Q(\mathbf X,\mathbf 0)^{-1}\|_{\omega}
	\frac{\nu_{\omega}
\sqrt{\sum_i s_i^2}  
	}{(1-h)^3}
	.
\]
\end{proof}

\section{Homotopy near infinity}
\label{sec:homotopy}

We keep the assumption that the tuple $\mathbf A$ satisfies Hypothesis~\ref{NDIH}, with splitting 
$\mathbb R^n = \mathbb R^l \times \mathbb R^{n-l}$ with $l \ge 1$.
Through this section we consider an homotopy path in the solution variety
of the form $(\mathbf g_t, \Omega_{\mathbf A}(\tilde {\mathbf X}_t, \tilde {\mathbf y}_t))$,
so that $\mathbf g_t \cdot \Omega_{\mathbf A}(\tilde {\mathbf X}_t, \tilde {\mathbf y}_t) \equiv \mathbf 0$.
We will assume that the {\em partially renormalized} condition metric integral
$\renlength{l}=	\renlength{l}(\mathbf g_t
\cdot 
\mathbf R(\mathbf 0, \tilde {\mathbf y}_t)
, \Omega_{\mathbf A}(\tilde {\mathbf X}_t, \mathbf 0
); 0,T)$ is finite,
\[
	\renlength{l} =
	\int_{0}^{T}
\left(\left\| \frac{\partial}{\partial t} \mathbf g_t 
\cdot 
\mathbf R(\mathbf 0, \tilde {\mathbf y}_t)
\right\|_{\mathbf g_t} 
+\left\| \frac{\partial}{\partial t} \Omega_{\mathbf A}(\tilde {\mathbf X}_t, 
\mathbf 0 
)
\right\|_{\boldsymbol \omega}
\right) \mu(\mathbf g_t, \Omega_{\mathbf A}(\tilde {\mathbf X}_t, \tilde {\mathbf y}_t))
\dd t .
\]
Above, the first norm in the integrand is the product Fubini-Study metric in 
multiprojective space 
$\MP(\mathscr P_{\mathbf A})= \mathbb P(\mathscr P_{A_1}) \times \cdots \times \mathbb P(\mathscr P(A_n))$.
The second norm is the norm in $T_{\boldsymbol \omega}\mathscr V_{\mathbf A}$ induced by the
embeding of 
$\mathscr V_{\mathbf A}$ in multiprojective
space $\mathbb P(\mathbb C^{A_1}) \times \dots \times \mathbb P(\mathbb C^{A_n})$.

\subsection{Partially renormalized homotopy}

The objective of this section is to produce a viable numerical method for path-following near toric infinity.
We will approximate the neighborhood of toric infinity $\boldsymbol \omega$ up to first order by
the tangent plane $T_{\boldsymbol \omega} \mathscr V_{\mathbf A}$. The systems of equations will be
approximated as in the previous sections, but this time by a family of partially renormalized systems
\begin{equation}\label{ersatz2}
	\defun{\mathbf Q = \mathbf Q_{t, \bar y}}{T_{\boldsymbol \omega}\mathscr V_{\mathbf A}}{\mathbb C^n}
	{(\mathbf X,\mathbf y)}{
\begin{pmatrix}
	\frac{1}{\|\omega_1\|\|(g_{t})_{1} \cdot R_1(\mathbf 0, \mathbf{\bar y})\|} (g_{t})_1  \cdot R_1(\mathbf 0, \mathbf{\bar y})\cdot \Omega_{A_1}(\mathbf X,\mathbf y)
\\
\vdots \\
	\frac{1}{\|\omega_n\|\|(g_{t})_{n} \cdot R_n(\mathbf 0, \mathbf{\bar y})\|} (g_{t})_n \cdot R_n(\mathbf 0, \mathbf{\bar y}) \cdot \Omega_{A_n}(\mathbf X,\mathbf y)
\end{pmatrix}
	.}
\end{equation}
The main difference is that $g_t$ depends on a parameter $t$. The values of $\mathbf{\bar y}$ will be computed
during  the approximation procedure described below. Once $\bar{\mathbf y}$ is fixed,
the partially renormalized Newton operator for $\mathbf Q_{t, \bar{\mathbf y}}$ is
\[
	N (\mathbf Q_{t, \bar {\mathbf y}}; \mathbf X,\mathbf y) = 
	(\mathbf X,\mathbf y) - D\mathbf Q_{t,\bar{\mathbf y}}(\mathbf X,\mathbf y)^{-1} 
	\mathbf Q_{t,\bar {\mathbf y}}(\mathbf X,\mathbf y).
\]
We will denote by $N^s(\mathbf Q_{t, \bar {\mathbf y}}; \cdot)$ the composition
of $s$ steps of Newton iteration, that is
\[
	N^{s+1}(\mathbf Q_{t, \bar {\mathbf y}}; \mathbf X,\mathbf y) =
N(\mathbf Q_{t, \bar {\mathbf y}}; 
N^s(\mathbf Q_{t, \bar {\mathbf y}}; \mathbf X,\mathbf y))
\]
where $\bar {\mathbf y}$ does not change.

\begin{definition}[Partially renormalized homotopy]\label{defprh} 
Let $c_{**}$ and $\alpha$ be constants to be determined, depending on $\mathbf A$ and on the splitting
$\mathbb R^l \times \mathbb R^{n-l}$.
	The  
	approximation algorithm for the partially renormalized homotopy path 
	$(\mathbf g_t \cdot R( \mathbf 0, \mathbf{\tilde y_t}), ( \mathbf{\tilde X}_t,\mathbf 0
	))_{t \in [t_0, T]}$ is given by the recurrence:
\begin{equation}\label{recurrence1}
	(\mathbf X_{j+1},\mathbf y_{j+1}) = 
	(\mathbf 0,\mathbf y_j) + N(\mathbf Q_{t_j,\mathbf y_j}; \mathbf X_j,\mathbf 0) 
\end{equation}
	where the step size selection satisfies
	\begin{equation}\label{selection-next}
		t_{j+1} = \min \left(\strut T, \inf \left \{t>t_j: c_{**} \beta_{j+1}(t) \mu_{j+1}(t) 
	> \alpha\right \}\right)
	\end{equation}
and the notations $\beta_j$ and $\mu_j$ stand for
	\begin{eqnarray*}
		\beta_j(t) &=& 
\left\|\strut
	D\mathbf Q_{t,\mathbf y_j}(\mathbf X_j,\mathbf 0)^{-1} \mathbf Q_{t,\mathbf y_j}(\mathbf X_j,0)
\right\|_{\omega}
	\\
	\mu_j(t) &=& 
\left\|\strut
	D\mathbf Q_{t,\mathbf y_j}(\mathbf X_j,\mathbf 0)^{-1} 
\right\|_{\omega}
.
\end{eqnarray*}
\end{definition}

\begin{theorem}\label{prhomotopy} Assume that $\mathbf A=(A_1, \dots, A_n)$ is in normal form
	with respect to the splitting $\mathbb R^l \times \mathbb R^{n-l}$.
	There are constants $c_{**} \ge 1$, $\alpha_*$ and $u_*$ with the
	following properties.
	Suppose that $(\mathbf X_0,\mathbf y_0)$ satisfies 
\[
	c_{**} \, \beta_0(t_0)\, \mu_0(t_0) \le \alpha \le \alpha_*
\]
	and the sequence of $(\mathbf X_j, \mathbf y_j)$ is produced according to Definition~\ref{defprh} for
	$0=t_0 \le t \le t_J \le T$. 
	Suppose in addition that for $0 \le j \le J$, 
	$\max_k |(X_{j})_k| \le \frac{1}{4}$. Then,
\begin{enumerate}[(a)]
\item
	\label{kind1}
For all $j \le J$,
		\[
	c_{**} \, \beta_j(t_j)\, \mu_j(t_j) \le \alpha .
\]
\item Let $(\mathbf{\tilde X}_t, \mathbf{\tilde y}_t + \mathbf y_{j})$ be the zero of $\mathbf Q_{t, \mathbf y_j}$ associated
	to $(\mathbf X_j, \mathbf y_j)$, $0 \le j \le J$. The path $(\mathbf{\tilde X}_t, \mathbf{\tilde y}_t)$ is
		continuous for $0 \le t \le T$.
\item
	$\max_{k\le l} |(\mathbf{\tilde X}_t)_k| \le 1/2$,
\item For all
	\label{kind2}
	$t_j \le t \le t_{j+1} \le t_J$, 
\[
	u_j(t):= \left\|\strut D\mathbf Q_{t,\tilde {\mathbf y}_t}(\tilde {\mathbf X}_t,0)^{-1}\right\|_{\omega} 
	\left\|(\mathbf X_{j+1}-\tilde {\mathbf X}_t,\mathbf y_{j+1}-\tilde {\mathbf y}_t) \right\|_{\omega}
			\le u_*.
		\]
\item
	\[
	J \le O\left(\strut 
		\renlength{l}\left(\mathbf g_t \cdot R(\mathbf 0, \mathbf {\tilde y}_t), 
		\Omega_{\mathbf A}(\tilde {\mathbf X}_t, \mathbf 0; 0,T\right) \right)
.
\]
\end{enumerate}
\end{theorem}

\subsection{Preliminaries}
\begin{lemma}\label{fRdist}
Let $\mathbf q_i \in \mathscr F_{A_i}$ and $\mathbf y \in \mathbb C^n$.
Then for all $i=1, \dots, n$, 
\[
	d_{\mathbb P}(\mathbf q_i, \mathbf q_i \cdot R_i(\mathbf 0,\mathbf y)) \le \sqrt{5}\ \|(\mathbf 0,\mathbf y)\|_{i,\omega} \nu_{i,\omega}\  .
\]
Moreover,
\[
	d_{\mathbb P}(\mathbf q, \mathbf q \cdot R(\mathbf 0,\mathbf y)) \le \sqrt{5}\ \|(\mathbf 0,\mathbf y)\|_{\omega}\  \nu_{\omega}.
\]
\end{lemma}

The proof is the same as of Lemma \toricII{5.1.1} so it will be omitted. The two following results are similar to
Theorem~\toricI{4.3.1} but still require a proof.

\begin{lemma}\label{var-mu} There is a constant $c$ depending only of $\mathbf A$ and of the splitting
	$\mathbb R^l \times \mathbb R^{n-1}$ such that, 
	if $h=\max_{1 \le i \le l} |X_i|<\frac{1}{2}$ for all $i \le l$,
	$\beta = \| (\mathbf X,\mathbf y)-(\mathbf X', \mathbf y')\|_{\omega} \le \beta_0=\frac{1}{2\sqrt{5}\nu_{\omega}}$,
	then
\[
	\frac{ \|D\mathbf Q_{t,\mathbf y}(\mathbf X,\mathbf 0)^{-1}\|_{\omega} }
	     { 1 + c \|D\mathbf Q_{t,\mathbf y}(\mathbf X,\mathbf 0)^{-1}\|_{\omega} \beta }
		\le
	\|D\mathbf Q_{t,\mathbf y'}(\mathbf X',\mathbf 0)^{-1}\|_{\omega}
.
\]
If furthermore 
	and $c \|D\mathbf Q_{t,\mathbf y}(\mathbf X,\mathbf 0)^{-1}\|_{\omega} \beta < 1$, 
	then
\[
	\|D\mathbf Q_{t,\mathbf y'}(\mathbf X',\mathbf 0)^{-1}\|_{\omega}
\\
		\le
	\frac{ \|D\mathbf Q_{t,\mathbf y}(\mathbf X,\mathbf 0)^{-1}\|_{\omega} }
		{ 1 - c \|D\mathbf Q_{t,\mathbf y}(\mathbf X,\mathbf 0)^{-1}\|_{\omega} \beta }.
\]
\end{lemma}

\begin{lemma}\label{var-mu2} There is a constant $c$ depending only of $\mathbf A$ and of the splitting
	$\mathbb R^l \times \mathbb R^{n-1}$ such that, 
	if $h=\max_{1 \le i \le l} |X_i|, |X_i'|<\frac{1}{2}$ for all $i \le l$,
	$d_{\mathbf P}(\mathbf g_t R(\mathbf 0, \mathbf y), \mathbf g_{t'} R(\mathbf 0, \mathbf y'))  \le \delta<\frac{1}{2}$,
	and $\beta = \| (\mathbf X,\mathbf 0)-(\mathbf X', \mathbf 0)\|_{\omega} 
	\le \beta_0=\frac{1}{2\sqrt{5}\nu_{\omega}}$,
	then
\[
	\frac{ \|D\mathbf Q_{t,\mathbf y}(\mathbf X,\mathbf 0)^{-1}\|_{\omega} }
	     { 1 + c \|D\mathbf Q_{t,\mathbf y}(\mathbf X,\mathbf 0)^{-1}\|_{\omega} (d_{\mathbb P}+\beta) }
		\le
	\|D\mathbf Q_{t',\mathbf y'}(\mathbf X',\mathbf 0)^{-1}\|_{\omega}
.
\]
If furthermore 
	and $c \|D\mathbf Q_{t,\mathbf y}(\mathbf X,\mathbf 0)^{-1}\|_{\omega} (d_{\mathbb P}+\beta) < 1$, 
	then
\[
	\|D\mathbf Q_{t',\mathbf y'}(\mathbf X',\mathbf 0)^{-1}\|_{\omega}
\\
		\le
	\frac{ \|D\mathbf Q_{t,\mathbf y}(\mathbf X,\mathbf 0)^{-1}\|_{\omega} }
		{ 1 - c \|D\mathbf Q_{t,\mathbf y}(\mathbf X,\mathbf 0)^{-1}\|_{\omega} (d_{\mathbb P}+\beta) }.
\]
\end{lemma}

Before proceeding to the proof of the above Lemmas, we need to establish a few basic facts. 
Given $\mathbf 0 \ne \mathbf q, \mathbf q' \in \mathscr P_{\mathbf A}$, we define the (multi-)projective, 
(multi-)chordal  
 distances as
\begin{eqnarray*}
	d_{\mathbb P}(\mathbf q, \mathbf q') &=&
	\sqrt{\sum_i \left( \inf_{\lambda_i \in \mathbb C_{\times}} \frac {\| q_i - \lambda q_i'\|}{\| q_i\|} \right)^2} 
\\
	d_{\mathrm{chordal}}(\mathbf q, \mathbf q') &=&
	\sqrt{\sum_i \left( \inf_{\lambda_i \in S^1 \mathbb C_{\times}} \frac {\| q_i - \lambda q_i'\|}{\| q_i\|} \right)^2} 
\end{eqnarray*}
\begin{lemma} \label{chordal}
	Suppose that $\max_i d_{\mathbb P}(q_i, q_i')\le \eta$. Then,
	\[
d_{\mathbb P}(\mathbf q, \mathbf q') \le 
d_{\mathrm{chordal}}(\mathbf q, \mathbf q') 
\le
	\frac{1}{\sqrt{1-\eta^2}}
	d_{\mathbb P}(\mathbf q, \mathbf q')
	\]
\end{lemma}
\begin{proof}[Proof of Lemma~\ref{chordal}]
	Rescale the $q_i$ and $q'_i$ so that $\|q_i\|=1$ and $\|q_i-\lambda_i q'_i\|$ minimal for $|\lambda_i|=1$.
	In this case $\langle q_i, q'_i \rangle > 0$ and the angle 
	$\alpha_i = \widehat{q_i, q'_i}$ is real, we take $\alpha_i \in [0, \pi/2]$. We can write

\begin{eqnarray*}
	d_{\mathbb P}(\mathbf q, \mathbf q') &=&
	\sqrt{\sum_i \left(\sin(\alpha_i)\right)^2}
\\
	d_{\mathrm{chordal}}(\mathbf q, \mathbf q') &=&
	\sqrt{\sum_i \left( 2 \sin\left( \frac{\alpha_i}{2} \right) \right)^2}
\end{eqnarray*}
	Concavity of the sine function implies that $\sin(\alpha_i) \le 2 \sin\left( \frac{\alpha}{2} \right)$
	for $0 \le \alpha \le \pi/2$.
	From the hypothesis, $\delta = \sin(\alpha_i) \le \eta$. 
	We claim that 
$2\sin \left( \arcsin \left(\frac{\delta}{2} \right) \right) \le
\frac{\delta}{\sqrt{1-\eta^2}}$.
	Indeed, define $d_{\mathbb P}(\delta)=
2\sin \left( \arcsin \left(\frac{\delta}{2} \right) \right)$. Then,
	\[
		{d_{\mathbb P}}'(\delta) = \frac{\cos\left( \arcsin \left(\frac{d}{2} \right)\right)}
		{\sqrt{1-\delta^2}} \le \frac{1}{\sqrt{1-\eta^2}}.
	\]
The mean value Theorem guarantees that there is $0 < \xi < 1$ such that
	\[
		d_{\mathbb P}(\delta) - d_{\mathbb P}(0) = \delta\, {d_{\mathbb P}}'(\xi \delta)
		\le
	\frac{\delta}{\sqrt{1-(\xi \eta)^2}}
		\le
	\frac{\delta}{\sqrt{1-\eta^2}}.
	\]
\end{proof}

\begin{proof}[Proof of Lemma ~\ref{var-mu}]
To simplify notations, write
\begin{eqnarray*}
	M &=& D\mathbf Q_{t,\mathbf y}(\mathbf X,\mathbf 0) \\
		M' &=& D\mathbf Q_{t,\mathbf y'}(\mathbf X,\mathbf 0) \\
		N &=& D\mathbf Q_{t,\mathbf y'}(\mathbf X',\mathbf 0) \\
\end{eqnarray*}
	Also, let $\mathbf q= \mathbf g_t \cdot R(\mathbf 0,\mathbf y)$ and  $\mathbf q'= \mathbf g_t \cdot R(\mathbf 0,\mathbf y')$.
\medskip
	\noindent \\ {\bf Step 1:} We initially claim that
	\begin{equation}\label{Mstep1}
		\|M-M'\| \le d_{\mathrm{chordal}}(\mathbf q, \mathbf q') \max_i
	\frac{1}{\|\omega_i\|} \left\| D\Omega_{A_i}(\mathbf X,\mathbf 0) \right\|_{\omega}
	\end{equation}
Indeed,
\[
	M= 
\diag{\frac{1}{\|\omega_i\| \| q_{i}\|} }
\begin{pmatrix}
	\vdots\\
	q_{i} D\Omega_{A_i}(\mathbf X, \mathbf 0)  
	\\
	\vdots
\end{pmatrix}
\]
	and $M'$ is obtained by replacing $\mathbf q$ with $\mathbf q'$. 
	Since $M$ (resp. $M'$) is
	invariant by rescaling of each $q_i$ (resp $q_i'$), we rescale so that for each $i$, 
	that $\|q_i\|=\|q_i'\|=1$ and $\|q_i-q_i'\|$ is minimal. Then we see that
	$d_{\mathrm{chordal}}(\mathbf q, \mathbf q') = \|\mathbf q-\mathbf q'\|$ establishing \eqref{Mstep1}.

	The chordal distance in \eqref{Mstep1} can be bounded as follows. From 
	Lemma \ref{fRdist},
\[
	d_{\mathbb P}(\mathbf q, \mathbf q') \le 
		\sqrt{5} \beta \nu_{\omega} \le \frac{1}{2}.
\]
	Then we can apply Lemma~\ref{chordal} with $\eta=1/2$ to obtain
	\begin{equation}\label{Mstep2}
d_{\mathrm{chordal}}(\mathbf q, \mathbf q') \le 
		\frac{2 \sqrt{5}}{\sqrt{3}} \beta \nu_{\omega} \le \frac{1}{2}.
\end{equation}

	For the rightmost term of \eqref{Mstep1}, suppose that $\|\mathbf u\|_{i,\omega} \le 1$. From Lemma \ref{lemma-high-1},
	\begin{eqnarray*}
\|D\Omega_{A_i}(\mathbf X, \mathbf 0) \mathbf u\| 
		&\le&
		\|\omega_i\| + \sum_{k\ge 2} \frac{1}{(k-1)!}\| D^{k}\Omega_{A_i}(\mathbf 0, \mathbf 0) (\mathbf X, \dots, \mathbf X,u)\| 
\\
	&\le&
\|\omega_i\|+	
		\nu_{\omega} 
	\sqrt{\#A_i} 
		\sum_{k \ge 2} k h^{k-1}
	\\
		&\le&  
\|\omega_i\|+	
		\nu_{\omega}
		\sqrt{\#A_i} \frac{2h-h^2}{(1-h)^2}
		.
	\end{eqnarray*}
Multiplying by \eqref{Mstep2} and dividing by $\|\omega_i\|$,
	\begin{equation}\label{varmu-1}
\|M-M'\| \le 
		\frac{2\sqrt{5}}{\sqrt{3}} \beta \nu_{\omega}
		\left(1+ 
		\max(s_i) \frac{2h-h^2}{(1-h)^2} \nu_{\omega}\right).
	\end{equation}
\medskip
\noindent \\
	{\bf Step 2}. Assume $\finsler{\mathbf u}{\omega}\le 1$ and recall that $\| \mathbf X - \mathbf X', \mathbf 0\|_{\omega} \le \beta$. 
	We expand:
	\begin{eqnarray*}
D\Omega_{A_i}(\mathbf X', \mathbf 0) \mathbf u - D\Omega_{A_i}(\mathbf X, \mathbf 0) \mathbf u 
&=&
		\sum_{k \ge 2} \frac{1}{(k-1)!} D^k \Omega_{A_i}(\mathbf X,\mathbf 0)((\mathbf X'-\mathbf X)^{\otimes k-1}, \mathbf u) 
\\
		&&\hspace{-6em}=
		\sum_{k \ge 2} \sum_{l \ge 0}
		\frac{k \binomial{k+l}{l}}{(k+l)!} D^{k+l} \Omega_{A_i}(\mathbf 0,\mathbf 0)
		(\mathbf X^{\otimes l}, (\mathbf X'-\mathbf X)^{\otimes k-1}, \mathbf u).
	\end{eqnarray*}
Passing to norms and applying Lemma~\ref{lemma-high-1} again,
\begin{eqnarray*}
\|
D\Omega_{A_i}(\mathbf X', \mathbf 0) \mathbf u - D\Omega_{A_i}(\mathbf X, \mathbf 0) \mathbf u 
\|
	&\le&
	\sqrt{\#A_i}
\sum_{k \ge 2} \sum_{l \ge 0}
	k \binomial{k+l}{l} h^l \beta^{k-1} \nu_{\omega}^{k}
\\
&=&
\sqrt{\#A_i}
\sum_{k \ge 2}
	k \frac{\beta^{k-1} \nu_{\omega}^{k}}{(1-h)^{k+1}}
\\
	&=&
	\frac{
		\sqrt{\#A_i}\nu_{\omega}}
	{(1-h)^2}
	\left( \frac{1}{ \left(1- \frac{\beta \nu_{\omega}}{1-h}\right)^2}-1\right)
\end{eqnarray*}
It follows that
\begin{equation}\label{varmu-2} 
	\| (N-M')\mathbf u\| \le
	\frac{	\sqrt{\sum_i s_i^2} \nu_{\omega}}
	{(1-h)^2}
	\left( \frac{1}{ \left(1- \frac{\beta \nu_{\omega}}{1-h}\right)^2}-1\right)
\end{equation}

\medskip 
\noindent \\ {\bf Step 3:}
Combine equations \eqref{varmu-1} and \eqref{varmu-2},  
\[
\| M-N\| \le 
		\frac{2\sqrt{5}}{\sqrt{3}} \beta \nu_{\omega}
		\left(1+ 
		\max(s_i) \frac{2h\nu_{\omega}}{(1-h)^2} \right)
+
	\frac{	\sqrt{\sum_i s_i^2} \nu_{\omega}}
	{(1-h)^2}
	\left( \frac{1}{ \left(1- \frac{\beta \nu_{\omega}}{1-h}\right)^2}-1\right)
.
\]
Replace $h$ by $\frac{1}{2}$ and use $\beta \le \beta_0 \le \frac{1}{2 \sqrt{5} \nu_{\omega}}$
to bound the higher order terms.  
The right-hand side is bounded above by 
a constant times $\beta$. This constant is no more than
\[
	\nu_{\omega} \left(\frac{2\sqrt{5}}{\sqrt{3}}(1+4 \nu_{\omega} \max s_i) + 
	\frac{4}{3}\frac{2\sqrt{5}-1}{6-2\sqrt{5}} \sqrt{\sum_i s_i^2}\right)
.
\]
Finally apply Lemma \toricI{4.3.3}.
\end{proof}

\begin{proof}[Proof of Lemma~\ref{var-mu2}] The proof is very similar to that of Lemma~\ref{var-mu},
	but this time we set
\begin{eqnarray*}
	M &=& D\mathbf Q_{t,\mathbf y}(\mathbf X,\mathbf 0) \\
		M' &=& D\mathbf Q_{t',\mathbf y'}(\mathbf X,\mathbf 0) \\
		N &=& D\mathbf Q_{t',\mathbf y'}(\mathbf X',\mathbf 0) \\
\end{eqnarray*}
	with $\mathbf q= \mathbf g_t \cdot R(\mathbf 0,\mathbf y)$ and  $\mathbf q'= \mathbf g_{t'} \cdot R(\mathbf 0,\mathbf y')$.
	The main difference is Step 1. By hypothesis, $d_{\mathbb P}(\mathbf q, \mathbf q') \le \delta \le 1/2$.
	So we obtain instead
\[
	\| M - M'\| \le 
		\frac{2}{\sqrt{3}} \delta 
		\left(1+ 
		\max(s_i) \frac{2h-h^2}{(1-h)^2} \nu_{\omega}\right).
\]
	The proof follows verbatim the proof of Lemma~\ref{var-mu}, but in the end
	the constant is bounded above by
	\[
	\left(\frac{2}{\sqrt{3}}(1+4 \nu_{\omega} \max s_i) + 
	\nu_{\omega} 
	\frac{4}{3}\frac{2\sqrt{5}-1}{6-2\sqrt{5}} \sqrt{\sum_i s_i^2}\right)
.
\]
\end{proof}

\subsection{Proof of Theorem~\ref{prhomotopy}}
The following is very similar to Lemma \toricII{5.2.1}, but the proof differs slightly. 

\begin{lemma}\label{lem-well-defined}
	Let $\alpha_0=\frac{13-3\sqrt{17}}{4}$ 
	and define $\psi(\alpha) = 1 - 4\alpha + 2 \alpha^2$.
	Let $\beta_0$ and $c$ be the constants of Lemma~\ref{var-mu}, let
	$c_*$ satisfy \eqref{c-star} for some $h=1/4$,
	and choose $c_{**}\ge\max(c_*,c)$.
	Denote by 
	$(\mathbf X_t',\mathbf y_t'-\mathbf y_j)=N(\mathbf Q_{t,\mathbf y_j}, (\mathbf X_j, \mathbf 0))$ 
	and
	$(\mathbf{\tilde X}_t,\tilde{\mathbf y}_t-\mathbf y_j)=\lim_{s \rightarrow \infty} N^s(\mathbf Q_{t,\mathbf y_j}, (\mathbf X_j, \mathbf 0))$.
	Assume that
		$h=\max_k |X_{j}|_k \le 1/4$
	and for all $t_j \le t \le t_{j_1}$,
\[
	c_{**} \, \beta_j(t) \mu_j(t) \le \alpha \le 
		  \alpha_*:=\min\left(\alpha_0, \frac{1}{8r_0(\alpha_0) \max_i \|\omega_i\|}\right)
\]
Then,
\begin{enumerate}[(a)]
		\item
		\[
	c_{**} \, \beta_{j+1}(t)\, \mu_{j+1}(t) 
	\le \alpha
.
\]
\item 
	$h'=\max_{k \le l} |(\mathbf X_{j+1})_k| \le 3/8$.
\item 
	$\tilde h=\max_{k \le l} |(\tilde {\mathbf X}_t)_k| \le 1/2$.
	\item The path $({\tilde {\mathbf X}}_t, \tilde {\mathbf y}_t)$ is continuous.
\end{enumerate}
\end{lemma}

\begin{proof}
	Part (a) is a direct consequence of the choice of $t_{j+1}$ in \eqref{selection-next}.
	By Lemma~\ref{inf-norms}(a),
\[
	h'\le h + \beta_j(t_j) 
\]
	and (b) follows from the bound
\[
	\beta_j(t_j) \le 
	\alpha \le \frac{1}{8 \max_i \|\omega_i\|}.
\]
	Item (c) is similar but uses 
	Theorem \ref{th-higher} to bound $\tilde h \le h + \beta_j(t) r_0(\alpha)$

	implies that
\[
	\beta(\mathbf Q_{t,\mathbf y_j}, (\mathbf X_j,\mathbf 0)) \gamma(\mathbf Q_{t,\mathbf y_j}, (\mathbf X_j,\mathbf 0)) \le \alpha_0
.
\]

	The first Newton iterate of $\mathbf Q_{t,\mathbf y_j}$ with initial condition $(\mathbf X_j, \mathbf 0)$ is
$(\mathbf X_{j+1}, \mathbf y_{j+1}-\mathbf y_j)$. 
We can invoke \ocite{Bezout1}*{Prop. 3 p.478} to recover that
\[
	\beta(\mathbf Q_{t,\mathbf y_j},(\mathbf X_{j+1}, \mathbf y_{j+1}-\mathbf y_j)) \le \frac{1-\alpha}{\psi(\alpha)} 
	\alpha \beta_j(t_j).
\]
	To establish continuity in item (d), fix $\tau \in [t_j, t_{j+1}]$ and $\epsilon>0$. Note that
	by construction,
	$(\mathbf {\tilde X_\tau}, \mathbf {\tilde y_\tau})$ is uniquely defined at transition
	points $\tau=t_i$. After a fixed number of
	Newton iterates,  we obtain a continuous path 
	$(\mathbf X''_t, \mathbf y''_t)$. Corollary \ref{cor-quadratic}(a) guarantees 
	that after 
	$O(\log \log \epsilon^{-1})$ iterates,
	$\| (\mathbf X''_t, \mathbf y''_t)-(\tilde {\mathbf X}_t, \tilde {\mathbf Y}_t) \|_{\omega} < \epsilon/3$. There is
	$\delta$ such that 
	$|t-\tau|<\delta$ implies $\| (\mathbf X''_t, \mathbf y''_t)-(\mathbf X''_{\tau}, \mathbf y''_{\tau}) \|_{\omega} < \epsilon/3$. It follows that
	$\| (\tilde {\mathbf X}_t, \tilde {\mathbf y}_t)-(\tilde {\mathbf X}_{\tau}, \tilde {\mathbf y}_{\tau}) \|_{\omega} < \epsilon$.
\end{proof}

We need to rework Lemma \toricII{5.2.2}:
\begin{lemma}\label{lem-u} 
	Let $t_j$, $t_{j+1}$ and $u(t)$, $t_j \le t \le t_{j+1}$ be as in Theorem~\ref{prhomotopy}.
	Assume the conditions of Lemma~\ref{lem-well-defined}. 
Set $u_0=\frac{5-\sqrt{17}}{4}$.
For $0 <\alpha<\alpha_{0}$, 
define
\[
u_{*}=
	\frac{\alpha r_0(\alpha)}
	{1 - r_0(\alpha) \alpha}
\]
\[
u_{**}=
	\frac{\alpha r_1(\alpha)}
	{1 - r_0(\alpha) \alpha}
\]
	and $u_{***}= \frac{\alpha  \psi(u_*)}{c_{**}(1+\alpha r_0(\alpha))(\psi(u_*) +u_*)}$.  
Then,
	\begin{enumerate}[(a)]
\item for $t_j \le t \le t_{j+1}$, $u_j(t) < u_* < u_0$, 
\item $u_{j}(t_j) \le u_{**}$, and 
\item If $t_{j+1}<T$, then $u_j(t_{j+1}) \ge u_{***}$.
\end{enumerate}
\end{lemma}

\begin{proof}
	Part (a): From the choice of $t \le t_{j+1}$,
	\[
		c_{**} \beta_{j+1}(t) \mu_{j+1}(t) \le \alpha < \alpha_* \le \alpha_0.
	\]
	We are in the conditions of Corollary~\ref{cor-quadratic}, and the
Newton iterates of $\mathbf Q_{t,\mathbf y_{j+1}}$ with initial value $(\mathbf X_{j+1},\mathbf 0)$
	converge quadratically to an actual zero 
	$(\tilde {\mathbf X}, \tilde {\mathbf y}+\mathbf y_{j+1})$ 
	of $\mathbf Q_{t,\mathbf y_{j+1}}$. Moreover,
\[
	\|(\mathbf X_{j+1},\mathbf 0) - (\tilde {\mathbf X}, \tilde {\mathbf y}+\mathbf y_{j+1})\|_{\omega}
	\le
	r_0(\alpha) \beta_{j+1}(t).
\]
	From Lemma~\ref{var-mu} and the choice $c \le c_{**}$,
\[
	\| D\mathbf Q_{t, y_{j+1}+\mathbf{\tilde y}}(\mathbf{\tilde X},\mathbf 0)^{-1} \|
	\le
	\frac{\mu_{j+1}}{1-c_{**} \mu_{j+1} r_0(\alpha) \beta_{j+1}(t)}
	\le \frac{\mu_{j+1}}{1-\alpha r_0(\alpha)}
\]
This establishes that
	\[
		u_j(t) \le \frac{\alpha r_0(\alpha)}{1-\alpha r_0(\alpha)}=:u_*.
	\]
	Part (b) is similar, but Corollary~\ref{cor-quadratic}(c) with iteration starting at
	$(\mathbf X_j, \mathbf y_j)$ implies rgar $\|(\mathbf X_{j+1},\mathbf y_{j+1})-(\mathbf{\tilde X}, \mathbf{\tilde y}) \|_{\omega} \le r_1(\alpha) \beta_j(t_j)$. 
	Therefore,
	\[
		u_j(t_j) \le \frac{\alpha r_1(\alpha)}{1-\alpha r_0(\alpha)}=:u_{**}.
	\]
	Part (c): Assume that $t_{j+1} \le T$ so by construction,
\[
	c_{**} \beta_{j+1}(t_{j+1}) \mu_{j+1}(t_{j+1}) = \alpha
.
\]
	From part (a) $u_j(t_{j+1}) \le u_*$. Let $(\mathbf X',\mathbf y'+\mathbf y_{j+1})$ be the first Newton iterate of $(\mathbf X_{j+1},\mathbf 0)$
	w.r.t. $\mathbf Q_{t, y_{j+1}}$. According to \ocite{BCSS}*{Prop. 1 p. 157},
\[
	\|(\mathbf X',\mathbf y')-(\mathbf{\tilde X}, \mathbf{\tilde y}) \|_{\omega} \le \frac{u_*}{\psi(u_*)} 
	\|(\mathbf X_{j+1},\mathbf y_{j+1})-(\mathbf{\tilde X}, \mathbf{\tilde y}) \|_{\omega} 
\]
From the triangle inequality,
	\[
		\beta_j(t_{j+1}) \le \|(\mathbf X',\mathbf y')-(\mathbf{\tilde X}, \mathbf{\tilde y}) \|_{\omega}  + \|(\mathbf X_{j+1},\mathbf y_{j+1})-(\mathbf{\tilde X},\mathbf {\tilde y}) \|_{\omega} 
		\le 
		\left( 1 + \frac{u_*}{\psi(u_*)}\right) 
		\|(\mathbf X_{j+1},\mathbf y_{j+1})-(\mathbf{\tilde X}, \mathbf{\tilde y}) \|_{\omega}
.	\]
From the unconditional bound of Lemma~\ref{var-mu},
	\begin{eqnarray*}
	\mu_j(t_{j+1})
		&=&
	\| D\mathbf Q_{t_{j+1}, \mathbf y_{j}}(\mathbf X_j,\mathbf 0)^{-1} \|
		\\&\le&
	(1+c\mu_j(t_{j+1})\beta_j(t_{j+1}) r_0(\alpha)
		\| D\mathbf Q_{t_{j+1}, \mathbf y_{j+1}+\mathbf{\tilde y}}(\mathbf{\tilde X},\mathbf 0)^{-1} \|
		\\&\le&
	(1+\alpha r_0(\alpha))
		\| D\mathbf Q_{t_{j+1}, \mathbf y_{j+1}+\mathbf{\tilde y}}(\mathbf{\tilde X},\mathbf 0)^{-1} \|
	.\end{eqnarray*}
Hence,
	\[
		\frac{\alpha}{c_{**}} \le 
	(1+\alpha r_0(\alpha))
	\left( 1 + \frac{u_*}{\psi(u_*)}\right) u_{j+1}(t_{j+1})
	\]
and part(c) follows.
\end{proof}
\begin{lemma}\label{cstarstarstar} Assume $c_{**}\ge 1$ is fixed.
	For $0 \le \alpha$ small enough, $u_{**} \le u_{***} \le u_{*} \le u_0$.
	Moreover, there is a constant $c_{***}$ so that
	$u_{***}-u_{**} \ge c_{***}\alpha$
\end{lemma}

\begin{proof} All we need is the Taylor development of $u_*$ (resp. $u_{**}$, $u_{***}$) w.r.t. $\alpha$, viz.
\begin{eqnarray*}
	u_*(\alpha) &=& \alpha + 2 \alpha^2 + 6 \alpha^3 + \dots \\
	u_{**}(\alpha) &=& \alpha^2 + 4 \alpha^3 + 16 \alpha^4 + \dots\\
	u_{***}(\alpha) &=& \frac{1}{c_{**}}(\alpha - 2 \alpha^2 - 4 \alpha^3 - \dots )
\end{eqnarray*}
\end{proof}

\begin{proof}[Proof of Theorem~\ref{prhomotopy}]
	Items (a-d) are probed in Lemma~\ref{lem-well-defined}(a-d).
	It remains to establish the lower bound of item (e) on the condition length.
	We will assume the following notations:
\begin{eqnarray*}
	d(\tau) &=& 
	d_{\mathbb P}\left( 
	\mathbf g_{\tau} \cdot R(\mathbf 0,\mathbf{\tilde y}_{\tau}) \, , \, 
	\mathbf g_{t_j}\cdot R(\mathbf 0,\mathbf{\tilde y}_{t_j}  )\strut \right) + \| (\mathbf{\tilde X}_{\tau}, \mathbf{\tilde y}_{\tau}) - (\mathbf{\tilde X}_{t_j}, \mathbf{\tilde y}_{t_j})\|_{\omega}  
\\
	d_{\mathrm{max}}(t) &=& 
	\max_{t_j \le \tau \le t} d(\tau)
	\\
	\tilde \mu(t) &=& 
	\|D\mathbf Q_{t, \mathbf{\tilde y}_{t}}(\mathbf{\tilde X}_{t},\mathbf 0)^{-1}\|
	\\
	\renlength{}=\renlength{l}(t_j, t_{j+1}) &=& \int_{t_j}^{t_{j+1}}
	\left(
	\left\| \frac{\partial}{\partial t} g_{\tau} \cdot R(\mathbf 0,\mathbf{\tilde y}_{\tau} ) \right\|_{g_{\tau}} 
	+ \left\|  \frac{\partial}{\partial t} (\mathbf{\tilde X}_{\tau}, \mathbf{\tilde y}_{\tau}) \right\|_{\omega}\right)
	\tilde \mu(\tau) \dd \tau 
	.
\end{eqnarray*}
At this point it suffices to prove by induction that at each step,
\[
	\frac{1}{c+ \frac{1+c u_*}{c_{***}\alpha}} \le \renlength{l}(t_j, t_{j+1})
.
\]
Clearly, 
	\begin{equation}\label{lengthA}
	d(t_{j+1}) \le
	d_{\mathrm{max}}(t_{j+1})
	\le
	\int_{t_j}^{t_{j+1}}
		\left\| \frac{\partial}{\partial t} \left(g_{\tau} \cdot R(\mathbf 0,\mathbf {\tilde y}_{\tau})\right)\right\|_{g_{\tau}  \cdot R(\mathbf 0,\mathbf y_{\tau})} + 
		\left\|  \frac{\partial}{\partial t} (\mathbf{\tilde X}_{\tau}, \mathbf{\tilde y}_{\tau}) \right\|_{\omega} \dd \tau
.
	\end{equation}
	Lemma \ref{var-mu} with $t_j \le \tau \le t_{j+1}$ yields
	\begin{equation}\label{lengthB}
	\frac{\tilde \mu(t_j)}
	{1+c \tilde \mu(t_j)d_{\mathrm{max}}(t_{j+1})}
	\le
	\tilde \mu(\tau)
	\le
	\frac{\tilde \mu(t_j)}
	{1-c \tilde \mu(t_j) d_{\mathrm{max}}(t_{j+1})}
\end{equation}
	Combining \eqref{lengthA} and \eqref{lengthB},
\[
	\frac{
		\tilde \mu(t_j)
	d_{\mathrm{max}}(t_{j+1}) 
	}
	{1+c \tilde \mu(t_j) d_{\mathrm{max}}(t_{j+1})}
	\le
	\renlength{l}(t_j, t_{j+1})=
	\renlength{}
.
\]
Rearranging terms, 
\begin{equation}\label{prhA}
\tilde \mu(t_j)
d(t_{j+1}) 
\le
\tilde \mu(t_j)
d_{\mathrm{max}}(t_{j+1}) 
\le 
	\frac{\renlength{}}{1-c\renlength{}}
\end{equation}
	The triangle inequality can be used to bound $d(t_{j+1})$ below.
\[
	\| (\mathbf X_{j+1}, \mathbf y_{j+1})-(\mathbf{\tilde X}_{t_{j+1}}, \mathbf{\tilde y}_{t_{j+1}}) \|_{\omega}
	-
	\| (\mathbf X_{j+1}, \mathbf y_{j+1})-(\mathbf{\tilde X}_{t_j}, \mathbf{\tilde y}_{t_j}) \|_{\omega}
	\le d(t_{j+1}).
\]
	Multiplying by $\tilde \mu(t_j)$,
\[
	u_j(t_{j+1}) \frac{\tilde \mu(t_{j})}{\tilde \mu(t_{j+1})} - u_j(t_j) \le \tilde \mu(t_j) d_{\mathrm{max}(t_{j+1})}
.
\]
Invoking Lemma~\ref{var-mu2},
\[
	u_j(t_{j+1})\left(\strut 1-c \tilde \mu(t_j) d(t_{j+1})\right)  - u_j(t_j) \le 
\frac{\renlength{}}{1-c\renlength{}}
.
\]
Using \eqref{prhA} again and rearranging,
\[
	u_j(t_{j+1})
	- u_j(t_j) \le 
\frac{\renlength{}}{1-c\renlength{}}
	+c \tilde \mu(t_j) d(t_{j+1}) u_j(t_{j+1}) 
\le
(1+c u_j(t_{j+1}))
\frac{\renlength{}}{1-c\renlength{}}
.
\]
The occurrences of $u_j$ on the right (resp. left) of the equation above can be bounded by
Lemma~\ref{lem-u}(a) (resp. Lemma~\ref{lem-u}(b,c) and also  
Lemma~\ref{cstarstarstar}),
\[
	c_{***}\alpha \le u_{***}-u_{**} \le 
	u_j(t_{j+1})
	- u_j(t_j) \le 
(1+cu_*) 
\frac{\renlength{}}{1-c\renlength{}}.
\]
Rearranging terms once again,
\[
	\frac{1}{c+ \frac{1+c u_*}{c_{***}\alpha}} \le \renlength{}
\]
\end{proof}
\section{Global analysis}
\label{sec:global}

In this section we prove the Main Theorem. We start with a toric variety $\mathscr V_{\mathbf A}$ that is $n$-dimensional and smooth. This implies Hypothesis~\ref{NDIH}. The main cover $\mathscr U$ in the main Theorem is essentially provided by Theorem~\ref{th-coords}.
Let $\mathbf A$ be the original tuple of supports. For each chart
$\Omega: \mathcal D_{\Omega} \rightarrow \mathscr V_{\mathbf A}$ from the Theorem, i
let $U=\Omega(\mathcal D_{\Omega})$ be
its image. Theorem~\ref{th-coords} implies that the set of all $U$ is a cover of $\mathscr V_{\mathbf A}$.
Coordinates $(\mathbf X, \mathbf y)$ satisfy
\[
	|X_1|, \dots, |X_l| < e^{-\Phi \|\Re (\mathbf y)\|_{\infty} - \Psi}
\]
and
\[
-\frac{\Phi^{n-l}-1}{\Phi-1}\Psi - \epsilon <
	\Re(y_{l+1}), \dots,  \Re(y_{n}) < \epsilon .
.
\]
For convenience, we add the domain of the main chart, viz.
\[
	U_0 = \left \{ v \in \mathscr V_{\mathbf A}: v=\mathbf v_{\mathbf A}(\mathbf z):
\max |\Re(z_k)| < 
\frac{\Phi^{n-1}-1}{\Phi-1}\Psi \right \}
.
\]

The path-following can start using the main chart and Theorem~\toricII{4.1.4} until a change 
of chart is necessary.  Then we pick a chart as in Theorem~\ref{th-coords}, as explained in Remarks~\ref{effective} and~\ref{algorithmic-reduction}.  The support tuple $\mathbf A$ may then be reduced to the normal form of Theorem~\ref{normal-form}.  We will argue that the algorithm in Theorem~\ref{prhomotopy} produces the correct homotopy continuation, within a number of steps that is linear on the condition length. 
The restrictions for this argument to work are as follows.
\begin{enumerate}[(A)]
	\item All the complexity bounds obtained so far depend on the renormalized condition length
		$\mathcal L$ or the partially renormalized condition length $\mathcal L_l$. We need
		bounds in terms of the natural condition length $\mathscr L$.
	\item We must be able to guarantee that during homotopy continuation, the Hypothesis
		$\max_{k \le l} |(\mathbf X_j)_k| \le e^{-h} \le 1/4$ in Theorem \ref{prhomotopy} is always satisfied
		for some $U \in \mathscr U$.
\end{enumerate} 

The following bound (proof postponed) addresses the issue (A).

\begin{theorem}\label{general-bound} 
	\begin{enumerate}[(a)]
\item
Assume that $\mathbf g_t V_{\mathbf A}( \mathbf {\tilde z}_t )\equiv 0$ 
for $T_1 \le t \le T2$. Let 
\[
\bar\ell = \max_{T_1 \le t \le T_2} 
			\max_i \left(\ell_i (\mathbf {\tilde z}_t) + \ell_i( \mathbf{\tilde z}_t) \right).
\]
	Then,
\[
	\renlength{}(\mathbf g_t \cdot R( \mathbf {\tilde z}_t); T_1,T_2) \le 
		\frac{4 \sqrt{2n} \nu(\mathbf 0) \max_i(\sqrt{\#A_i})}{\lambda_0}
	e^{3 \bar \ell} \condlength(g_t, \tilde z_t; T_1,T_2) .
\]
\item Now assume $\mathbf B=\mathbf A^S$ in normal form with respect to the splitting
	$\mathbb R^l \times \mathbb R^{n-l}$, and suppose that
	$\mathbf g_t \Omega_{\mathbf B}(\mathbf{\tilde X}_t, \mathbf {\tilde y}_t )\equiv 0$. 
	Suppose in addition that equation \eqref{h-bound} holds for all 
$(\mathbf{\tilde X}_t, \mathbf {\tilde y}_t )$, $T_1 \le t \le T_2$, that is
	\[
		\| \mathbf{\tilde X}_t\|_{\infty} <
		\frac{\lambda_{\omega}}{
			8 \nu_{\omega}
		\max_i\left( \sqrt{\#A_i} e^{ 2\max_i( \ell_i(\mathbf {\tilde y}_t), 
		\ell_i(-\mathbf {\tilde y}_t) )}\right)}.
	\]
Let,
\[
\bar\ell = \max_{T_1 \le t \le T_2} 
			\max_i \left(\ell_i (\mathbf {\tilde y}_t) + \ell_i (\mathbf{\tilde y}_t) \right).
\]
	Then,
\[
	\renlength{l}(\mathbf g_t \cdot R(\mathbf 0, \mathbf {\tilde y}_t),(\mathbf{\tilde X}_t, \mathbf 0); T_1,T_2) \le 
	\frac{28\,\sqrt{n} \nu(\boldsymbol \omega) \max_i(\sqrt{\#A_i})}{\lambda_{\boldsymbol \omega}} 
			e^{3 \bar \ell} \condlength(g_t, (\mathbf {\tilde X}_t, \mathbf{\tilde y}_t); T_1,T_2) .
\]
\end{enumerate}
\end{theorem}

\begin{proof}[Proof of the Main Theorem]
	Let $S = (\Xi, \Theta)$ range over all monomial transforms in $\mathfrak S$
	leaving the supports in normal form,
	for each of the domains in $\mathscr U$ except possibly the main chart.
	The constants $\Phi$ and $\Psi$ are
\begin{eqnarray*}
	\Phi &=& 4 \max_S \left( \max_i \max_{\mathbf b \in A_i^S} \|\mathbf b\|_1 \right)
\\
	\Psi &=& \max_S \left(
	\log( \nu_{\omega}) 
	- \log (\lambda_{\omega})
	\right) 
	+ 
	\frac{1}{2} \log(\max_i \#A_i) 
	+\log (8) 
\end{eqnarray*}

	This choice guarantees that for any $(\mathbf{\tilde X}_t, \mathbf{\tilde y}_t) \in U$ for $U \in \mathscr U$ will
	satisfy
	\[
		\| \mathbf{\tilde X}_t\|_{\infty} <
		\frac{\lambda_{\omega}}{
			8 \nu_{\omega}
		\max_i\left( \sqrt{\#A_i} e^{ 2\max_i( \ell_i(\mathbf {\tilde y}_t), 
		\ell_i(-\mathbf {\tilde y}_t) )}\right)}.
.	\]

	In particular, Theorem~\ref{general-bound} will bound the renormalized and partially renormalized
	condition length linearly in the condition length. We have always
\[
	\bar \ell \le 2 
	\max \left( \max_{\mathbf a \in A_i} \left( \|\mathbf a\|_1 \frac{\Phi^{n-1}-1}{\Phi - 1} \Psi
	\right), 
	\max_S \left(\max_{\mathbf b \in A_i^S} \left( \|\mathbf b\|_1 \frac{\Phi^{n-1}-1}{\Phi - 1} \Psi
	\right) \right) \right).
\]
	If $(\tilde v_t)_{t \in [T_1, T_2]} \subset U_0$, write $\tilde v_t=\mathbf v_{\mathbf A}(\tilde{\mathbf z_t})$.
	The algorithm supporting the Main Theorem is the recurrence in  
	Definition~\toricII{4.1.2}. Its correctness follows from Theorem~\toricII{4.1.4},
	as well as the step bound of $1+\frac{1}{\delta_*}
	(\renlength{}(g_t, \tilde z_t; T_1,T_2))$
	for the constant $\delta_*$ in the Theorem.
	Theorem~\ref{general-bound} yields
\[
\renlength{}(g_t, \tilde z_t; T_1,T_2) 
\le 
		\frac{4 \sqrt{2n} \nu(\mathbf 0) \max_i(\sqrt{\#A_i})}{\lambda_0}
	e^{3 \bar \ell}
\condlength(g_t, \tilde z_t; T_1,T_2) .
\]

	Now we consider the case $(\tilde v_t)_{t \in T_1, T_2} \subset U_{S}$ for some monomial transform $S$. The notation
	now is $\tilde v_t = \Omega_{\mathbf A^S}(\tilde{\mathbf X}_t, \tilde{\mathbf y}_t)$.
	We assume inductively that
\[
	c_{**} \, \beta_0(t_j)\, \mu_0(t_j) \le \alpha \le \alpha_*.
\]
	for $t_0=T_1$. From the choice of $\Psi$, we know at each step that
	$\|\mathbf{\tilde X}_{t_j}\|_{\infty} < 1/8$.

	We claim that for $\alpha$ small enough, $\|\mathbf X_j\|_{\infty} < 1/4$. 
	Indeed, from Corollary~\ref{cor-quadratic}
\[
		\|\, 
		(\mathbf{\tilde X}_{t}, \tilde y_{t}) - 
		(\mathbf X_{j}, \mathbf y_{j})\, \|_{\omega} \le r_0(\alpha) \beta_{j}(t) .
\]
	and Lemma~\ref{inf-norms}(a) implies that
\[
	\| \mathbf X_{j} \|_{\infty} < \frac{1}{8} + r_0(\alpha) \alpha < \frac{1}{4}
\]
after making $\alpha$ small enough. This fulfills the hypotheses
of Theorem~\ref{prhomotopy}. By induction, as long as $\tilde v_t)_{t \in [T_1, T_2]} \subset U$, 
the number of steps satisfies
\[
	J \le O( 
\renlength{l}(g_t, (\mathbf{\tilde X}_t,\mathbf{ \tilde y}_t)); T_1,T_2) 
).
\]
Now Theorem~\ref{general-bound}(b) guarantees that there is a constant $C_U$ so that the
number $J$ of steps is bounded above by
\[
	J \le C_U 
\condlength(g_t, (\mathbf{\tilde X}_t, \mathbf{\tilde y}_t)); T_1,T_2) 
).
\]
\end{proof}

\subsection{Proof of Theorem~\ref{general-bound}}

We will need two more Lemmas before proving Theorem~\ref{general-bound}.

\begin{lemma}\label{main-chart-bound} In the conditions of Theorem~\ref{general-bound}(a),
\[
	\begin{split}
	\left\| \frac{\partial}{\partial t} (\mathbf g_t \cdot R(\mathbf {\tilde z}_t))\right\|_{\mathbf g_t \cdot R(\mathbf z_t)}
	+\nu(\mathbf 0) \left\| \frac{\partial}{\partial t} \mathbf{\tilde z}_t
	\right\|_{\mathbf 0}
		\le &\\ &\hspace{-8em}\le
		 e^{\ell_i(\mathbf {\tilde z}_t)+\ell_i(-\mathbf {\tilde z}_t) }
	\left( \left\| \frac{\partial}{\partial t} \mathbf g_t \right\|_{\mathbf g_t}
		+
		\frac{4 \sqrt{2n} \nu(\mathbf 0) \max_i(\sqrt{\#A_i})}{\lambda_0}
		\left\| \frac{\partial}{\partial t} \mathbf {\tilde z}_t
	\right\|_{\mathbf {\tilde z}_t} \right)
	\end{split}
\]
\end{lemma}
\begin{proof}
Without loss of generality, scale $\mathbf g_t$ so that $\langle (g_i)_t, (\dot g_i)_t \rangle \equiv 0$ for each $i$.
	Fix a value of $t$ and through the momentum action, assume that $m_{i}(\mathbf {\tilde z}_t)=\mathbf 0$ at this point.
	From the product rule, $ \frac{\partial}{\partial t} ((g_i)_t \cdot R_i(\mathbf {\tilde z}_t))=
	(\dot g_i)_t \cdot R_i(\mathbf {\tilde z}_t)+
	(g_i)_t \cdot R_i(\mathbf {\tilde z}_t) \cdot \diag{A_i \frac{\partial {\mathbf {\tilde z}_t}}{\partial t}}$.
	Next, recall that the row vector notation for the $g_i$'s. Also, write $p_i=(g_i)_t R_i({\tilde z}_t)$. Under that notation, 
	$e^{-\ell_i(-\mathbf {\tilde z}_t)} \|g_i\|\le
	\|p_i\| \le e^{\ell_i(\mathbf {\tilde z}_t)} \|g_i\|$ and
	\begin{eqnarray*}
		\left\| (\dot g_i)_t \cdot R_i(\mathbf {\tilde z}_t) \right\|_{\mathbf g_t \cdot R(\mathbf {\tilde z}_t))}
		&=&
	\left\|
	\frac{1}
	{\|p_i\|}
		(\dot g_i)_t \cdot R_i(\mathbf {\tilde z}_t) 
	\left( I - 
	\frac{1}{\|p_i\|^2} p_i^* p_i\right)
	\right\| \\
		&\le&
	\left\|
	\frac{1}
	{\|p_i\|}
		(\dot g_i)_t \cdot R_i(\mathbf {\tilde z}_t) 
	\right\| \\
		&\le&
		e^{ \ell_i(\mathbf {\tilde z}_t)+\ell_i(-\mathbf {\tilde z}_t) }
		\frac{\|(\dot g_i)_t\|}
		{\|(g_i)_t\|}
\end{eqnarray*}
and hence
	\begin{equation}\label{metric-part-1}
	\left\| \dot {\mathbf g}_t \cdot R_i(\mathbf {\tilde z}_t) \right\|_{\mathbf g_t \cdot R(\mathbf {\tilde z}_t))}
\le
		e^{\max_i (\ell_i(\mathbf {\tilde z}_t)+\ell_i(-\mathbf {\tilde z}_t)) }
		\frac{\|\dot {\mathbf g}_t\|}
		{\|{\mathbf g}_t\|}
.
\end{equation}
	By the definition of $\nu_{i}(\mathbf 0)$,
\[
	\frac{ \| p_i \cdot \diag{A_i \frac{\partial {\mathbf {\tilde z}_t}}{\partial t}} \|}{\|p_i\|} \le
	\nu_i(\mathbf 0) \left\| \frac{\partial {\mathbf {\tilde z}}_t}{\partial t} \right\|_{i,\mathbf 0} 
\]
and hence
	\begin{equation}\label{metric-part-2}
	\left\| \mathbf g_t \cdot R(\mathbf {\tilde z}_t) \cdot 
	\diag{A_i \frac{\partial{\mathbf {\tilde z}}_t}{\partial t}} 
	\right\|
	_{\mathbf g_t \cdot R(\mathbf {\tilde z}_t)}
	\le
	\nu(\mathbf 0) \left\| \frac{\partial {\mathbf {\tilde z}}_t}{\partial t} \right\|_{\mathbf 0} 
	.
	\end{equation}
	Equations~\eqref{metric-part-1} and \eqref{metric-part-2} together imply that
\[
	\left\| \frac{\partial}{\partial t} (\mathbf g_t \cdot R(\mathbf {\tilde z}_t))\right\|_{\mathbf g_t \cdot R(\mathbf z_t)}
	+\nu(\mathbf 0) \left\| \frac{\partial}{\partial t} \mathbf{\tilde z}_t
	\right\|_{\mathbf 0}
\ \le \
		 e^{\ell_i(\mathbf {\tilde z}_t)+\ell_i(-\mathbf {\tilde z}_t) }
	\left\| \frac{\partial}{\partial t} \mathbf g_t \right\|_{\mathbf g_t}
	+2 \nu(\mathbf 0) \left\| \frac{\partial}{\partial t} \mathbf {\tilde z}_t
	\right\|_{\mathbf 0} 
.\]
Lemma~\ref{metric1} finishes the proof.
\end{proof}

\begin{lemma}\label{any-chart-bound} In the conditions of Theorem~\ref{general-bound}(b),
\[
\begin{split}
\left\| \frac{\partial}{\partial t} \mathbf g_t 
\cdot 
\mathbf R(\mathbf 0, \tilde {\mathbf y}_t)
\right\|_{\mathbf g_t} 
+\left\| \frac{\partial}{\partial t} \Omega_{\mathbf A}(\tilde {\mathbf X}_t, 
\mathbf 0 
)
\right\|_{\boldsymbol \omega}
	\le& \\ &\hspace{-10em}\le
		 e^{\ell_i(\mathbf {\tilde y}_t)+\ell_i(-\mathbf {\tilde y}_t) }
	\left( \sqrt{2}\left\| \frac{\partial}{\partial t} \mathbf g_t \right\|_{\mathbf g_t}
	+
	\frac{28\,\sqrt{n} \nu(\boldsymbol \omega) \max_i(\sqrt{\#A_i})}{\lambda_{\boldsymbol \omega}} 
	\left\| \frac{\partial}{\partial t} \mathbf {\tilde y}_t
	\right\|_{\mathbf {\tilde y}_t} \right)
\end{split}
\]
\end{lemma}
\begin{proof}[Sketch of the proof]
	Argue exactly as in Lemma~\ref{main-chart-bound} to obtain
\[
\left\| \frac{\partial}{\partial t} \mathbf g_t 
\cdot 
\mathbf R(\mathbf 0, \tilde {\mathbf y}_t)
\right\|_{\mathbf g_t} 
+\left\| \frac{\partial}{\partial t} \Omega_{\mathbf A}(\tilde {\mathbf X}_t, 
\mathbf 0 
)
\right\|_{\boldsymbol \omega}
	\le
		 e^{\ell_i(\mathbf {\tilde y}_t)+\ell_i(-\mathbf {\tilde y}_t) }
	\left\| \frac{\partial}{\partial t} \mathbf g_t \right\|_{\mathbf g_t}
	+2 \nu(\boldsymbol \omega) \left\| \frac{\partial}{\partial t} \mathbf {\tilde y}_t
	\right\|_{\boldsymbol \omega}
\]
	and then apply Lemma~\ref{metric2}.
\end{proof}

\begin{proof}[Proof of Theorem~\ref{general-bound}]
Part (a) follows directly from Theorem~\ref{cost-of-renorm-legacy} combined with Lemma ~\ref{main-chart-bound}.
	Part (b) is similar. By hypothesis, equation 
\eqref{h-bound} holds for all $(\mathbf{\tilde X}_t,\mathbf {\tilde y}_t)$
with $T_1 \le t \le T_2$. Then Theorem
~\ref{cost-of-renorm} applies along the path. This bounds the condition number, the speed is bounded by
Lemma~\ref{any-chart-bound}.
\end{proof}

\section{Final remarks}
In this paper, we gave a `bundle' of local algorithms for homotopy on toric varieties. The number of homotopy steps
is linearly bounded by the condition length. It is possible to produce a global algorithm by swapping charts where
convenient. For instance, one can coarsely estimate the condition length at each step from the available data
and swap charts when it is larger than the theoretical bound, or maybe try to swap charts at every 10 steps. As
mentioned in Remark~\ref{rem-global}, the regularity of the path may affect the worst case complexity bound. 
The following problems are left open.

\begin{enumerate}
	\item
	Let $\mathbf f \in \mathscr P_{\mathbf A}$ be fixed and non-degenerate. Show that the expected condition length
	of the random linear path $\mathbf g_t = \mathbf f + t \mathbf g$, $t \in [0,1]$, $\mathbf g$ Gaussian random polynomial, is 
	polynomially bounded by the condition number of $\mathbf f$.
\item
	Let $\mathbf g_t \in \mathscr P_{\mathbf A}$ be a random path as above. Bound the expected number of chart
	swaps needed to cover its lifting $(\mathbf g_t, \mathbf {\tilde v}_t)$.
\item Also, the implementation of the algorithm suggested here is left for future work.
\end{enumerate}
All the theory in this paper makes sense if we replace the Riemannian metric
$\|\mathbf u\|_{\mathbf z} = \sqrt{ \sum_i \| \mathbf u\|_{i,\mathbf z}}$ 
by the Finsler metric $\finsler{\mathbf u}{\mathbf z} = \max_i \| \mathbf u\|_{i,\mathbf z}$.
The Finsler metric is more natural in toric varieties such as $\mathscr V_{\mathbf A}$, 
and the $\sqrt{n}$ factors in Theorems~\ref{cost-of-renorm-legacy}, \ref{cost-of-renorm} and
\ref{general-bound} go away. This requires redefining condition numbers to use a
Finsler operator norm, for instance replacing \eqref{defmu} by
		\[
\mu_{\text{Fin.}}(\mathbf f, \mathbf v)
=
\max_i \left\| 
\left(
\diag{\frac{1}{\| f_i\| \|V_{A_i}(\mathbf Z)\|}}
\begin{pmatrix}
\vdots \\
	f_i \, \diag{V_{A_i}(Z)} A_i\\
\vdots
\end{pmatrix}
\diag{\mathbf Z}^{-1}
\right)^{-1}
	\right\|_{i,\mathbf v}
		\]
which may be computationally costly.
\renewcommand{\MR}[1]{}
\newcommand{\readablebib}[3]{\bib{#1}{#2}{#3}\smallskip}

\end{document}